\documentclass[aos]{imsart}

\usepackage{amsthm,amsmath,natbib}
\RequirePackage[colorlinks,citecolor=blue,urlcolor=blue,breaklinks]{hyperref}

\arxiv{1606.00304v3}

\startlocaldefs
\usepackage{latexsym,amsmath,amssymb,amsfonts,amsthm,bbm,tensor}
\usepackage{graphicx,psfrag}
\usepackage{graphics,enumerate}

\newtheorem{eg}{Example}
\newtheorem{thm}{Theorem}
\newtheorem{prop}[thm]{Proposition}
\newtheorem{lemma}[thm]{Lemma}
\newtheorem{cor}[thm]{Corollary}

\DeclareMathOperator*{\tr}{tr}
\DeclareMathOperator*{\Var}{Var}
\DeclareMathOperator*{\Cov}{Cov}

\newcommand{\vertiii}[1]{{\left\vert\kern-0.25ex\left\vert\kern-0.25ex\left\vert #1 
    \right\vert\kern-0.25ex\right\vert\kern-0.25ex\right\vert}}

\begin{document}

\begin{frontmatter}

\title{Efficient multivariate entropy estimation via $k$-nearest neighbour distances}
\runtitle{Efficient entropy estimation}
\begin{aug}
\author{\fnms{Thomas} B. \snm{Berrett}\thanksref{t1,m1}\ead[label=e2]{t.berrett@statslab.cam.ac.uk}},
\author{\fnms{Richard} J. \snm{Samworth}\thanksref{t2,m1}\ead[label=e1]{r.samworth@statslab.cam.ac.uk}}
\and
\author{\fnms{Ming} \snm{Yuan}\thanksref{t3,m2}\ead[label=e3]{myuan@stat.wisc.edu}}
\ead[label=u1,url]{http://www.statslab.cam.ac.uk/\~{}rjs57}
\ead[label=u2,url]{http://www.statslab.cam.ac.uk/\~{}tbb26}
\ead[label=u3,url]{http://pages.stat.wisc.edu/\~{}myuan/}
\thankstext{t1}{Research supported by a Ph.D. scholarship from the SIMS fund.}
\thankstext{t2}{Research supported by an EPSRC Early Career Fellowship and a grant from the Leverhulme Trust.}
\thankstext{t3}{Research supported by NSF FRG Grant DMS-1265202 and NIH Grant 1-U54AI117924-01.}
\runauthor{T. B. Berrett, R. J. Samworth and M. Yuan}
\affiliation{University of Cambridge\thanksmark{m1}}
\affiliation{University of Wisconsin--Madison\thanksmark{m2}}

\address{Statistical Laboratory \\ Wilberforce Road \\ Cambridge \\ CB3 0WB \\ United Kingdom\\ 
          \printead{e1}\\\printead{e2}\\ \printead{u1}\\\printead{u2}}
\address{Department of Statistics \\University of Wisconsin--Madison \\ Medical Sciences Center \\ 1300 University Avenue \\ Madison, WI 53706 \\ United States of America \\ \printead{e3} \\ \printead{u3}}

\end{aug}

\begin{abstract}
Many statistical procedures, including goodness-of-fit tests and methods for independent component analysis, rely critically on the estimation of the entropy of a distribution.  In this paper, we seek entropy estimators that are efficient and achieve the local asymptotic minimax lower bound with respect to squared error loss.  To this end, we study weighted averages of the estimators originally proposed by \citet{Kozachenko:87}, based on the $k$-nearest neighbour distances of a sample of $n$ independent and identically distributed random vectors in $\mathbb{R}^d$.  A careful choice of weights enables us to obtain an efficient estimator in arbitrary dimensions, given sufficient smoothness, while the original unweighted estimator is typically only efficient when $d \leq 3$.  In addition to the new estimator proposed and theoretical understanding provided, our results facilitate the construction of asymptotically valid confidence intervals for the entropy of asymptotically minimal width.
\end{abstract}



\end{frontmatter}

%\usepackage{amsmath, amssymb,amsfonts,bbm,verbatim,multirow,color}
%\usepackage{epic,eepic,psfrag,epsfig}
%\usepackage[sf,bf,SF,footnotesize]{subfigure}
%\usepackage{graphicx}
%\usepackage{amsthm,breakcites}
%\usepackage{amscd}
%\usepackage{epsfig}
%\usepackage{fullpage}
%\usepackage{natbib}
%\usepackage{setspace}
%\usepackage{enumerate}
%\usepackage{xr,array}
%\externaldocument{WKLAoSRevSupp}
%\usepackage[small]{caption}
%\RequirePackage[colorlinks,citecolor=blue,urlcolor=blue,breaklinks]{hyperref}
%\def\mathbi#1{\textbf{\em #1}}

%\newtheorem{thm}{Theorem}
%\newtheorem{defn}[thm]{Definition}
%\newtheorem{lemma}[thm]{Lemma}
%\newtheorem{sublemma}[thm]{Sublemma}
%\newtheorem{prop}[thm]{Proposition}
%\newtheorem{cor}[thm]{Corollary}
%\newtheorem{eg}{Example}
%\renewcommand{\thefootnote}{\fnsymbol{footnote}}
%\renewcommand{\baselinestretch}{1.25}

%\newcommand{\red}[1]{\textcolor{red}{#1}}
%\newcommand{\blue}[1]{\textcolor{blue}{#1}}
%\newcommand{\green}[1]{\textcolor{green}{#1}}
%\numberwithin{equation}{section}
%\DeclareMathOperator*{\argmax}{argmax}
%\DeclareMathOperator*{\tr}{tr}
%\DeclareMathOperator*{\Var}{Var}
%\DeclareMathOperator*{\Cov}{Cov}
%\newenvironment{prooftitle}[1]{{\noindent \textsc{Proof #1}}\\}

%\def\hat{\widehat}

%\newpage

\section{Introduction}
%\subsection{Definition of entropy and applications}

The concept of entropy plays a central role in information theory, and has found a wide array of uses in other disciplines, including statistics, probability and combinatorics.  The \emph{(differential) entropy} of a random vector $X$ with density function $f$ is defined as
\[
	H = H(X)= H(f):= -\mathbb{E}\{\log f(X)\} = -\int_{\mathcal{X}} f(x) \log f(x) \, dx 
\]
where $\mathcal{X} := \{ x : f(x) >0 \}$. It represents the average information content of an observation, and is usually thought of as a measure of unpredictability.  

In statistical contexts, it is often the estimation of entropy that is of primary interest, for instance in goodness-of-fit tests of normality \citep{Vasicek:76} or uniformity \citep{Cressie:76}, tests of independence \citep{Goria:05}, independent component analysis \citep{Miller:03} and feature selection in classification \citep{Kwak:02}.  See, for example, \citet{Beirlant:97} and \citet{Paninski2003} for other applications and an overview of nonparametric techniques, which include methods based on sample spacings in the univariate case \citep[e.g.][]{HG2009}, histograms \citep{Hall:93} and kernel density estimates \citep{Paninski:08,Hero:13}, among others.  The estimator of \citet{Kozachenko:87} is particularly attractive as a starting point, both because it generalises easily to multivariate cases, and because, since it only relies on the evaluation of $k$th-nearest neighbour distances, it is straightforward to compute.

To introduce this estimator, for $n \geq 2$, let $X_1,\ldots,X_n$ be independent random vectors with density $f$ on $\mathbb{R}^d$.  Write $\|\cdot\|$ for the Euclidean norm on $\mathbb{R}^d$, and for $i=1,\dots,n$, let $X_{(1),i},\ldots,X_{(n-1),i}$ denote a permutation of $\{X_1,\ldots,X_n\} \setminus \{X_i\}$ such that $\|X_{(1),i} - X_i\| \leq \ldots \leq \|X_{(n-1),i} - X_i\|$.  For conciseness, we let
\[
\rho_{(k),i} := \|X_{(k),i} - X_i\|
\]
denote the distance between $X_i$ and the $k$th nearest neighbour of $X_i$.  The Kozachenko--Leonenko estimator of the entropy $H$ is given by
\begin{equation}
\label{Eq:KLEstimator}
\hat{H}_n = \hat{H}_n(X_1,\ldots,X_n) := \frac{1}{n}\sum_{i=1}^n \log \biggl(\frac{\rho_{(k),i}^d V_d (n-1)}{e^{\Psi(k)}}\biggr),
\end{equation}
where $V_d := \pi^{d/2}/\Gamma(1 + d/2)$ denotes the volume of the unit $d$-dimensional Euclidean ball and where $\Psi$ denotes the digamma function. In fact, this is a generalisation of the estimator originally proposed by \citet{Kozachenko:87}, which was defined for $k=1$. For integers $k$ we have $\Psi(k) = -\gamma + \sum_{j=1}^{k-1} 1/j$ where $\gamma := 0.577216\ldots$ is the Euler--Mascheroni constant, so that $e^{\Psi(k)}/k \rightarrow 1$ as $k \rightarrow \infty$. This estimator can be regarded as an attempt to mimic the `oracle' estimator $H_n^* := -n^{-1}\sum_{i=1}^n \log f(X_i)$, based on a $k$-nearest neighbour density estimate that relies on the approximation
\[
	\frac{k}{n-1} \approx V_d \rho_{(k),1}^d f(X_1).
\]
It turns out that, when $d \leq 3$ and other regularity conditions hold, the estimator $\hat{H}_n$ in~\eqref{Eq:KLEstimator} has the same asymptotic behaviour as $H_n^*$, in that
\[
n^{1/2}(\hat{H}_n  - H) \stackrel{d}{\rightarrow} N\bigl(0,\mathrm{Var} \log f(X_1)\bigr).
\]
We will see that in such settings, this estimator is asymptotically efficient, in the sense of, e.g., \citet[][p.~367]{vanderVaart1998}.  However, when $d \geq 4$, a non-trivial bias typically precludes its efficiency.  Our main object of interest, therefore, will be a generalisation of the estimator~\eqref{Eq:KLEstimator}, formed as a weighted average of Kozachenko--Leonenko estimators for different values of $k$, where the weights are chosen to try to cancel the dominant bias terms.  More precisely, for a weight vector $w = (w_1,\ldots,w_k)^T \in \mathbb{R}^k$ with $\sum_{j=1}^k w_j = 1$, we consider the estimator
\[
%\label{Eq:WeightedKL}
\hat{H}_n^w := \frac{1}{n} \sum_{i=1}^n \sum_{j=1}^k w_j \log \xi_{(j),i},
\]
where $\xi_{(j),i}:=e^{-\Psi(j)} V_d(n-1) \rho_{(j),i}^d$.  Weighted estimators of this general type have been considered recently \citep[e.g.][]{Hero:13,MSGH2016}, though our construction of the weights and our analysis is new.  In particular, we show that under stronger smoothness assumptions, and with a suitable choice of weights, the weighted Kozachenko--Leonenko estimator is efficient in arbitrary dimensions.

%These results motivate our second main contribution, namely the proposal of a new entropy estimator, formed as a weighted average of Kozachenko--Leonenko estimators for different values of $k$.  We show that it is possible to choose the weights in such a way as to cancel the dominant bias terms, thereby yielding an efficient estimator in arbitrary dimensions, given sufficient smoothness. \textbf{New:} The idea of performing bias reduction by taking a weighted average of estimators is not new; see for example~\citet{Hero:13}. In this paper a general framework for ensemble estimation is proposed, and the particular case of a weighted average of entropy estimators based on kernel density estimators is studied. The authors use sample splitting and prove that the mean squared error of their estimator is $O(1/n)$. The conditions they use include $f$ being bounded below by a positive constant on its support and being $d$ times differentiable on the interior of its support. In the current paper we allow $f$ to have unbounded support, and require that it is $d/2$ times differentiable (among other regularity conditions). \textbf{(Cite Samworth (2012) optimal weighted classification?)}

There have been several previous studies of the (unweighted) Kozachenko--Leonenko estimator, but results on the rate of convergence have until now confined either to the case $k=1$ or (very recently) the case where $k$ is fixed as $n$ diverges.  The original \citet{Kozachenko:87} paper proved consistency of the estimator under mild conditions in the case $k=1$.  \citet{Tsybakov:96} proved that the mean squared error of a truncated version of the estimator is $O(n^{-1})$ when $k=1$ and $d=1$ under a condition that is almost equivalent to an exponential tail; \citet{BiauDevroye2015} showed that the bias vanishes asymptotically while the variance is $O(n^{-1})$ when $k=1$ and $f$ is compactly supported and bounded away from zero on its support.  Very recently, in independent work and under regularity conditions, \citet{DelattreFournier:16} derived the asymptotic normality of the estimator when $k=1$, confirming the suboptimal asymptotic variance in this case.  %\citet{SinghPoczos2016} provide bias and variance bounds under a tail condition that implies that the Lebesgue measure of the support of the density is finite, but since the variance bound is $O(k^2/n)$, suboptimal rates are obtained when $k$ diverges.
%Other very recent work includes \citet{SinghPoczos2016} and \citet{GOV2016}
Previous works on the general $k$ case include \citet{Singh:03}, where heuristic arguments were presented to suggest the estimator is consistent for general $d$ and general fixed $k$ and has variance $O(n^{-1})$ for $d=1$ and general fixed $k$.  \citet{GOV2016} obtain a mean squared error bound of $O(n^{-1})$ up to polylogarithmic factors for fixed $k$ and $d \leq 2$, though the only densities which the authors can show satisfy their tail condition have bounded support.  \citet{SinghPoczos2016} obtain a similar bound (without the polylogarithmic factors, but explicitly assuming bounded support) for fixed $k$ and $d \leq 4$.  \citet{Mnatsakanov:08} allow $k$ to diverge with $n$, and show that the estimator is consistent for general $d$.  

%In this paper, we substantially weaken regularity conditions compared with these previous works, and show that with an appropriate choice of $k$, efficient entropy estimators can be obtained in arbitrary dimensions, even in cases where the support of the true density is the whole of $\mathbb{R}^d$.  
Plug-in kernel methods are also popular for entropy estimation.  \citet{Paninski:08}, for example, show that a smaller bandwidth than would be required for a consistent density estimator can still yield a consistent entropy estimator.  A $k$-nearest neighbour density estimate can be regarded as a kernel estimator with a bandwidth that depends both on the data and on the point at which the estimate is required.  \citet{Hero:13} obtain the parametric rate of convergence for a plug-in kernel method, assuming bounded support and at least $d$ derivatives in the interior of the support.

%This is similar to the $k$-nearest neighbour methods we study, in that fixed $k$ estimators may be consistent for entropy estimation but they are not consistent for density estimation. A drawback of kernel methods is that they require $f$ to be $d$ times differentiable to achieve the parametric rate of convergence, whereas we will show that our method can achieve this rate with only $d/2$ derivatives (and other regularity conditions). Most results on kernel estimators additionally assume that $f$ is bouded below by a positive constant on its support.

Importantly, the class of densities considered in our results allows the support of the density to be unbounded; for instance, it may be the whole of $\mathbb{R}^d$.  Such settings present significant new challenges and lead to different behaviour compared with more commonly-studied situations where the underlying density is compactly supported and bounded away from zero on its support.  To gain intuition, consider the following second-order Taylor expansion of $H(f)$ around a density estimator $\hat{f}$:
\[
%\label{Eq:CompactTaylor}
H(f) \approx - \int_{\mathbb{R}^d} f(x) \log \hat{f}(x) \, dx - \frac{1}{2}\biggl(\int_{\mathbb{R}^d} \frac{f^2(x)}{\hat{f}(x)} \, dx - 1\biggr).
\]
When $f$ is bounded away from zero on its support, one can estimate the (smaller order) second term on the right-hand side, thereby obtaining efficient estimators of entropy in higher dimensions \citep[][]{Laurent1996}; however, when $f$ is not bounded away from zero on its support such procedures are no longer effective.  To the best of our knowledge, therefore, this is the first time that a nonparametric entropy estimator has been shown to be efficient in multivariate settings for densities having unbounded support.  (We remark that when $d=1$, the histogram estimator of \citet{Hall:93} is known to be efficient under fairly strong tail conditions.)  

The outline of the rest of the paper is as follows.  In Section~\ref{Sec:Main}, we give our main results on the mean squared error and asymptotic normality of weighted Kozachenko--Leonenko estimators, and discuss confidence interval construction.  These main results arise from asymptotic expansions for the bias and variance, which are stated in Section~\ref{Sec:BiasVariance}.  Here, we also give examples to illustrate densities satisfying our conditions, discuss how they may be weakened, and address the fixed $k$ case.  Corresponding lower bounds are presented in Section~\ref{Sec:LowerBound}.  Proofs of main results are presented in Section~\ref{sec:proof} with auxiliary material and detailed bounds for various error terms deferred to the Appendix, which appears as the supplementary material \citet{BSY2017}.

%In this paper we derive the leading term in the asymptotic expansion of the bias when $d \geq 3$ under regularity conditions, and show that for $d=1$ and $d=2$ the bias is smaller than $n^{-1/2}$ in many cases. Our results hold under weak conditions on $k$ which allow it to diverge with $n$; previously, results which give rates of convergence have only been shown for fixed $k$. For $d \leq 3$ and $k / \log n \rightarrow \infty$ we find the leading term in the asymptotic expansion of the variance, which matches the local asymptotic minimax lower bound, and show that the estimator is asymptotically normally distributed. To the best of our knowledge, these are the first results to prove the asymptotic efficiency of a nonparametric entropy estimator for $d \geq 2$. As well as providing important theoretical understanding, our results have several methodological implications. We use our results to construct asymptotic confidence intervals for $H(f)$ and also show that prewhitening the data may help to reduce the bias, through both theoretical and empirical evidence.

We conclude the introduction with some notation used throughout the paper.  For $x \in \mathbb{R}^d$ and $r > 0$, let $B_x(r)$ be the closed Euclidean ball of radius $r$ about $x$, and let $B_x^\circ(r) := B_x(r) \setminus \{x\}$ denote the corresponding punctured ball.  We write $\|A\|_{\mathrm{op}}$ and $|A|$ for the operator norm and determinant, respectively, of $A \in \mathbb{R}^{d \times d}$, and let $\|A\|$ denote the vectorised Euclidean norm of a vector, matrix or array.  For a smooth function $f:\mathbb{R}^d \rightarrow [0,\infty)$, we write $\dot{f}(x),\ddot{f}(x)$ and $f^{(m)}(x)$ respectively for the gradient vector of $f$ at $x$, Hessian matrix of $f$ at $x$ and the array with $(j_1,\ldots,j_m)$th entry $\frac{\partial^m f(x)}{\partial x_{j_1}\ldots \partial x_{j_m}}$.  We also write $\Delta f(x) := \sum_{j=1}^d \frac{\partial^2 f}{\partial x_j^2}(x)$ for its Laplacian, and $\|f\|_\infty := \sup_{x \in \mathbb{R}^d} f(x)$ for its uniform norm.  

\section{Main results}
\label{Sec:Main}

We begin by introducing the class of densities over which our results will hold.  Let $\mathcal{F}_d$ denote the class of all density functions with respect to Lebesgue measure on $\mathbb{R}^d$.  For $f \in \mathcal{F}_d$ and $\alpha > 0$, let
\[
\mu_\alpha(f) := \int_{\mathbb{R}^d} \|x\|^\alpha f(x) \, dx.
\]
Now let $\mathcal{A}$ denote the class of decreasing functions $a:(0,\infty) \rightarrow [1,\infty)$ satisfying $a(\delta) = o(\delta^{-\epsilon})$ as $\delta \searrow 0$, for every $\epsilon > 0$.  If $a \in \mathcal{A}$, $\beta > 0$ and $f \in \mathcal{F}_d$ is $m := \lceil \beta \rceil -1$-times differentiable and $x \in \mathcal{X}$, we define $r_a(x):=\{8d^{1/2}a(f(x))\}^{-1/(\beta \wedge 1)}$ and
\[
M_{f,a,\beta}(x) := \max \biggl\{ \max_{t=1,\ldots, m} \frac{\|f^{(t)}(x)\|}{f(x)} \, , \, \sup_{y \in B_x^\circ(r_a(x))} \frac{\|f^{(m)}(y)-f^{(m)}(x)\|}{f(x) \|y-x\|^{\beta- m}} \biggr\}.
\]
The quantity $M_{f,a,\beta}(x)$ measures the smoothness of derivatives of $f$ in neighbourhoods of $x$, relative to $f(x)$ itself.  Note that these neighbourhoods of $x$ are allowed to become smaller when $f(x)$ is small.  Finally, for $\Theta := (0,\infty)^4 \times \mathcal{A}$, and $\theta = (\alpha,\beta,\nu,\gamma,a) \in \Theta$, let
\[
\mathcal{F}_{d,\theta} := \biggl\{f \in \mathcal{F}_d: \mu_\alpha(f) \leq \nu, \|f\|_\infty \leq \gamma, \sup_{x:f(x) \geq \delta} M_{f,a,\beta}(x) \leq a(\delta) \ \forall \delta > 0\biggr\}.
\]         
We note here that Lemma~\ref{Lemma:15over7} in the online supplement can be used to derive a nestedness property of the classes with respect to the smoothness parameter, namely that if $\theta = (\alpha,\beta,\gamma,\nu,a) \in \Theta$, $\beta' \in (0,\beta)$ and $a'(\delta) = 15d^{\lceil \beta \rceil/2}a(\delta)$, then $\mathcal{F}_{d,\theta} \subseteq \mathcal{F}_{d,\theta'}$, where $\theta' = (\alpha,\beta',\gamma,\nu,a') \in \Theta$.  In Section~\ref{Sec:Assumptions} below, we discuss the requirements of the class $\mathcal{F}_{d,\theta}$ in greater detail, and give several examples, including Gaussian and multivariate-$t$ densities, which belong to $\mathcal{F}_{d,\theta}$ for suitable $\theta$.

We now introduce the class of weights $w = (w_1,\ldots,w_k)^T$ that we consider.  For $k \in \mathbb{N}$, let
\begin{align}
\label{Eq:Wk}
	\mathcal{W}^{(k)} := \biggl\{ w \in \mathbb{R}^k : &\sum_{j=1}^k w_j \frac{\Gamma(j+2 \ell/d)}{\Gamma(j)} =0 \quad \text{for} \, \, \ell=1, \ldots, \lfloor d/4 \rfloor \nonumber \\
	&\sum_{j=1}^k w_j= 1\, \text{and} \, \, w_j=0 \, \, \text{if}\, j \notin \{ \lfloor k/d \rfloor, \lfloor 2k/d \rfloor, \ldots, k\} \biggr\}.
\end{align}
Our main result below shows that for appropriately chosen weight vectors in $\mathcal{W}^{(k)}$, the normalised risk of the weighted Kozachenko--Leonenko estimator $\hat{H}_n^w$ converges in a uniform sense to that of the oracle estimator $H_n^* := -n^{-1}\sum_{i=1}^n \log f(X_i)$.  Theorem~\ref{Thm:LowerBound} in Section~\ref{Sec:LowerBound} shows that this limiting risk is optimal.   
\begin{thm}
\label{unifweightedclt}
Fix $d \in \mathbb{N}$ and $\theta = (\alpha,\beta,\nu,\gamma,a) \in \Theta$ with $\alpha>d$ and with $\beta>d/2$.  Let $k_0^*=k_{0,n}^*$ and $k_1^*=k_{1,n}^*$ denote any two deterministic sequences of positive integers with $k_0^* \leq k_1^*$, with $k_0^* / \log^5 n \rightarrow \infty$ and with $k_1^*=O(n^{\tau_1})$ and $k_1^* = o(n^{\tau_2})$, where
\[
\tau_1 < \min \biggl( \frac{2 \alpha}{5\alpha+3d} \, , \, \frac{\alpha-d}{2\alpha} \, , \, \frac{4\beta^*}{4\beta^*+3d}\biggr), \tau_2 := \min\biggl(1-\frac{d/4}{1+\lfloor d/4 \rfloor},1-\frac{d}{2\beta}\biggr)
\]
and $\beta^* := \beta \wedge 1$.  There exists $k_d \in \mathbb{N}$, depending only on $d$, such that for each $k \geq k_d$, we can find $w = w^{(k)} \in \mathcal{W}^{(k)}$ with $\sup_{k \geq k_d} \|w^{(k)}\| < \infty$.  For such $w$,
\begin{equation}
\label{Eq:Eff}
\sup_{k \in \{k_0^*,\ldots,k_1^*\}} \sup_{f \in \mathcal{F}_{d, \theta}} n\mathbb{E}_f\bigl\{(\hat{H}_n^w - H_n^*)^2\bigr\} \rightarrow 0
\end{equation}
as $n \rightarrow \infty$.  In particular,
\[
\sup_{k \in \{k_0^*,\ldots,k_1^*\}} \sup_{f \in \mathcal{F}_{d, \theta}} \bigl| n\mathbb{E}_f\{(\hat{H}_n^w - H(f))^2\} -V(f) \bigr| \rightarrow 0,
\]
where $V(f) := \mathrm{Var}_f \log f(X_1) = \int_{\mathcal{X}} f \log^2 f - H(f)^2$.
\end{thm}  
We remark that the level of smoothness we require for efficiency in Theorem~\ref{unifweightedclt}, namely $\beta > d/2$ is more than is needed for the two-stage estimator of \citet{Laurent1996} in the case where $f$ is compactly supported and bounded away from zero on its support, where $\beta > d/4$ suffices.  As alluded to in the introduction, the fact that the function $x \mapsto -x\log x$ is non-differentiable at $x=0$ means that the entropy functional is no longer smooth when $f$ has full support, so the arguments of \citet{Laurent1996} can no longer be applied and very different behaviour may occur \citep{LNS1999,CaiLow2011}.

It is also useful, e.g.~for the purposes of constructing confidence intervals for the entropy, to understand the asymptotic normality of the estimator.  To this end, let $\mathcal{H}$ denote the class of functions $h:\mathbb{R} \rightarrow \mathbb{R}$ with $\|h\|_\infty \leq 1$ and $|h(x) - h(y)| \leq |x-y|$ for all $x,y \in \mathbb{R}$.  For probability measures $P, Q$ on $\mathbb{R}$, we write
\[
d_{\mathrm{BL}}(P,Q) := \sup_{h \in \mathcal{H}}\biggl| \int_{-\infty}^\infty h \, d(P-Q)\biggr|
\]
for the bounded Lipschitz distance between $P$ and $Q$.  Recall that $d_{\mathrm{BL}}$ metrises weak convergence.  The asymptotic variance $V(f)$ can be estimated analogously to $H(f)$ by $\hat{V}_n^w := \max(\tilde{V}_n^w,0)$, where
\[
\tilde{V}_n^w := \frac{1}{n} \sum_{i=1}^n \sum_{j=1}^k w_j \log^2 \xi_{(j),i} - (\hat{H}_n^w)^2.
\]
Fixing $q \in (0,1)$, this suggests that a natural asymptotic $(1-q)$-level confidence interval for $H(f)$ is given by 
\[
I_{n,q} := \bigl[\hat{H}_n^w - n^{-1/2}z_{q/2}(\hat{V}_n^w)^{1/2},\hat{H}_n^w + n^{-1/2}z_{q/2}(\hat{V}_n^w)^{1/2}\bigr],
\]
where $z_q$ is the $(1-q)$th quantile of the standard normal distribution; see also \citet{DelattreFournier:16}.  Write $\mathcal{L}(Z)$ for the distribution of a random variable $Z$.
\begin{thm}
\label{Thm:UnifWeightedCLT}
Under the conditions of Theorem~\ref{unifweightedclt}, we have
\[
\sup_{k \in \{k_0^*,\ldots,k_1^*\}} \sup_{f \in \mathcal{F}_{d, \theta}} d_{\mathrm{BL}}\Bigl(\mathcal{L}\bigl(n^{1/2}(\hat{H}_n^w - H(f))\bigr),N\bigl(0,V(f)\bigr)\Bigr) \rightarrow 0
\]
as $n \rightarrow \infty$.  Consequently,
\[
\sup_{q \in (0,1)} \sup_{k \in \{k_0^*,\ldots,k_1^*\}} \sup_{f \in \mathcal{F}_{d, \theta}} \Bigl| \mathbb{P}_f\bigl(I_{n,q} \ni H(f)\bigr) - (1-q)\Bigr| \rightarrow 0.
\]
\end{thm}
We remark that the choice $k = k_n = \lceil \log^6 n \rceil$ with $w = w^{(k)} \in \mathcal{W}^{(k)}$ satisfying $\sup_{k \geq k_d} \|w^{(k)}\| < \infty$ for the weighted Kozachenko--Leonenko estimator satisfies the conditions for efficiency in Theorem~\ref{unifweightedclt} whenever $f \in \mathcal{F}_{d,\theta}$ with $\theta = (\alpha,\beta,\gamma,\nu,a) \in \Theta$ satisfying $\alpha > d$ and $\beta > d/2$; knowledge of the precise values of $\alpha$ and $\beta$ is not required.  Moreover, the uniformity of the asymptotics in $k$ means that if $\hat{k}_n = \hat{k}_n(X_1,\ldots,X_n)$ is a data-driven choice of $k$, the conclusions Theorem~\ref{Thm:UnifWeightedCLT} remain valid provided that 
%\[
%\mathbb{P}\biggl(\bigcup_{n=1}^\infty \bigcap_{m=n}^\infty \bigl\{\hat{k}_m \in \{k_0^*,\ldots,k_1^*\}\bigr\}\biggr) = 1
%\]
%in the first case, and provided that 
$\mathbb{P}(\hat{k}_n < k_0^*) + \mathbb{P}(\hat{k}_n > k_1^*) \rightarrow 0$. 

\section{Bias and variance expansions for Kozachenko--Leonenko estimators}
\label{Sec:BiasVariance}
\subsection{Bias}
The proof of~\eqref{Eq:Eff} is derived from separate expansions for the bias and variance of the weighted Kozachenko--Leonenko estimator, and we treat the bias in this subsection.  To gain intuition, we initially focus for simplicity of exposition on the unweighted estimator
\[
\hat{H}_n = \frac{1}{n}\sum_{i=1}^n \log \xi_i,
\]
where we have written $\xi_i$ as shorthand for $\xi_{(k),i}$.  For $x \in \mathbb{R}^d$ and $u \in [0,\infty)$, we introduce the sequence of distribution functions
\[
%\label{Eq:Fnxu}
F_{n,x}(u) := \mathbb{P}( \xi_i \leq u | X_i = x) = \sum_{j=k}^{n-1} \binom{n-1}{j}p_{n,x,u}^j(1-p_{n,x,u})^{n-1-j},
\]
where
\[
p_{n,x,u} := \int_{B_x(r_{n,u})} f(y) \, dy\qquad {\rm and} \qquad r_{n,u} := \biggl\{\frac{e^{\Psi(k)} u}{V_d (n-1)}\biggr\}^{1/d}.
\]
Further, for $u \in [0,\infty)$, define the limiting (Gamma) distribution function 
\[
%\label{Eq:Fxu}
F_x(u) := \exp\{-uf(x)e^{\Psi(k)}\}\sum_{j=k}^\infty \frac{1}{j!} \bigl\{uf(x)e^{\Psi(k)}\bigr\}^j = e^{-\lambda_{x,u}}\sum_{j=k}^\infty \frac{\lambda_{x,u}^j}{j!},
\]
where $\lambda_{x,u} := uf(x)e^{\Psi(k)}$. That this is the limit distribution for each fixed $k$ follows from a Poisson approximation to the Binomial distribution and the Lebesgue differentiation theorem. We therefore expect that 
\begin{align*}
	\mathbb{E}(\hat{H}_n) &= \int_\mathcal{X} f(x) \int_0^\infty \log u \,dF_{n,x}(u) \,dx \approx \int_\mathcal{X} f(x) \int_0^\infty \log u \,dF_x(u) \,dx \\
	& = \int_{\mathcal{X}} f(x) \int_0^\infty \log\Bigl(\frac{te^{-\Psi(k)}}{f(x)}\Bigr) e^{-t}\frac{t^{k-1}}{(k-1)!} \, dt \, dx = H.
\end{align*}
Although we do not explicitly use this approximation in our asymptotic analysis of the bias, it motivates much of our development.  It also explains the reason for using $e^{\Psi(k)}$ in the definition of $\xi_{(k),i}$, rather than simply $k$.   Lemma~\ref{Lemma:Bias} below gives an expression for the asymptotic bias of the unweighted Kozachenko--Leonenko estimator. 
\begin{lemma}
\label{Lemma:Bias}
	Fix $d \in \mathbb{N}$ and $\theta = (\alpha,\beta,\nu,\gamma,a) \in \Theta$.  Let $k^* = k_n^*$ denote any deterministic sequence of positive integers with $k^* = O(n^{1-\epsilon})$ as $n \rightarrow \infty$ for some $\epsilon >0$. Then there exist $\lambda_1, \ldots, \lambda_{\lceil \beta/2 \rceil -1} \in \mathbb{R}$, depending only on $f$ and $d$, such that $\sup_{f \in \mathcal{F}_{d,\theta}} \max_{l=1,\ldots,\lceil \beta/2 \rceil -1} |\lambda_l| < \infty$ and for each $\epsilon > 0$,
\[
	\sup_{f \in \mathcal{F}_{d,\theta}}\biggl| \mathbb{E}_f(\hat{H}_n) - H - \sum_{l=1}^{\lceil \beta/2 \rceil -1} \frac{\Gamma(k+2l/d) \Gamma(n)}{ \Gamma(k) \Gamma(n+2l/d)} \lambda_l\biggr| = O \biggl( \max \biggl\{ \frac{k^{\frac{\alpha}{\alpha+d} - \epsilon}}{n^{\frac{\alpha}{\alpha+d} - \epsilon}}\, , \, \frac{k^{\frac{\beta}{d}}}{n^{\frac{\beta}{d}}} \biggr\} \biggr)
\]
as $n \rightarrow \infty$, uniformly for $k \in \{1,\ldots,k^*\}$, where $\lambda_l=0$ if $2l \geq d \alpha/(\alpha+d)$.
\end{lemma}
When $d \geq 3$, $\alpha > 2d/(d-2)$ and $\beta > 2$, we have
\[
	\lambda_1 = - \frac{1}{2(d+2)V_d^{2/d}} \int_\mathcal{X} \frac{\Delta f(x)}{f(x)^{2/d}} \,dx,
\]
which is finite under these assumptions; cf.\ the second part of Proposition~\ref{Prop:Moment} in Section~\ref{Sec:biasproof}.  Moreover, since, for each $l > 0$, we have $\frac{\Gamma(n)}{ \Gamma(n+2l/d)} = n^{-2l/d}\bigl\{1 + O(n^{-1})\bigr\}$, we deduce from Lemma~\ref{Lemma:Bias} that in this setting, 
\[
\sup_{f \in \mathcal{F}_{d,\theta}}\biggl| \mathbb{E}_f(\hat{H}_n) - H + \frac{\Gamma(k+2/d)}{2(d+2)V_d^{2/d}\Gamma(k)n^{2/d}} \int_\mathcal{X} \frac{\Delta f(x)}{f(x)^{2/d}} \,dx\biggr| = o\Bigl(\frac{k^{2/d}}{n^{2/d}}\Bigr)
.
\]
In particular, when $d \geq 4$ and $\int_\mathcal{X} \frac{\Delta f(x)}{f(x)^{2/d}} \,dx \neq 0$, the bias of the unweighted Kozachenko--Leonenko estimator precludes its efficiency.

%Moreover, when $k_0^* = k_{0,n}^*$ satisfies $k_0^* \rightarrow \infty$ as $n \rightarrow \infty$ and $k \in \{k_0^*,\ldots,k^*\}$, the dominant contributions to the bias simplify.  
%In fact, we have for each $l > 0$ that
%\[
%\frac{\Gamma(n)}{ \Gamma(n+2l/d)} = \frac{1}{n^{2l/d}}\bigl\{1 + O(n^{-1})\bigr\},
%\]
%uniformly for $k \in \{k_0^*,\ldots,k^*\}$.  Thus, when $d \geq 3$, $\alpha > 2d/(d-2)$ and $\beta > 2$, we have
%\[
%\sup_{f \in \mathcal{F}_{d,\theta}}\biggl| \mathbb{E}_f(\hat{H}_n) - H + \frac{1}{2(d+2)V_d^{2/d}} \frac{k^{2/d}}%{n^{2/d}}\int_\mathcal{X} \frac{\Delta f(x)}{f(x)^{2/d}} \,dx\biggr| = o\Bigl(\frac{k^{2/d}}{n^{2/d}}\Bigr),
%\]
%uniformly for $k \in \{k_0^*,\ldots,k^*\}$.  

On the other hand, Lemma~\ref{Lemma:Bias} motivates the definition of the class of weight vectors $\mathcal{W}^{(k)}$ in~\eqref{Eq:Wk}, and facilitates the expansion for the bias of the weighted Kozachenko--Leonenko estimator in Corollary~\ref{Cor:WeightedBias} below.  In particular, since $2( \lfloor d/4 \rfloor + 1)/d > 1/2$, we see that this result provides conditions under which the bias is $o(n^{-1/2})$ for suitably chosen $k$.  This explains why we let $\ell$ take values in the range $\{1,\ldots,\lfloor d/4 \rfloor\}$ in~\eqref{Eq:Wk}.
\begin{cor}
\label{Cor:WeightedBias}
Assume the conditions of Lemma~\ref{Lemma:Bias}.  If $w = w^{(k)} \in \mathcal{W}^{(k)}$ for $k \geq k_d$ and $\sup_{k \geq k_d} \|w^{(k)}\| < \infty$, then for every $\epsilon > 0$,
\[
\sup_{f \in \mathcal{F}_{d,\theta}} \bigl| \mathbb{E}_f(\hat{H}_n^w) - H(f)\bigr| = O \biggl( \max \biggl\{ \frac{k^{\frac{\alpha}{\alpha+d} - \epsilon}}{n^{\frac{\alpha}{\alpha+d} - \epsilon}} \, , \, \frac{k^{\frac{2( \lfloor d/4 \rfloor +1)}{d}}}{n^{\frac{2( \lfloor d/4 \rfloor +1)}{d}}} \, , \, \frac{k^{\frac{\beta}{d}}}{n^{\frac{\beta}{d}}}\biggr\} \biggr),
\]
uniformly for $k \in \{1,\ldots,k^*\}$.
\end{cor}
The proof of Lemma~\ref{Lemma:Bias} is given in Section~\ref{Sec:biasproof}, but we present here some of the main ideas that are particularly relevant for the case $d \geq 3$, $\alpha > 2d/(d-2)$ and $\beta \in (2,4]$. First, note that
\begin{equation}
\label{Eq:dFnxu}
	\frac{dF_{n,x}(u)}{du} = \mathrm{B}_{k,n-k}(p_{n,x,u})\frac{\partial p_{n,x,u}}{\partial u},
\end{equation}
where $\mathrm{B}_{a,b}(s) := \mathrm{B}_{a,b}^{-1}s^{a-1}(1-s)^{b-1}$ denotes the density of a $\mathrm{Beta}(a,b)$ random variable at $s \in (0,1)$, with $\mathrm{B}_{a,b} := \Gamma(a) \Gamma(b) / \Gamma(a + b)$.  For $x \in \mathcal{X}$ and $r>0$, define $h_x(r) := \int_{B_x(r)} f(y)\,dy$.  Since $h_x(r)$ is a continuous, non-decreasing function of $r$, we can define a left-continuous inverse for $s \in (0,1)$ by
\begin{equation}
\label{Eq:hxinvdef}
h_x^{-1}(s) := \inf\{r > 0:h_x(r) \geq s\} = \inf\{r > 0:h_x(r) = s\} ,
\end{equation}
so that $h_x(r) \geq s$ if and only if $r \geq h_x^{-1}(s)$. We use the approximation
\[
	V_d f(x) h_x^{-1}(s)^d \approx s - \frac{s^{1+2/d} \Delta f(x)}{2(d+2)V_d^{2/d}f(x)^{1+2/d}}
\]
for small $s > 0$, which is formalised in Lemma~\ref{Lemma:hxinvbounds}(ii) in Section~\ref{Sec:biasproof}. In the case $d \geq 3$, $\alpha > 2d/(d-2)$ and $\beta \in (2,4]$, the proof of Lemma~\ref{Lemma:Bias} can be seen as justifying the use of the above approximation in the following:
\begin{align*}
	\mathbb{E}(\hat{H}_n) &= \int_\mathcal{X} \! f(x) \!\int_0^\infty \log u \,dF_{n,x}(u) \,dx \\
	&= \int_\mathcal{X}\! f(x) \!\int_0^1 \log \biggl( \frac{V_d(n-1)h_x^{-1}(s)^d}{e^{\Psi(k)}} \biggr)\mathrm{B}_{k,n-k}(s) \,ds \,dx \\
	& \approx \int_\mathcal{X}\! f(x) \! \int_0^1 \biggl\{ \log \Bigl( \frac{(n-1)s}{e^{\Psi(k)}f(x)} \Bigr) - \frac{V_d^{-2/d} s^{2/d} \Delta f(x)}{2(d+2)f(x)^{1+2/d}} \biggr\}\mathrm{B}_{k,n-k}(s) \,ds \,dx \\
	& = \log(n-1) - \Psi(n) + H - \frac{V_d^{-2/d} \Gamma(k+2/d) \Gamma(n)}{2(d+2)\Gamma(k) \Gamma(n+2/d)} \int_\mathcal{X} \frac{\Delta f(x)}{f(x)^{2/d}} \,dx.
\end{align*}
Note that $\log(n-1) - \Psi(n) = -1/(2n) +o(1/n)$, which leads to the given bias expression.  The proof in other cases proceeds along similar lines.  These heuristics make clear that the function $h_x^{-1}(\cdot)$ plays a key role in understanding the bias. This function is in general complicated, though some understanding can be gained from the following uniform density example, where it can be evaluated explicitly.  This leads to an exact expression for the bias, even though the discontinuities mean that the density does not belong to $\mathcal{F}_{1,\theta}$ for any $\theta \in \Theta$.
\begin{eg}
Consider the uniform distribution, $U[0,1]$. For $x \leq 1/2$, we have
\[
	h_x^{-1}(s) =  \begin{cases} s/2, & \mbox{if } s \leq 2x \\ s-x, & \mbox{if } 2x<s \leq 1. \end{cases}
\]
It therefore follows that
\begin{align*}
	&\mathbb{E}(\hat{H}_n) - H = 2 \int_0^{1/2} \int_0^\infty \log u \,dF_{n,x}(u) \,dx \\
	&=2 \int_0^{1/2} \int_0^1 \log \biggl( \frac{2(n-1)h_x^{-1}(s)}{e^{\Psi(k)}} \biggr) \mathrm{B}_{k,n-k}(s) \,ds \,dx \\
	&=2 \int_0^1 \!\! \mathrm{B}_{k,n-k}(s) \biggl\{ \int_0^{s/2} \!\! \log (2(s-x)) \,dx + \int_{s/2}^{1/2}\!\! \log s \,dx \biggr\} \,ds +\log\Bigl(\frac{n-1}{e^{\Psi(k)}}\Bigr) \\
	& = \frac{k}{n} ( \log 4 -1) + \log(n-1) - \Psi(n).
\end{align*}
\end{eg}

\subsection{Discussion of conditions and weakening of conditions}
\label{Sec:Assumptions}

Recall the definitions of the quantity $M_{f,a,\beta}(x)$ and $\mathcal{A}$ from Section~\ref{Sec:Main}.  In addition to standard moment and boundedness assumptions, the condition $f \in \mathcal{F}_{d,\theta}$ requires that 
\begin{equation}
\label{Eq:GrowthCond}
\sup_{x:f(x) \geq \delta} M_{f,a,\beta}(x) \leq a(\delta) \quad \text{for all $\delta > 0$ and some $a \in \mathcal{A}$.}
\end{equation}
In this subsection, we explore the condition~\eqref{Eq:GrowthCond} further, with the aid of several examples.  

The condition~\eqref{Eq:GrowthCond} is reminiscent of more standard H\"{o}lder smoothness assumptions,
%; indeed, if $\mathcal{X}$ is a bounded, open subset of $\mathbb{R}^d$ and $\inf_{x \in \mathcal{X}} f(x) > 0$, then~\eqref{Eq:GrowthCond} holds when $f$ has H\"{o}lder smoothness $\beta$ on $\mathcal{X}$.  However, when $\inf_{x \in \mathcal{X}} f(x) = 0$, 
though we also require that the partial derivatives of the density vary less where $f$ is small.  On the other hand, we also allow the neighbourhoods of $x$ in the definition of $M_{f,a,\beta}(x)$ to shrink where $f(x)$ is small.  Roughly speaking, the condition requires that the partial derivatives of the density decay nearly as fast as the density itself in the tails of the distribution.   As a simple stability property, if~\eqref{Eq:GrowthCond} holds for a density $f_0$, then it also holds for any density from the location-scale family:
\[
\{f_\Sigma(\cdot)=|\Sigma|^{-1/2} f_0\bigl(\Sigma^{-1/2}(\cdot - \mu)\bigr): \mu \in \mathbb{R}^d, \Sigma=\Sigma^T \in \mathbb{R}^{d \times d} \ \text{positive definite}\}.
\]
This observation allows us to consider canonical representatives of location-scale families in the examples below.
\begin{prop}
\label{Prop:Examples}
For each of the following densities $f$, and for each $d \in \mathbb{N}$, there exists $\theta \in \Theta$ such that $f \in \mathcal{F}_{d,\theta}$: 
\begin{enumerate}[(i)]
\item $f(x) = f(x_1,\ldots,x_d) = (2\pi)^{-d/2}e^{-\|x\|^2/2}$, the standard normal density;
\item $f(x) = f(x_1,\ldots,x_d) \propto (1+ \|x\|^2 / \rho)^{-\frac{d+\rho}{2}}$, the multivariate-$t$ distribution with $\rho > 0$ degrees of freedom.
\end{enumerate}
Moreover, the following univariate density $f$ also belongs to $\mathcal{F}_{1,\theta}$ for suitable $\theta \in \Theta$:
\[
f(x) \propto \exp\Bigl(-\frac{1}{1-x^2}\Bigr)\mathbbm{1}_{\{x \in (-1,1)\}}.
\]
\end{prop}
The final part of Proposition~\ref{Prop:Examples} is included because it provides an example of a density $f$ that belongs to $\mathcal{F}_{1,\theta}$ for suitable $\theta \in \Theta$, even though there exist points $x_0 \in \mathbb{R}$ with $f(x_0) = 0$.   %Here it is crucial that the neighbourhoods of $x$ in the definition of $M_{f,a,\beta}(x)$ are allowed to shrink as $f(x)$ becomes small.

On the other hand, there are also examples, such as Example~\ref{Eg:Gamma} below, where the behaviour of $f$ near a point $x_0$ with $f(x_0) = 0$ precludes $f$ belonging to $\mathcal{F}_{d,\theta}$ for any $\theta \in \Theta$.  To provide some guarantees in such settings, we now give a very general condition under which our approach to studying the bias can be applied.  
\begin{prop}
\label{Prop:WeakCond}
Assume that $f$ is bounded, that $\mu_\alpha(f) < \infty $ for some $\alpha > 0$, and let $k^*$ be as in Lemma~\ref{Lemma:Bias}.  Let $a_n := 3(k+1)\log(n-1)$, let $r_x := \bigl\{\frac{2a_n}{V_d(n-1)f(x)}\bigr\}^{1/d}$, and assume further that there exists $\beta > 0$ such that the function on $\mathcal{X}$ given by 
\[
C_{n,\beta}(x) := \left\{ \begin{array}{ll} \sup_{y \in B_x^\circ(r_x)} |f(y)-f(x)|/\|y-x\|^\beta & \mbox{if $\beta \leq 1$,} \\
\sup_{y \in B_x^\circ(r_x)} \|\dot{f}(y) - \dot{f}(x)\|/\|y-x\|^{\beta-1} & \mbox{if $\beta > 1$,} \end{array} \right.
\]
is real-valued.  Suppose that $\mathcal{X}_n \subseteq \mathcal{X}$ is such that
\begin{equation}
\label{Eq:WeakCond}
\sup_{x \in \mathcal{X}_n} \Bigl(\frac{a_n}{n-1}\Bigr)^{\tilde{\beta}/d} \ \frac{C_{n,\tilde{\beta}}(x)}{f(x)^{1+\tilde{\beta}/d}} \rightarrow 0
\end{equation}
as $n \rightarrow \infty$, where $\tilde{\beta} := \beta \wedge 2$.  Then writing $q_n := \int_{\mathcal{X}_n^c} f$, we have for every $\epsilon > 0$ that
\begin{equation}
\label{Eq:WeakBias}
\mathbb{E}_f(\hat{H}_n) - H = O\biggl(\max\biggl\{ \frac{k^{\tilde{\beta}/d}}{n^{\tilde{\beta}/d}}\int_{\mathcal{X}_n} \frac{C_{n,\tilde{\beta}}(x)}{f(x)^{\tilde{\beta}/d}} \, dx \, , \, q_n^{1-\epsilon} \, , \, q_n\log n \, , \, \frac{1}{n}\biggr\}\biggr),
\end{equation}
uniformly for $k \in \{1,\ldots,k^*\}$.  
\end{prop}
To aid interpretation of Proposition~\ref{Prop:WeakCond}, we first remark that if $f \in \mathcal{F}_{d,\theta}$ for some $\theta = (\alpha,\beta,\gamma,\nu,a) \in \Theta$, then~\eqref{Eq:WeakCond} holds, with $\mathcal{X}_n := \{x \in \mathcal{X}:f(x) \geq \delta_n\}$, where $\delta_n$ is defined in~\eqref{Eq:deltan} below.  On the other hand, if $f \notin \mathcal{F}_{d,\theta}$, we may still be able to obtain explicit bounds on the terms in~\eqref{Eq:WeakBias} on a case-by-case basis, as in the following example. %When~\textbf{(A1)($\beta$)} does not hold, the integral on the right-hand side of~\eqref{Eq:WeakBias} may well diverge to infinity as $n$ increases even when $d > \tilde{\beta}$, which explains why our bounds on the bias are less explicit than in Theorem~\ref{biasthm}(i).  
\begin{eg}
\label{Eg:Gamma}
For $a > 1$, consider $f(x) = \Gamma(a)^{-1}x^{a-1}e^{-x}\mathbbm{1}_{\{x > 0\}}$, the density of the $\Gamma(a,1)$ distribution.  Then for any $\tau \in (0,1)$ small enough, we may take
\[
\mathcal{X}_n = \biggl[\Bigl(\frac{k}{n}\Bigr)^{\frac{1}{a}-\tau},(1-\tau)\log \frac{n}{k}\biggr]
\]
to deduce from Proposition~\ref{Prop:WeakCond} that for every $\epsilon > 0$,
\[
\mathbb{E}_f(\hat{H}_n) - H = o\Bigl(\frac{k^{1-\epsilon}}{n^{1-\epsilon}}\Bigr),
\]
uniformly for $k \in \{1,\ldots,k^*\}$. 
\end{eg}
Similar calculations show that the bias is of the same order for $\mathrm{Beta}(a,b)$ distributions with $a,b > 1$.  

\subsection{Asymptotic variance and normality}
\label{Sec:Var}

We now study the asymptotic variance of Kozachenko--Leonenko estimators under the assumption that the tuning parameter $k$ is diverging with $n$; the fixed $k$ case is deferred to the next subsection.
\begin{lemma}
\label{varthm}
Let $\theta = (\alpha,\beta,\gamma,\nu,a) \in \Theta$ with $\alpha > d$ and $\beta > 0$.  Let $k_0^*=k_{0,n}^*$ and $k_1^*=k_{1,n}^*$ denote any two deterministic sequences of positive integers with $k_0^* \leq k_1^*$, with $k_0^* / \log^5 n \rightarrow \infty$ and with $k_1^*=O(n^{\tau_1})$, where $\tau_1$ satisfies the condition in Theorem~\ref{unifweightedclt}.  Then for any $w = w^{(k)} \in \mathcal{W}^{(k)}$ with $\sup_{k \geq k_d} \|w^{(k)}\| < \infty$, we have
\[
\sup_{k \in \{k_0^*, \ldots, k_1^*\}}\sup_{f \in \mathcal{F}_{d,\theta}} \bigl|n\mathrm{Var}_f \hat{H}_n^w - V(f)\bigr| \rightarrow 0
\]
as $n \rightarrow \infty$. 
\end{lemma}
The proof of this lemma is lengthy, and involves many delicate error bounds, so we outline the main ideas in the unweighted case here.  First, we argue that 
\begin{align*}
	\Var \hat{H}_n &= n^{-1} \Var \log \xi_1 + (1-n^{-1}) \Cov( \log \xi_1, \log \xi_2) \\
	&= n^{-1} V(f) + \Cov\bigl( \log (\xi_1 f(X_1)), \log (\xi_2 f(X_2))\bigr) + o(n^{-1}), 
\end{align*}
where we hope to exploit the fact that $\xi_1 f(X_1) \stackrel{p}{\rightarrow} 1$.  The main difficulties in the argument are caused by the fact that handling the covariance above requires us to study the joint distribution of $(\xi_1,\xi_2)$, and this is complicated by the fact that $X_2$ may be one of the $k$ nearest neighbours of $X_1$ or vice versa, and more generally, $X_1$ and $X_2$ may have some of their $k$ nearest neighbours in common.  Dealing carefully with the different possible events requires us to consider separately the cases where $f(X_1)$ is small and large, as well as the proximity of $X_2$ to $X_1$.  Finally, however, we can apply a normal approximation to the relevant multinomial distribution (which requires that $k \rightarrow \infty$) to deduce the result.  We remark that under stronger conditions on $k$, it should also be possible to derive the same conclusion about the asymptotic variance of $\hat{H}_n$ while only assuming similar conditions on the density to those required in Proposition~\ref{Prop:WeakCond}, but we do not pursue this here.

\subsection{Fixed $k$}

A crucial step in the proof of Lemma~\ref{varthm} is the normal approximation to a certain multinomial distribution (cf.\ the bound on the term $W_4$).  This normal approximation is only valid when $k \rightarrow \infty$ as $n \rightarrow \infty$.  In this subsection, we present evidence to suggest that, when $k$ is fixed (i.e.\ not depending on $n$), then Kozachenko--Leonenko estimators are inefficient.  For simplicity, we focus on the unweighted version of estimator.
%Although the main focus in this paper is on efficient estimation and we are mainly concerned with the case of diverging $k$. To provide some justification for this approach, in this section we present heuristic arguments to suggest that efficiency can not be achieved with fixed $k$. 

Define the functions
\[
	\alpha_r(s,t) := \frac{1}{V_d} \mu_d \bigl( B_0(s^{1/d}) \cap B_{r^{1/d}e_1}(t^{1/d}) \bigr),
\]
where $e_1=(1,0, \ldots, 0)$ is the first element of the standard basis for $\mathbb{R}^d$ and $\mu_d$ denotes Lebesgue measure on $\mathbb{R}^d$. Also define the functions $T_k$ on $[0,\infty)^3$ by
\begin{align*}
	T_k(r,s,t) := e^{\alpha_r(s,t)} \sum_{\ell=0}^{L(r,s,t)} \sum_{i=0}^{I(r,s)-\ell} \sum_{j=0}^{J(r,t)-\ell} &\frac{\{s-\alpha_r(s,t)\}^i \{t-\alpha_r(s,t)\}^j \alpha_r^\ell(s,t)}{i! j! \ell!} \\
	& - \sum_{i=0}^{I(r,s)} \sum_{j=0}^{J(r,t)} \frac{s^it^j}{i! j!},
\end{align*}
where $L(r,s,t):=k-1-\mathbbm{1}_{\{r < \max(s,t)\}}$, $I(r,s):=k-1-\mathbbm{1}_{\{r<s\}}$, $J(r,t):=k-1-\mathbbm{1}_{\{r<t\}}$.

In the case $k=1$, this function appears in~\citet{DelattreFournier:16}, where the authors show that, under certain regularity conditions,
\[
	\lim_{n \rightarrow \infty} n \Var \hat{H}_n - V(f) = \Psi'(1) + \int_{[0,\infty)^3} e^{-s-t} \frac{T_1(r,s,t)}{st} \,dr \,ds \,dt - 1 + 2 \log 2.
\]
More generally, Poisson approximation to the same multinomial distribution mentioned above, together with analysis similar to the proof of Lemma~\ref{varthm}, suggests that for (fixed) $k \geq 2$,
\begin{align}
\label{Eq:AsympInf}
	&\lim_{n \rightarrow \infty} n \Var \hat{H}_n - V(f) = \,  \Psi'(k) + \int_{[0,\infty)^3} e^{-s-t} \frac{T_k(r,s,t)}{st} \,dr \,ds \,dt -1 \nonumber \\
	&\hspace{0.8cm}+ 2^{-(2k-2)} \binom{2k-2}{k-1} \{ \Psi(2k-1)-\Psi(k)-\log2\} \nonumber \\
	&\hspace{0.8cm}+ \frac{1}{k-1} \sum_{j=0}^{k-2} 2^{-k-j} \binom{k+j-1}{j} [1- (k-j) \{ \Psi(k+j)- \log2 - \Psi(k) \}].
\end{align}
Here, the $\Psi'(k)$ term arises as in~\eqref{Eq:LongDisplay}, the integral term arises from the Poisson approximation, the $-1$ arises as in~\eqref{Eq:LongerDisplay}, and the remaining terms come from the fact that $X_1$ can be one of the $k$ nearest neighbours of $X_2$, or vice-versa, which induces a singular component into the joint distribution function $F_{n,x,y}$ of $(\xi_1,\xi_2)$ given $(X_1,X_2)=(x,y)$.  It is interesting to observe that this asymptotic inflation of the variance is distribution-free; in Table~\ref{Table:Sim}, we tabulate numerical values for~\eqref{Eq:AsympInf} for a few values of $d$ and $k$.  These agree with those obtained by \citet{DelattreFournier:16} for the case $k=1$.

\begin{table}
\begin{center}
\begin{tabular}{ |c|c|c|c|c|c|c| } 
 \hline
$d \backslash k$ & $1$ & $2$ & $3$ & $4$ & $5$ \\ 
\hline
 $1$ & $2.14$ & $0.97$ & $0.64$ & $0.48$ & $0.39$ \\ 
\hline
 $2$ & $2.29 $ &$1.01$ & $0.64$ & $0.47$ & $0.38$ \\ 
 \hline
$3$ & $2.42$ & $ 1.03$ & $0.64$ & $0.47$ & $0.37$\\
\hline
$5$ & $2.61$& $1.05$ & $0.65$ & $0.47$ & $0.37$\\
\hline
$10$ & $2.85$ &$1.10$ & $0.68$ & $0.50$ & $0.40$ \\
\hline
\end{tabular}
\end{center}
\caption{\label{Table:Sim}Asymptotic variance inflation~\eqref{Eq:AsympInf} of the Kozachenko--Leonenko estimator for fixed $k$.} 
\end{table}
%\textbf{This table was made with $200000000$ samples per cell. The last two columns have are the result of combining the original samples with another $1200000000$ samples.}
%We expect that this can be proved using similar methods to those used in this paper, though we do not pursue this here in the interests of brevity.

%As was mentioned in the introduction, the asymptotic variance in Theorem~\ref{CLT} is best possible, as demonstrated by the following result, which is a consequence of Theorem 4.1 of \citet{IbragimovK1991}; see also \citet{Laurent1996}. Here $\Theta$ is the set of density functions on $\mathbb{R}^d$ such that $\int f \log^2 f < \infty$. 
%\begin{thm}
%\label{minimax}
%For any $f$ for which $\int f \log^2 f < \infty$ and $\epsilon>0$, let $V_\epsilon=\{ \tilde{f} \in \Theta : \int |f-\tilde{f}| < \epsilon\}$. Then for any estimator $\tilde{H}_n$,
%\[
%	\inf_{\epsilon} \liminf_{n \rightarrow \infty} \sup_{\tilde{f} \in V_{\epsilon}} n \mathbb{E}_{\tilde{f}} \{( \tilde{H}_n - H(\tilde{f}))^2 \} \geq \Var \log f(X_1).
%\]
%\end{thm}

\section{Lower bounds}
\label{Sec:LowerBound}

In this section, we address the optimality in a local asymptotic minimax sense of the limiting normalised risk $V(f)$ given in Theorem~\ref{unifweightedclt} using ideas of semiparametric efficiency \citep[e.g.][Chapter~25]{vanderVaart1998}.  For $f \in \mathcal{F}_{d,\theta}$, $t \geq 0$ and a Borel measurable function $g:\mathbb{R}^d \rightarrow \mathbb{R}$, define $f_{t,g}:\mathbb{R}^d \rightarrow [0,\infty)$ by
\begin{equation}
\label{Eq:Paths}
f_{t,g}(x) := \frac{2c(t)}{1 + e^{-2tg(x)}}f(x),
\end{equation}
where $c(t) := \bigl(\int_{\mathbb{R}^d}  \frac{2}{1 + e^{-2tg(x)}}f(x) \, dx\bigr)^{-1}$.  This definition ensures that $\{f_{t,g}:t \geq 0\}$ is differentiable in quadratic mean at $t=0$ with score function $g$ \citep[e.g.][Example~25.16]{vanderVaart1998}.  We say $(\tilde{H}_n)$ is an estimator sequence if $\tilde{H}_n: (\mathbb{R}^d)^{\times n} \rightarrow \mathbb{R}$ is a measurable function for each $n \in \mathbb{N}$.
\begin{thm}
\label{Thm:LowerBound}
Fix $d \in \mathbb{N}$, $\theta = (\alpha,\beta,\gamma,\mu,a) \in \Theta$ and $f \in \mathcal{F}_{d,\theta}$.  For $\lambda \in \mathbb{R}$, let $g_\lambda := \lambda\{\log f + H(f)\}$.  Then, writing $\mathcal{I}$ for the set of finite subsets of $\mathbb{R}$, we have for any estimator sequence $(\tilde{H}_n)$ that
\begin{equation}
\label{Eq:LowerBound}
\sup_{I \in \mathcal{I}} \liminf_{n \rightarrow \infty} \max_{\lambda \in I} n \mathbb{E}_{f_{n^{-1/2},g_\lambda}} \bigl[\bigl\{\tilde{H}_n-H(f_{n^{-1/2},g_\lambda})\bigr\}^2\bigr] \geq V(f).
\end{equation}
Moreover, whenever $t|\lambda| \leq \min(1,\{144V(f)\}^{-1/2})$, we have $f_{t,g_\lambda} \in \mathcal{F}_{d,\tilde{\theta}}$, where $\tilde{\theta} := \bigl(\alpha,\beta,4\gamma,4\mu, \tilde{a}\bigr) \in \Theta$, and $\tilde{a} \in \mathcal{A}$ is defined in~\eqref{Eq:atilde} in the online supplement.
\end{thm}
The proof of Theorem~\ref{Thm:LowerBound} reveals that, at every $f \in \mathcal{F}_{d,\theta}$, the entropy functional $H$ is differentiable relative to the tangent set $\{g_\lambda:\lambda \in \mathbb{R}\}$ with efficient influence function
\[
\tilde{\psi}_f := - \log f - H(f).
\]
This observation, together with Theorem~\ref{unifweightedclt}, confirms that under the assumptions on $\theta$, $w$ and $k$ in that result, the weighted Kozachenko--Leonenko estimator $\hat{H}_n^w$ is (asymptotically) efficient at $f \in \mathcal{F}_{d,\theta}$ in the sense that
\[
n^{1/2}\{\hat{H}_n^w - H(f)\} = \frac{1}{n^{1/2}}\sum_{i=1}^n \tilde{\psi}_f(X_i) + o_p(1)
\]
\citep[cf.][p.~367]{vanderVaart1998}.  Moreover, the second part of Theorem~\ref{Thm:LowerBound} and Theorem~\ref{unifweightedclt} imply in particular that, under these same conditions on $\theta$, $w$ and $k$, the estimator $\hat{H}_n^w$ attains the local asymptotic minimax lower bound, in the sense that
\[
\sup_{I \in \mathcal{I}} \lim_{n \rightarrow \infty} \max_{\lambda \in I} n \mathbb{E}_{f_{n^{-1/2},g_\lambda}} \bigl[\bigl\{\hat{H}_n^w - H(f_{n^{-1/2},g_\lambda})\bigr\}^2\bigr] = V(f).
\]

\section{Proofs of main results}
\label{sec:proof}
\subsection{Auxiliary results and proofs of Lemma~\ref{Lemma:Bias} and Corollary~\ref{Cor:WeightedBias}}
\label{Sec:biasproof}

Throughout the proofs, we write $a \lesssim b$ to mean that there exists $C > 0$, depending only on $d \in \mathbb{N}$ and $\theta \in \Theta$, such that $a \leq Cb$.  The proof of Lemma~\ref{Lemma:Bias} relies on the following two auxiliary results, whose proofs are given in Appendix~\ref{Appendix:Auxiliary}.
\begin{prop}
\label{Prop:Moment}
Let $\theta = (\alpha,\beta,\gamma,\nu,a) \in \Theta$, $d \in \mathbb{N}$ and $\tau \in \bigl(\frac{d}{\alpha+d},1\bigr]$.  Then
\[
\sup_{f \in \mathcal{F}_{d,\theta}} \int_{\{x:f(x) < \delta\}} a\bigl(f(x)\bigr) f(x)^\tau \, dx \rightarrow 0 
\]
as $\delta \searrow 0$.  Moreover, for every $\rho > 0$,
\[
\sup_{f \in \mathcal{F}_{d,\theta}} \int_{\mathcal{X}} a\bigl(f(x)\bigr)^\rho f(x)^\tau < \infty.
\]
\end{prop}
Recall the definition of $h_x^{-1}(\cdot)$ in~\eqref{Eq:hxinvdef}.  The first part of Lemma~\ref{Lemma:hxinvbounds} below provides crude but general bounds; the second gives much sharper bounds in a more restricted region.
%For $x \in \mathcal{X}$ and $r>0$ define $h_x(r) := \int_{B_x(r)} f(y)\,dy$.  Since $h_x(r)$ is a continuous, non-decreasing function of $r$, we can define a left-continuous inverse for $s \in (0,1)$ by
%\[
%h_x^{-1}(s) := \inf\{r > 0:h_x(r) \geq s\} = \inf\{r > 0:h_x(r) = s\} ,
%\]
%so that $h_x(r) \geq s$ if and only if $r \geq h_x^{-1}(s)$.  %The behaviour of $h_x^{-1}$ turns out to be important in our analysis of both the bias and the variance of generalised Kozachenko--Leonenko estimators; in particular, we will frequently make use of the bounds stated in the lemma below.
\begin{lemma}
\label{Lemma:hxinvbounds}
\begin{enumerate}
\item[(i)] Let $f \in \mathcal{F}_d$ and let $\alpha > 0$.  Then for every $s \in (0,1)$ and $x \in \mathbb{R}^d$,
\[
\Bigl(\frac{s}{V_d\|f\|_\infty}\Bigr)^{1/d} \leq h_x^{-1}(s) \leq \|x\| + \Bigl(\frac{\mu_\alpha(f)}{1-s}\Bigr)^{1/\alpha}.
\]
\item[(ii)] Fix $\theta = (\alpha,\beta,\gamma,\nu,a) \in \Theta$, and let $\mathcal{S}_n \subseteq (0,1)$, $\mathcal{X}_n \subseteq \mathbb{R}^d$ be such that 
\[
%\label{Eq:cn}
C_n := \sup_{f \in \mathcal{F}_{d,\theta}} \sup_{s \in \mathcal{S}_n} \sup_{x \in \mathcal{X}_n} \frac{a(f(x))^{d/(1 \wedge \beta)}s}{f(x)} \rightarrow 0.
\]
Then there exists $n_* = n_*(d,\theta) \in \mathbb{N}$ such that for all $n \geq n_*$, $s \in \mathcal{S}_n$, $x \in \mathcal{X}_n$ and $f \in \mathcal{F}_{d,\theta}$, we have
\[
\biggl|V_df(x)h_x^{-1}(s)^d - \sum_{l=0}^{\lceil \beta/2 \rceil -1} b_l(x) s^{1+2l/d} \biggr| \lesssim s \biggl\{ \frac{a(f(x))^{d/(2 \wedge \beta)}s}{f(x)} \biggr\}^{\beta/d},
\]
where $b_0(x)=1$ and $|b_l(x)| \lesssim a(f(x))^lf(x)^{-2l/d}$ for $l \geq 1$.  Moreover, if $\beta > 2$, then
\[
b_1(x) = -\frac{\Delta f(x)}{2(d+2)V_d^{2/d} f(x)^{1+2/d}}.
\]
\end{enumerate}
\end{lemma}
We are now in a position to prove Lemma~\ref{Lemma:Bias}.
\begin{proof}[Proof of Lemma~\ref{Lemma:Bias}](i)
We initially prove the result in the case $d \geq 3$, $\alpha > 2d/(d-2)$ and $\beta \in (2,4]$, where it suffices to show that 
\begin{align*}
\sup_{f \in \mathcal{F}_{d,\theta}}\biggl| \mathbb{E}_f(\hat{H}_n) - H &+ \frac{\Gamma(k+2/d) \Gamma(n)}{ 2(d+2)V_d^{2/d}\Gamma(k) \Gamma(n+2/d)}\int_\mathcal{X} \frac{\Delta f(x)}{f(x)^{2/d}} \,dx\biggr| \\
&\hspace{4cm}= O \biggl( \max \biggl\{ \frac{k^{\frac{\alpha}{\alpha+d} - \epsilon}}{n^{\frac{\alpha}{\alpha+d} - \epsilon}}\, , \, \frac{k^{\frac{\beta}{d}}}{n^{\frac{\beta}{d}}} \biggr\} \biggr)
\end{align*}
as $n \rightarrow \infty$, uniformly for $k \in \{1,\ldots,k^*\}$.  Fix $f \in \mathcal{F}_{d,\theta}$.  Define $c_n :=a(k/(n-1))^{1/(1 \wedge \beta)}$, let
\begin{equation}
\label{Eq:deltan}
\delta_n := k c_n^d \log^2 (n-1) /(n-1)
\end{equation}
and let $\mathcal{X}_n:=\{x:f(x) \geq \delta_n\}$.  Recall that $a_n := 3(k+1) \log(n-1)$ and 
let
\[
u_{x,s} := \frac{V_d(n-1)h_x^{-1}(s)^d}{e^{\Psi(k)}}.
\]
The proof is based on~\eqref{Eq:dFnxu} and Lemma~\ref{Lemma:hxinvbounds}(ii), which allow us to make the transformation $s=p_{n,x,u}=h_x(r_{n,u})$.  Writing $R_i, \, i=1, \ldots, 5$ for remainder terms to be bounded at the end of the proof, we can write
\begin{align*}
	&\mathbb{E}(\hat{H}_n) = \int_{\mathcal{X}} f(x) \int_0^\infty \log u \, dF_{n,x}(u) \,dx \\
	&= \int_{\mathcal{X}_n} f(x) \int_0^1 \mathrm{B}_{k,n-k}(s)\log u_{x,s} \,ds\,dx +R_1 \\
	& = \int_{\mathcal{X}_n} f(x) \int_0^\frac{a_n}{n-1} \mathrm{B}_{k,n-k}(s)\log u_{x,s} \,ds\,dx +R_1 +R_2 \\
	& = \int_{\mathcal{X}_n} \! \! f(x) \int_0^{ \frac{a_n}{n-1}} \Bigl\{ \log \Bigl( \frac{(n-1)s}{e^{\Psi(k)}f(x)} \Bigr) \\
	& \hspace{100pt} -\frac{V_d^{-2/d}s^{2/d} \Delta f(x)}{2(d+2)f(x)^{1+2/d}} \Bigr\} \mathrm{B}_{k,n-k}(s) \,ds \,dx + \sum_{i=1}^3 R_i \\
	& = \int_{\mathcal{X}_n} f(x) \biggl\{ \log \biggl( \frac{n-1}{f(x)} \biggr)  - \Psi(n) - \frac{ V_d^{-2/d}\mathrm{B}_{k+2/d,n-k} \, \Delta f(x)}{2(d+2)\mathrm{B}_{k,n-k} \, f(x)^{1+2/d}} \biggr\} \,dx + \sum_{i=1}^4 R_i \\
	&  = H \! + \! \log(n-1) \! - \! \Psi(n) - \frac{V_d^{-2/d}\Gamma(k+2/d)\Gamma(n)}{2(d+2)\Gamma(k)\Gamma(n+2/d)} \int_{\mathcal{X}_n} \! \frac{\Delta f(x)}{f(x)^{2/d}} \,dx+ \sum_{i=1}^5 R_i.
\end{align*}
After multiplying the integrand by an appropriate positive power of $\delta_n/f(x)$, the first part of Proposition~\ref{Prop:Moment} tells us that for every $\epsilon > 0$,
\[
	\sup_{k \in \{1,\ldots,k^*\}} \frac{k^{2/d}}{n^{2/d}}\sup_{f \in \mathcal{F}_{d,\theta}} \int_{\mathcal{X}_n^c} \frac{\Delta f(x)}{f(x)^{2/d}} \,dx = O \biggl( \frac{k^{\frac{\alpha}{\alpha+d} - \epsilon}}{n^{\frac{\alpha}{\alpha+d} - \epsilon}}\biggr)
\]
as $n \rightarrow \infty$.  Since $\log(n-1)- \Psi(n) = O(1/n)$, it now remains to bound $R_1, \ldots, R_5$.  Henceforth, to save repetition, we adopt without further mention the convention that whenever an error term inside $O(\cdot)$ or $o(\cdot)$ depends on $k$, this error is uniform for $k \in \{1,\ldots,k^*\}$; thus $g(n,k) = h(n,k) + o(1)$ as $n \rightarrow \infty$ means $\sup_{k \in \{1,\ldots,k^*\}}|g(n,k) - h(n,k)| \rightarrow 0$ as $n \rightarrow \infty$.

\medskip

\emph{To bound $R_1$}. By Lemma~\ref{Lemma:hxinvbounds}(i), we have $V_d^\alpha \mu_\alpha(f)^d \|f\|_\infty^\alpha \geq \alpha^\alpha d^d/(\alpha+d)^{\alpha+d}$.  Hence
\begin{align}
\label{Eq:LotsofTerms}
|&\log u_{x,s}| \leq \log(n-1) +|\Psi(k)| - \log s + |\log \|f\|_\infty| + |\log V_d| \nonumber \\ 
&\hspace{2cm}+ \frac{d}{\alpha}|\log \mu_\alpha(f)| - \frac{d}{\alpha}\log(1-s) + d\log\biggl(1+\frac{\|x\|}{\mu_\alpha^{1/\alpha}(f)}\biggr) \nonumber \\
&\leq \log(n-1) +  |\Psi(k)| - \log s + \max\biggl\{\log \gamma \, , \, \frac{1}{\alpha} \log\biggl(\frac{V_d^\alpha\nu^d(\alpha+d)^{\alpha+d}}{\alpha^\alpha d^d}\biggr)\biggr\} \nonumber \\
&\hspace{0.5cm}+ |\log V_d| + \frac{d}{\alpha}\max\biggl\{\log \nu \, , \, \frac{1}{d}\log\biggl(\frac{V_d^\alpha  \gamma^\alpha (\alpha+d)^{\alpha+d}}{\alpha^\alpha d^d}\biggr)\biggr\} - \frac{d}{\alpha}\log(1-s) \nonumber \\
&\hspace{4cm}+ d\log\biggl(1+\frac{\|x\|(\alpha+d)^{\frac{1}{\alpha} + \frac{1}{d}}V_d^{1/d}\gamma^{1/d}}{\alpha^{1/d}d^{1/\alpha}}\biggr).
\end{align}
Moreover, for any $C_0, C_1 \geq 0, \epsilon \in (0,\alpha)$ and $\epsilon' \in (0,\epsilon)$, we have by H\"older's inequality that
\begin{align*}
	\sup_{f \in \mathcal{F}_{d,\theta}} \int_{\mathcal{X}_n^c}f(x)&\bigl\{C_0 + \log(1+C_1\|x\|)\bigr\} \, dx \nonumber \\
&\leq \delta_n^{\frac{\alpha-\epsilon'}{\alpha+d}} \sup_{f \in \mathcal{F}_{d,\theta}}\int_{\mathcal{X}} f(x)^{\frac{d+\epsilon'}{\alpha+d}} \bigl\{C_0+\log(1+C_1\|x\|)\bigr\} \, dx \nonumber \\
	& \leq \delta_n^{\frac{\alpha-\epsilon'}{\alpha+d}}(1+\nu)^\frac{d+\epsilon'}{\alpha+d} \biggl[ \int_{\mathbb{R}^d} \frac{\bigl\{C_0 + \log(1+C_1 \|x\| )\bigr\}^{\frac{\alpha+d}{\alpha-\epsilon'}}}{(1+\|x\|^\alpha)^\frac{d+\epsilon'}{\alpha-\epsilon'}} \,dx \biggr]^{\frac{\alpha-\epsilon'}{\alpha+d}} \nonumber \\
&= o \biggl( \frac{k^{\frac{\alpha}{\alpha+d}-\epsilon}}{n^{\frac{\alpha}{\alpha+d}-\epsilon}} \biggr).
\end{align*}
Since $|\mathbb{E}(\log \mathrm{B})| = \Psi(a+b) - \Psi(a)$ when $\mathrm{B} \sim \mathrm{Beta}(a,b)$, we deduce that for each $\epsilon > 0$,
\[
R_1 = \int_{\mathcal{X}_n^c} f(x) \int_0^1 \mathrm{B}_{k,n-k}(s) \log u_{x,s} \,ds \,dx = o \biggl( \frac{k^{\frac{\alpha}{\alpha+d}-\epsilon}}{n^{\frac{\alpha}{\alpha+d}-\epsilon}} \biggr)
\]
as $n \rightarrow \infty$, uniformly for $f \in \mathcal{F}_{d,\theta}$.  

\medskip

\emph{To bound $R_2$.} For random variables $\mathrm{B}_1 \sim \text{Beta}(k,n-k)$ and $B_2 \sim \text{Bin}\bigl(n-1,a_n/(n-1)\bigr)$ we have that for every $\epsilon >0$,
\begin{equation}
\label{Eq:betatail}
	\mathbb{P}\bigl(\mathrm{B}_1 \geq a_n/(n-1)\bigr) = \mathbb{P}(B_2 \leq k-1) \leq \exp\biggl(-\frac{(a_n-k+1)^2}{2a_n}\biggr) = o( n^{-(3- \epsilon)} ),
\end{equation}
where the inequality follows from standard bounds on the left-hand tail of the binomial distribution (see, e.g.\ \citet{Shorack:09}, Equation~(6), page 440).  Now, for any $C_1 > 0$, we have $\alpha \log(1+C_1\|x\|) \leq (1+C_1\|x\|)^\alpha-1$, so that $\sup_{f \in \mathcal{F}_{d,\theta}} \int_{\mathcal{X}} f(x) \log(1+C_1\|x\|) \, dx < \infty$.  Moreover,
\[
-\int_{\frac{a_n}{n-1}}^1 \log(1-s)\mathrm{B}_{k,n-k}(s) \, ds \leq \frac{n-1}{n-k-1} \int_{\frac{a_n}{n-1}}^1 \mathrm{B}_{k,n-k-1}(s) \, ds = o(n^{-(3-\epsilon)}),
\]
for every $\epsilon > 0$, by a virtually identical argument to~\eqref{Eq:betatail}.  We therefore deduce from these facts and~\eqref{Eq:LotsofTerms} that for each $\epsilon>0$,
\begin{equation}
\label{Eq:R2bound}
	R_2= \int_{\mathcal{X}_n} f(x) \int_\frac{a_n}{n-1}^1 \mathrm{B}_{k,n-k}(s) \log u_{x,s} \,ds \,dx= o(n^{-(3-\epsilon)}),
\end{equation}
which again holds uniformly in $f \in \mathcal{F}_{d,\theta}$. %(With further work, it can in fact be shown that $R_2 = o(n^{-(3-\epsilon)})$, but this is not necessary for our overall conclusion.)

\medskip

\emph{To bound $R_3$.} We can write
\begin{align*}
	R_3 %&= \int_{\mathcal{X}_n} f(x) \int_0^\frac{a_n}{n-1} \biggl\{ \log \biggl(\frac{V_d f(x) h_x^{-1}(s)^d}{s} \biggr) + \frac{s^{2/d} \Delta f(x)}{2(d+2)V_d^{2/d} f(x)^{1+2/d}} \biggr\} \mathrm{B}_{k,n-k}(s) \, ds \,dx \\
	& = \int_{\mathcal{X}_n} f(x) \int_0^\frac{a_n}{n-1} \biggl[ \biggl\{ \log \biggl(\frac{V_d f(x) h_x^{-1}(s)^d}{s} \biggr) - \frac{V_d f(x) h_x^{-1}(s)^d -s}{s} \biggl\} \\
	& \hspace{50pt} + \biggl\{ \frac{V_d f(x) h_x^{-1}(s)^d -s}{s}+ \frac{V_d^{-2/d}s^{2/d} \Delta f(x)}{2(d+2) f(x)^{1+2/d}} \biggr\} \biggr] \mathrm{B}_{k,n-k}(s) \, ds \,dx \\
	&=: R_{31}+R_{32},
\end{align*}
say.  Now, note that 
\[
\sup_{k \in \{1,\ldots,k^*\}} \sup_{f \in \mathcal{F}_{d,\theta}} \sup_{s \in (0,a_n/(n-1)]} \sup_{x \in \mathcal{X}_n} \frac{a(f(x))^d s}{f(x)} \leq \frac{6}{\log(n-1)} \rightarrow 0.
\]
It follows by Lemma~\ref{Lemma:hxinvbounds}(ii) that there exist a constant $C = C(d,\theta) >0$ and $n_1 = n_1(d,\theta) \in \mathbb{N}$ such that for $n \geq n_1$, $k \in \{1,\ldots,k^*\}$, $s \leq a_n/(n-1)$ and $x \in \mathcal{X}_n$, 
\[
	\biggl| \frac{V_df(x)h_x^{-1}(s)^d-s}{s} + \frac{s^{2/d} \Delta f(x)}{2(d+2)V_d^{2/d}f(x)^{1+2/d}} \biggr| \leq C\biggl\{ \frac{sa(f(x))^{d/2}}{f(x)} \biggr\}^{\beta/d},
\]
and
\[
	\biggl| \frac{V_df(x)h_x^{-1}(s)^d-s}{s}\biggr| \leq \frac{d^{1/2}V_d^{-2/d}s^{2/d}a(f(x))}{2(d+2)f(x)^{2/d}}+ C\biggl\{ \frac{sa(f(x))^{d/2}}{f(x)} \biggr\}^{\beta/d} \leq \frac{1}{2}.
\]
Thus, for $n \geq n_1$ and $k \in \{1,\ldots,k^*\}$, using the fact that $| \log(1+z)-z| \leq z^2$ for $|z| \leq 1/2$,
\begin{align*}
	|R_{31}| &\leq 2 \int_{\mathcal{X}_n} f(x) \int_0^1 \biggl[ \biggl\{ \frac{dV_d^{-4/d}s^{4/d}a(f(x))^2}{4(d+2)^2f(x)^{4/d}} \\
	& \hspace{4cm} + C^2 \biggl\{ \frac{sa(f(x))^{d/2}}{f(x)} \biggr\}^{2\beta/d} \biggr] \mathrm{B}_{k,n-k}(s) \,ds \,dx \\
	& \leq \frac{ dV_d^{-4/d}\Gamma(k+4/d) \Gamma(n)}{2(d+2)^{2} \Gamma(k) \Gamma(n+4/d)} \int_{\mathcal{X}_n} a(f(x))^2 f(x)^{1-4/d} \, dx \\
	& \hspace{1cm} + \frac{2C^{2} \Gamma(k+2\beta/d) \Gamma(n)}{ \Gamma(k) \Gamma(n+2\beta/d)} \int_{\mathcal{X}_n } a(f(x))^\beta f(x)^{1-2\beta/d} \,dx.
\end{align*}
On the other hand, we also have for $n \geq n_1$ and $k \in \{1,\ldots,k^*\}$ that
\begin{align*}
	|R_{32}| & \leq C \int_{\mathcal{X}_n} f(x) \int_0^1 \biggl\{  \frac{sa(f(x))^{d/2}}{f(x)} \biggr\}^{\beta/d} \mathrm{B}_{k,n-k}(s) \,ds \,dx \\
	& \leq \frac{C \Gamma(k+\beta/d) \Gamma(n)}{\Gamma(k) \Gamma(n+\beta/d)} \int_{\mathcal{X}_n} a(f(x))^{\beta/2} f(x)^{1-\beta/d} \,dx.
\end{align*}
Multiplying each of the integrals by $f(x)/\delta_n$ to an appropriate positive power if necessary and by the second part of Proposition~\ref{Prop:Moment}, for every $\epsilon > 0$,
\[
	\max(|R_{31}|,|R_{32}|) = O \biggl( \max \biggl\{ \frac{k^{\frac{\alpha}{\alpha+d} - \epsilon}}{n^{\frac{\alpha}{\alpha+d} - \epsilon}}\, , \, \frac{k^{\frac{\beta}{d}}}{n^{\frac{\beta}{d}}} \biggr\} \biggr),
\]
uniformly for $f \in \mathcal{F}_{d,\theta}$.

\medskip

\emph{To bound $R_4$}. We have
\[
R_4 = \int_{\mathcal{X}_n} \! \! f(x) \int_{\frac{a_n}{n-1}}^1 \biggl\{ \log \biggl( \frac{(n-1)s}{e^{\Psi(k)}f(x)} \biggr) - \frac{V_d^{-2/d}s^{2/d} \Delta f(x)}{2(d+2)f(x)^{1+2/d}} \biggr\} \mathrm{B}_{k,n-k}(s) \,ds \,dx.
\]
Consider the random variable $\mathrm{B}_1 \sim \text{Beta}(k,n-k)$. Then, using \eqref{Eq:betatail} and the fact that $(n-1)s/e^{\Psi(k)} \geq 1$ for $s \geq a_n/(n-1)$ and $n \geq 3$, we conclude that for every $\epsilon > 0$ and $n \geq 3$,
\begin{align*}
	|R_4| &\leq \biggl\{ \log \Bigl(\frac{n-1}{e^{\Psi(k)}}\Bigr) \! + \! \int_{\mathcal{X}_n} \! \! \! f(x) \biggl( |\log f(x)| + \frac{a(f(x))}{f(x)^{\frac{2}{d}}V_d^{\frac{2}{d}}} \biggr) \, dx \biggr\} \mathbb{P}\Bigl(\mathrm{B}_1 \geq \frac{a_n}{n-1}\Bigr) \\
	&=o(n^{-(3-\epsilon)}),
\end{align*}
uniformly for $f \in \mathcal{F}_{d,\theta}$, where, by Lemma~\ref{Lemma:VfBounds}(i) in the online supplement, we have $\sup_{f \in \mathcal{F}_{d,\theta}} \int_{\mathcal{X}_n} f(x) |\log f(x)| \,dx < \infty$.

\medskip

\emph{To bound $R_5$.} We use the fact that for $f \in \mathcal{F}_{d,\theta}$, $x \in \mathcal{X}$ and $\epsilon' > 0$, 
\begin{align*}
|\log f(x)| &\leq \bigl|\log \|f\|_\infty\bigr| + \log \Bigl(\frac{\|f\|_\infty}{f(x)}\Bigr) \\
&\leq \max\biggl\{\log \gamma \, , \,\log V_d + \frac{1}{\alpha} \log\biggl(\frac{\nu^d(\alpha+d)^{\alpha+d}}{\alpha^\alpha d^d}\biggr)\biggr\} + \frac{1}{\epsilon'}\Bigl(\frac{\gamma}{f(x)}\Bigr)^{\epsilon'}.
\end{align*}
It follows from the first part of Proposition~\ref{Prop:Moment} (having replaced $a(\delta)$ with $\max\{a(\delta),|\log \delta|\}$ if necessary) that for each $\epsilon > 0$,
\[
	R_5 = \int_{\mathcal{X}_n^c} f(x) \{ \log(n-1) - \Psi(n) - \log f(x) \} \,dx = o\biggl(\frac{k^{\frac{\alpha}{\alpha+d} - \epsilon}}{n^{\frac{\alpha}{\alpha+d} - \epsilon}} \biggr)
\]
uniformly in $f \in \mathcal{F}_{d,\theta}$.  The claim follows when $d \geq 3$, $\alpha > 2d/(d-2)$ and $\beta \in (2,4]$.

\medskip

%\emph{To bound $R_6$}. By Stirling's formula,
%\begin{align*}
%R_6 &= \frac{n^{2/d}\Gamma(n)}{\Gamma(n+2/d)} - 1 \\
%&= \exp\biggl\{\frac{2}{d} -\Bigl(n - \frac{1}{2} +\frac{2}{d}\Bigr)\log\Bigl(1+\frac{2}{dn}\Bigr) + O(n^{-1})\biggr\} - 1 = O(n^{-1}),
%\end{align*}
%as required.  

We now consider the case where either $d \leq 2$ or $\alpha \leq 2d/(d-2)$ or $\beta \in (0,2]$, for which we need only show that
\[
\sup_{f \in \mathcal{F}_{d,\theta}}| \mathbb{E}_f(\hat{H}_n) - H| = O \biggl( \max \biggl\{ \frac{k^{\frac{\alpha}{\alpha+d} - \epsilon}}{n^{\frac{\alpha}{\alpha+d} - \epsilon}}\, , \, \frac{k^{\frac{\beta}{d}}}{n^{\frac{\beta}{d}}} \biggr\} \biggr).
\]
The calculation here is very similar, but we approximate $\log u_{x,s}$ simply by $\log \bigl( \frac{(n-1)s}{e^{\Psi(k)}f(x)} \bigr)$.  Writing $R_1',\ldots,R_5'$ for the modified error terms, we obtain
\[
\mathbb{E}_f(\hat{H}_n) = H + \log(n-1) - \Psi(n) + \sum_{i=1}^5 R_i'.
\]
Here, $R_1' = R_1 = o\bigl\{\bigl(\frac{k}{n}\bigr)^{\alpha/(\alpha+d) - \epsilon}\bigr\}$, and $R_2' = R_2 = o(n^{-(3-\epsilon)})$, for every $\epsilon > 0$ in both cases.  On the other hand,
\begin{align*}
R_3' &= \int_{\mathcal{X}_n} f(x) \int_0^\frac{a_n}{n-1} \log \biggl(\frac{V_d f(x) h_x^{-1}(s)^d}{s} \biggr) \mathrm{B}_{k,n-k}(s) \, ds \,dx \\
	&= O \biggl( \max \biggl\{ \frac{k^{\frac{\alpha}{\alpha+d} - \epsilon}}{n^{\frac{\alpha}{\alpha+d} - \epsilon}} \, , \, \frac{k^{\beta/d}}{n^{\beta/d}} \biggr\} \biggr)
\end{align*}
for every $\epsilon > 0$, by Lemma~\ref{Lemma:hxinvbounds}(ii).  Similarly, for every $\epsilon > 0$,
\[
R_4' = \int_{\mathcal{X}_n} f(x) \int_{\frac{a_n}{n-1}}^1 \log\biggl(\frac{(n-1)s}{e^{\Psi(k)}f(x)}\biggr) \, \mathrm{B}_{k,n-k}(s) \, ds \, dx = o(n^{-(3-\epsilon)}),
\]
and $R_5' = R_5 = o\bigl\{\bigl(\frac{k}{n}\bigr)^{\alpha/(\alpha+d) - \epsilon}\bigr\}$.  All of these bounds hold uniformly in $f \in \mathcal{F}_{d, \theta}$, so the claim is established for this setting.

Finally, consider now the case $d \geq 3$, $\alpha > 2d/(d-2)$ and $\beta > 4$.  Again the calculation is very similar to the earlier cases, with the with the main difference being that in bounding the error corresponding to $R_3$, we require a higher-order Taylor expansion of
\[
	\log \biggl( 1+ \frac{V_df(x)h_x^{-1}(s)^d -s}{s} \biggr).
\]
This can be done using Lemma~\ref{Lemma:hxinvbounds}(ii); we omit the details for brevity.
\end{proof}

\begin{proof}[Proof of Corollary~\ref{Cor:WeightedBias}]
It is convenient to write $d' := \lfloor d/4\rfloor + 1$ and $\beta' := \lceil \beta/2\rceil-1$.  We have
\begin{align*}
|\mathbb{E}_f(\hat{H}_n^w) - H| &= \biggl|\sum_{j=1}^k w_j\biggl\{\mathbb{E}_f(\log \xi_{(j),1}) - H - \sum_{l=1}^{\lfloor d/4\rfloor} \frac{\Gamma(j+2l/d)\Gamma(n)}{\Gamma(j)\Gamma(n+2l/d)} \lambda_l\biggr\}\biggr| \\
&\leq \biggl|\sum_{j=1}^k w_j\biggl\{\mathbb{E}_f(\log \xi_{(j),1}) - H - \sum_{l=1}^{\beta'} \frac{\Gamma(j+2l/d)\Gamma(n)}{\Gamma(j)\Gamma(n+2l/d)} \lambda_l\biggr\}\biggr| \\
&\hspace{3cm}+ \biggl|\sum_{j=1}^k w_j \sum_{l=d'}^{\beta'} \frac{\Gamma(j+2l/d)\Gamma(n)}{\Gamma(j)\Gamma(n+2l/d)} \lambda_l\biggr\}\biggr|.
\end{align*}
The first term can be bounded, uniformly for $f \in \mathcal{F}_{d,\theta}$ and $k \in \{1,\ldots,k^*\}$, using Lemma~\ref{Lemma:Bias}.  For the second term, we can use monotonicity properties of ratios of gamma functions to write
\begin{align*}
\biggl|\sum_{j=1}^k w_j &\sum_{l=d'}^{\beta'} \frac{\Gamma(j+2l/d)\Gamma(n)}{\Gamma(j)\Gamma(n+2l/d)} \lambda_l\biggr\}\biggr| \leq \max_{d' \leq \ell \leq \beta'} |\lambda_\ell| \sum_{j=1}^k |w_j| \sum_{l=d'}^{\beta'} \frac{\Gamma(k+2l/d)\Gamma(n)}{\Gamma(k)\Gamma(n+2l/d)} \\
&\leq d^{1/2}\|w\|\bigl(\beta'-d'+1\bigr)\frac{\Gamma(k+2d'/d)\Gamma(n)}{\Gamma(k)\Gamma(n+2d'/d)}\max_{d' \leq l \leq \beta'} |\lambda_l| = O\Bigl(\frac{k^{2d'/d}}{n^{2d'/d}}\Bigr),
\end{align*}
uniformly for $f \in \mathcal{F}_{d,\theta}$.  The result follows.
\end{proof}

\subsection{Proof of Lemma~\ref{varthm}}

Since this proof is long, we focus here on the main argument, and defer proofs of bounds on the many error terms to Appendix~\ref{Appendix:varproof}.
%This proof is long, so we focus here on the main argument, deferring proofs of bounds on the many error terms to Appendix~\ref{Appendix:varproof}.
\begin{proof}[Proof of Lemma~\ref{varthm}]
We employ the same notation as in the proof of Lemma~\ref{Lemma:Bias}, except that we redefine $\delta_n$ so that $\delta_n := k c_n^d \log^3 (n-1) /(n-1)$. We write $\mathcal{X}_n := \{x:f(x) \geq \delta_n\}$ for this newly-defined $\delta_n$.  Similar to the proof of Lemma~\ref{Lemma:Bias}, all error terms inside $O(\cdot)$ and $o(\cdot)$ that depend on $k$ are uniform for $k \in \{k_0^*,\ldots,k_1^*\}$, and we now adopt the additional convention that, where relevant, these error terms are also uniform for $f \in \mathcal{F}_{d,\theta}$.  By the nested properties of the classes $\mathcal{F}_{d,\theta}$ with respect to the smoothness parameter $\beta$, we may assume without loss of generality that $\beta \in (0,1]$.  We first deal with the variance of the unweighted estimator $\hat{H}_n$, and note that
\begin{align}
\label{Eq:ThreeMainTerms}
	\Var \hat{H}_n &= n^{-1} \Var \log \xi_1 + ( 1- n^{-1}) \Cov (\log \xi_1, \log \xi_2) \nonumber \\
	& = n^{-1} \Var \log \xi_1 + ( 1- n^{-1}) \bigl\{ \Cov\bigl( \log (\xi_1 f(X_1)), \log (\xi_2 f(X_2))\bigr) \nonumber \\
	& \hspace{75pt} - 2 \Cov\bigl(\log(\xi_1 f(X_1)), \log f(X_2)\bigr) \bigr\}.
\end{align}
We claim that for every $\epsilon > 0$,
\begin{equation}
\label{Eq:Varlogxi}
\Var \log \xi_1 = V(f) + \frac{1}{k}\{1+o(1)\} + O\biggl\{\max\biggl(\frac{k^{\beta/d}}{n^{\beta/d}} \log n \, , \, \frac{k^{\frac{\alpha}{\alpha+d}-\epsilon}}{n^{\frac{\alpha}{\alpha+d}-\epsilon}}\biggr)\biggr\}
\end{equation}
as $n \rightarrow \infty$.  The proof of this claim uses similar methods to those in the proof of Lemma~\ref{Lemma:Bias}.  In particular, writing $S_1,\ldots,S_5$ for remainder terms to be bounded later, we have
\begin{align}
\label{Eq:LongDisplay}
	\mathbb{E}(\log^2 \xi_1) &= \int_\mathcal{X} f(x) \int_0^\infty \log^2 u \,dF_{n,x}(u) \,dx \nonumber \\
	&\hspace{-0.5cm}= \int_{\mathcal{X}_n} f(x) \int_0^1 \mathrm{B}_{k,n-k}(s) \log^2 u_{x,s} \,ds \,dx + S_1 \nonumber\\
	&\hspace{-0.5cm}=\int_{\mathcal{X}_n} f(x) \int_0^\frac{a_n}{n-1} \mathrm{B}_{k,n-k}(s) \log^2 u_{x,s} \,ds \,dx + S_1 +S_2 \nonumber \\
	&\hspace{-0.5cm}=\int_{\mathcal{X}_n} f(x) \int_0^\frac{a_n}{n-1} \log^2 \biggl( \frac{(n-1)s}{e^{\Psi(k)}f(x)} \biggr) \mathrm{B}_{k,n-k}(s) \,ds \,dx + S_1 +S_2 + S_3 \nonumber \\
	&\hspace{-0.5cm}=\int_{\mathcal{X}_n} f(x) \bigl[ \log^2 f(x) - 2 \{\log(n-1)- \Psi(n)\}\log f(x) \nonumber \\
	& \hspace{50pt} + \Psi'(k)-\Psi'(n)+\{\log(n-1)-\Psi(n)\}^2 \bigr] \,dx + \sum_{i=1}^4 S_i \nonumber \\
%	&= \int_\mathcal{X} f(x) \log^2 f(x) \,dx + 2H(f)\{\log(n-1)- \Psi(n)\} \nonumber \\
%	& \hspace{100pt} + \Psi'(k)-\Psi'(n) + \{\log(n-1)-\Psi(n)\}^2  +\sum_{i=1}^5 S_i \nonumber \\
	&\hspace{-0.5cm}= \int_\mathcal{X} f(x) \log^2 f(x) \,dx + \sum_{i=1}^5 S_i + \frac{1}{k}\{1+o(1)\}, 
\end{align}
as $n \rightarrow \infty$.  %Moreover, from Theorem~\ref{biasthm}, we know that $\{\mathbb{E}(\log \xi_1)\}^2 = H(f)^2 + o(1)$ and it now suffices to bound the remainder terms.
In Appendix~\ref{Appendix:S}, we show that for every $\epsilon > 0$,
\begin{equation}
\label{Eq:Sbounds}
\sum_{i=1}^5 |S_i| = O\biggl\{\max\biggl(\frac{k^{\beta/d}}{n^{\beta/d}} \log n \, , \, \frac{k^{\frac{\alpha}{\alpha+d}-\epsilon}}{n^{\frac{\alpha}{\alpha+d}-\epsilon}}\biggr)\biggr\}
\end{equation}
as $n \rightarrow \infty$.  Combining~\eqref{Eq:LongDisplay} with~\eqref{Eq:Sbounds} and Lemma~\ref{Lemma:Bias}, we deduce that~\eqref{Eq:Varlogxi} holds.
%\[
%\Var \log \xi_1 = V(f) + O\biggl[\max\biggl\{\frac{k^{\beta/d}}{n^{\beta/d}} \log n \, , \, \frac{k^{\frac{\alpha}{\alpha+d} - \epsilon}}{n^{\frac{\alpha}{\alpha+d} - \epsilon}} \, , \, \frac{1}{k}\biggr\}\biggr].
%\]

The next step of our proof consists of showing that for every $\epsilon > 0$,
\begin{equation}
\label{Eq:Cov1}
\Cov \bigl( \log (\xi_1 f(X_1)), \log f(X_2) \bigr) =O \biggl( \max\biggl\{\frac{k^{-\frac{1}{2}+\frac{2\alpha-\epsilon}{\alpha+d}}}{n^{\frac{2\alpha-\epsilon}{\alpha+d}}} \, , \, \frac{k^{\frac{1}{2}+\frac{\beta}{d}}}{n^{1+\frac{\beta}{d}}}\log^{2+\beta/d} n\biggr\} \biggr)
\end{equation}
as $n \rightarrow \infty$.  Define
\begin{align*}
F_{n,x}^-(u) &:= \sum_{j=k}^{n-2} \binom{n-2}{j}p_{n,x,u}^j(1-p_{n,x,u})^{n-2-j}, \\
\tilde{F}_{n,x}(u) &:= \sum_{j=k-1}^{n-2} \binom{n-2}{j}p_{n,x,u}^j(1-p_{n,x,u})^{n-2-j},
\end{align*}
so that 
\[
\mathbb{P}(\xi_1 \leq u|X_1=x,X_2=y) = \left\{ \begin{array}{ll} F_{n,x}^-(u) & \mbox{if $\|x-y\| > r_{n,u}$} \\
\tilde{F}_{n,x}(u) & \mbox{if $\|x-y\| \leq r_{n,u}$.} \end{array} \right.
\]
Writing $\tilde{u}_{n,x,y} := V_d(n-1)\|x-y\|^de^{-\Psi(k)}$, we therefore have that
\begin{align}
\label{Eq:Cov}
\Cov \bigl(& \log (\xi_1 f(X_1)), \log f(X_2) \bigr) \nonumber \\
&= \int_{\mathcal{X} \times \mathcal{X}} f(x)f(y) \log f(y) \int_{\tilde{u}_{n,x,y}}^\infty \log\bigl(uf(x)\bigr) \, d(\tilde{F}_{n,x} - F_{n,x}^-)(u) \, dx \, dy  \nonumber \\
&\hspace{2cm}- H(f)\int_{\mathcal{X}} f(x) \int_0^\infty \log\bigl(uf(x)\bigr) \, d(F_{n,x}^- - F_{n,x})(u) \, dx.
\end{align}
To deal with the first term in~\eqref{Eq:Cov}, we make the substitution 
\begin{equation}
\label{Eq:yxz}
y = y_{x,z} := x + \frac{r_{n,1}}{f(x)^{1/d}}z,
\end{equation}
and let $d_n := (24 \log n)^{1/d}$.  Writing $T_1,T_2,T_3$ for remainder terms to be bounded later, for every $\epsilon > 0$ and for $k \geq 2$,
\begin{align}
\label{Eq:Main1}
	&\int_{\mathcal{X} \times \mathcal{X}} f(x)f(y) \log f(y) \int_{\tilde{u}_{n,x,y}}^\infty \log (uf(x)) \,d(\tilde{F}_{n,x}-F^-_{n,x})(u) \,dy \,dx \nonumber \\
	& =\!r_{n,1}^d \! \int_{ \mathcal{X}_n}\! \int_{B_0(d_n)} \!\!\!\!\!\!f(y_{x,z}) \log f(y_{x,z})\! \int_{\frac{ \|z\|^d}{f(x)}}^\infty \! \log (uf(x)) \,d(\tilde{F}_{n,x}-F^-_{n,x})(u) dzdx \!+\! T_1 \nonumber \\
	& =\!r_{n,1}^d \!\int_{ \mathcal{X}_n} \!\!\!\!f(x) \log f(x)\! \int_{B_0(d_n)} \! \int_{\frac{ \|z\|^d}{f(x)}}^\infty \! \log (uf(x)) \,d(\tilde{F}_{n,x}-F^-_{n,x})(u)dzdx \!+ \!T_1 \!+\! T_2 \nonumber \\
	&=\! \frac{k-1}{n\!-\!k\!-\!1}\! \int_{ \mathcal{X}_n} \!\!\!\! f(x) \log f(x) dx \!\! \int_0^{\frac{a_n}{n-1}} \!\! \log \Bigl( \frac{(n\!-\!1)s}{e^{\Psi(k)}} \Bigr) \! \mathrm{B}_{k,n-k-1}(s) \Bigl( 1 \!-\! \frac{(n\!-\!2)s}{k-1} \Bigr) ds \nonumber \\
	& \hspace{10cm}+ \sum_{i=1}^3 T_i \nonumber \\
%	& = \frac{k-1}{n-k-1} \int_{ \mathcal{X}_n} f(x) \log f(x) \, dx \bigl[ -\frac{n+2k-2}{(n-1)(n-2)} \{ \log(n-1) - \Psi(n-1) \} - \frac{n(n-k-1)}{(n-1)^2(n-2)} \bigr] + \sum_{i=1}^3 T_i \\
	& = \frac{H(f)}{n} + O(n^{-2}) +o\biggl(\frac{k^{\frac{\alpha}{\alpha+d}-\epsilon}}{n^{1+\frac{\alpha}{\alpha+d}-\epsilon}}\biggr) + \sum_{i=1}^3 T_i.
\end{align}
In Appendix~\ref{Appendix:T}, we show that for every $\epsilon > 0$,
\begin{equation}
\label{Eq:Tbound}
\sum_{i=1}^3 |T_i| = O \biggl( \max\biggl\{\frac{k^{-\frac{1}{2}+\frac{2\alpha}{\alpha+d} - \epsilon}}{n^{\frac{2\alpha}{\alpha+d} - \epsilon}} \, , \, \frac{k^{\frac{1}{2}+\frac{\beta}{d}}}{n^{1+\frac{\beta}{d}}}\log^{2+\beta/d}n\biggr\} \biggr)
\end{equation}
as $n \rightarrow \infty$.  We now deal with the second term in~\eqref{Eq:Cov}.  Writing $U_1, U_2$ for remainder terms to be bounded later, for every $\epsilon > 0$,
\begin{align}
\label{Eq:2ndterm}
\int_{\mathcal{X}} & f(x) \int_0^\infty \log\bigl(uf(x)\bigr) \, d(F_{n,x}^- - F_{n,x})(u) \, dx \nonumber \\
%&= \int_{\mathcal{X}_n} f(x) \int_0^{u_n^*(x)} \log\bigl(uf(x)\bigr) \, d(F_{n,x}^- - F_{n,x})(u) \, dx + U_1 \nonumber \\
&= \int_{\mathcal{X}_n} \!\!\!f(x) \!\! \int_0^\frac{a_n}{n-1} \log (u_{x,s}f(x)) \mathrm{B}_{k,n-k-1}(s)\Bigl\{\frac{(n-1)s - k}{n-k-1}\Bigr\} \, ds \, dx + U_1 \nonumber \\
&= \int_{\mathcal{X}_n} \!\!\!f(x) \!\!\int_0^1 \log\Bigl(\frac{(n-1)s}{e^{\Psi(k)}}\Bigr) \mathrm{B}_{k,n-k-1}(s)\Bigl\{\frac{(n-1)s - k}{n-k-1}\Bigr\} \, ds \, dx + U_1 + U_2 \nonumber \\
&= \frac{1}{n-1} + U_1 + U_2 + o\biggl(\frac{k^{\frac{\alpha}{\alpha+d} - \epsilon}}{n^{1 + \frac{\alpha}{\alpha+d} - \epsilon}}\biggr).
\end{align}
In Appendix~\ref{Appendix:U}, we show that for every $\epsilon > 0$, 
\begin{equation}
\label{Eq:Ubound}
|U_1|+|U_2| = O\biggl(\frac{k^{1/2}}{n}\max\biggl\{\frac{k^{\beta/d}}{n^{\beta/d}} \, , \, \frac{k^{\frac{\alpha}{\alpha+d} - \epsilon}}{n^{\frac{\alpha}{\alpha+d} - \epsilon}}\biggr\}\biggr).
\end{equation}
From~\eqref{Eq:Cov},~\eqref{Eq:Main1},~\eqref{Eq:Tbound},~\eqref{Eq:2ndterm} and~\eqref{Eq:Ubound}, we conclude that~\eqref{Eq:Cov1} holds.

By~\eqref{Eq:ThreeMainTerms}, it remains to consider $\Cov\bigl( \log(\xi_1 f(X_1)), \log( \xi_2 f(X_2))\bigr)$. We require some further notation.  Let $F_{n,x,y}$ denote the conditional distribution function of $(\xi_1,\xi_2)$ given $X_1=x, X_2=y$.  Let $a_n^- := (k - 3k^{1/2}\log^{1/2} n) \vee 0$, $a_n^+ := (k + 3k^{1/2}\log^{1/2} n) \wedge (n-1)$, and let
\[
	v_x := \inf \{ u \geq 0 : (n-1)p_{n,x,u}=a_n^+ \}, \quad l_x := \inf \{ u \geq 0 : (n-1)p_{n,x,u}=a_n^-\},
\]
so that $\mathbb{P}\{\xi_1 \leq l_{X_1}\} = o(n^{-(9/2-\epsilon)})$ and $\mathbb{P}\{\xi_1 \geq v_{X_1}\} = o(n^{-(9/2-\epsilon)})$ for every $\epsilon > 0$.  For pairs $(u,v)$ with $u \leq v_x$ and $v \leq v_y$, let $(M_1,M_2,M_3) \sim \mathrm{Multi}(n-2; p_{n,x,u}, p_{n,y,v}, 1-p_{n,x,u}-p_{n,y,v})$, and write
\begin{align*}
%\label{Eq:Multi}
G_{n,x,y}(u,v) &:= \mathbb{P}( M_1 \geq k, M_2 \geq k),
%\tilde{G}_{n,x,y}(u,v) &:= \mathbb{P}( M_1 \geq k - \mathbbm{1}_{\{y \in B_x(r_{n,u})\}}, M_2 \geq k - \mathbbm{1}_{\{x \in B_y(r_{n,v})\}}).
\end{align*}
so that $F_{n,x,y}(u,v) = G_{n,x,y}(u,v)$ for $\|x-y\| > r_{n,u} + r_{n,v}$.  Write
\[
	\Sigma := \begin{pmatrix} 1 & \alpha_z \\ \alpha_z & 1 \end{pmatrix}
\]
with $\alpha_z := V_d^{-1} \mu_d\bigl(B_0(1) \cap B_z(1)\bigr)$ for $z \in \mathbb{R}^d$, let $\Phi_\Sigma(s,t)$ denote the distribution function of a $N_2(0,\Sigma)$ random vector at $(s,t)$, and let $\Phi$ denote the standard univariate normal distribution function.  Writing $W_i$ for remainder terms to be bounded later, and writing $h(u,v) := \log(uf(x))\log(vf(y))$ as shorthand, we have
\begin{align}
\label{Eq:LongerDisplay}
	 &\Cov( \log(\xi_1 f(X_1)), \log( \xi_2 f(X_2))) \nonumber \\
	& = \int_{\mathcal{X} \times \mathcal{X}} f(x)f(y) \int_0^\infty \int_0^\infty  h(u,v) \,d(F_{n,x,y} -F_{n,x}F_{n,y})(u,v) \, dx \, dy \nonumber \\
	& = \int_{\mathcal{X} \times \mathcal{X}} f(x)f(y)  \int_{[l_x,v_x] \times [l_y,v_y]}  h(u,v) \,d(F_{n,x,y} - F_{n,x}F_{n,y})(u,v) \, dx \, dy + W_1 \nonumber \\
	& = \int_{\mathcal{X} \times \mathcal{X}} \! f(x)f(y)  \int_{[l_x,v_x] \times [l_y,v_y]} \! \! \! \! \! \! \! \! \! h(u,v) \,d(F_{n,x,y} \!-\!G_{n,x,y})(u,v) dx  dy - \frac{1}{n} + \sum_{i=1}^2 W_i \nonumber \\
	& = \int_{\mathcal{X}_n \times \mathcal{X}}\!\!\!\! f(x)f(y) \int_{l_x}^{v_x}\!\!\! \int_{l_y}^{v_y} \frac{(F_{n,x,y}-G_{n,x,y})(u,v)}{uv} \,du \,dv \,dx \, dy - \frac{1}{n} + \sum_{i=1}^3 W_i \nonumber \\
	& = \frac{r_{n,1}^d}{k} \int_{B_0(2)} \int_{-\infty}^\infty \int_{-\infty}^\infty \{ \Phi_\Sigma(s,t) - \Phi(s) \Phi(t) \} \,ds \,dt \,dz  - \frac{1}{n} + \sum_{i=1}^4 W_i \nonumber \\
	& = \frac{e^{\Psi(k)}}{k(n-1)} - \frac{1}{n} + \sum_{i=1}^4 W_i = O \biggl( \frac{1}{nk} \biggr) + \sum_{i=1}^4 W_i.
\end{align}
The proof in the unweighted case is completed by showing in Appendix~\ref{Appendix:W} that for every $\epsilon > 0$,
\begin{align*}
\sum_{i=1}^4 &|W_i| \\
&= O \biggl( \max \biggl\{ \frac{\log^\frac{5}{2} n}{n k^\frac{1}{2}}, \frac{k^{\frac{3}{2}+\frac{\alpha-\epsilon}{\alpha+d}}}{n^{1+\frac{\alpha-\epsilon}{\alpha+d}}}, \frac{k^{\frac{3}{2}+\frac{2\beta}{d}}}{n^{1+\frac{2\beta}{d}}} , \frac{k^{(1+\frac{d}{2\beta})\frac{\alpha-\epsilon}{\alpha+d}}}{n^{1+\frac{\alpha-\epsilon}{\alpha+d}}},\frac{k^{\frac{1}{2}+\frac{\beta}{d}} \log n}{n^{1+\frac{\beta}{d}}},\frac{k^{\frac{2\alpha-\epsilon}{\alpha+d}}}{n^{\frac{2\alpha-\epsilon}{\alpha+d}}} \biggr\} \biggr) 
\end{align*}
as $n \rightarrow \infty$.

The proof in the weighted case uses similar arguments; details are deferred to Appendix~\ref{Appendix:W}.
\end{proof}

\subsection{Proofs of Theorems~\ref{unifweightedclt} and~\ref{Thm:UnifWeightedCLT}}

\begin{proof}[Proof of Theorem~\ref{unifweightedclt}]
Writing $j_t:=\lfloor tk/d \rfloor$ for $t=1,\ldots,d$ and $d' := \lfloor d/4 \rfloor +1$ for convenience, a sufficient condition for $\mathcal{W}^{(k)} \neq \emptyset$ is that the matrix $A^{(k)} \in \mathbb{R}^{d' \times d'}$ with $(l,t)^{th}$ entry
\[
A_{lt}^{(k)}= \Gamma(j_t)^{-1} \Gamma(j_t+2(l-1)/d)k^{-2(l-1)/d},
\]
is invertible.  This follows because, writing $e_1 := (1, 0, \ldots, 0)^T \in \mathbb{R}^{d'}$ we can then define $w = w^{(k)} \in \mathcal{W}^{(k)}$ by setting
\[
(w_{j_t})_{t=1}^{\lfloor d/4 \rfloor +1} := (A^{(k)})^{-1}e_1
\]
and setting all other entries of $w$ to be zero.  Now define $A \in \mathbb{R}^{d' \times d'}$ to have $(l,t)^{th}$ entry $A_{lt}:=(t/d)^{2(l-1)/d}$.  Since $x^{-a} \Gamma(x)^{-1} \Gamma(x+a) \rightarrow 1$ as $x \rightarrow \infty$ for $a \in \mathbb{R}$, we have $\|A^{(k)}-A\| \rightarrow 0$ as $k \rightarrow \infty$. Now, $A$ is a Vandermonde matrix (depending only on $d$) and as such has determinant
\[
	|A| = \prod_{1 \leq t_1 < t_2 \leq d'} d^{-2/d}(t_2^{2/d}-t_1^{2/d}) >0.
\]
Hence, by the continuity of the determinant and eigenvalues of a matrix, we have that there exists $k_d>0$ such that, for $k \geq k_d$, the matrix $A^{(k)}$ is invertible and
\[
	\|(A^{(k)})^{-1} e_1 \| \leq |\lambda_{\min}(A^{(k)})|^{-1} \leq 2 |\lambda_{\min}(A)|^{-1},
\]
where $\lambda_{\min}(\cdot)$ denotes the eigenvalue of a matrix with smallest absolute value.  It follows that, for each $k \geq k_d$, there exists $w^{(k)} \in \mathcal{W}^{(k)}$ satisfying $\sup_{k \geq k_d} \|w^{(k)}\| < \infty$, as required.  

Now, by Corollary~\ref{Cor:WeightedBias} and the fact that $w \in \mathcal{W}^{(k)}$, we have for $\epsilon > 0$ sufficiently small,
\[
	\mathbb{E}_f(\hat{H}_n^{w})-H(f)= O \biggl( \max \biggl\{ \frac{k^{\frac{\alpha}{\alpha+d} - \epsilon}}{n^{\frac{\alpha}{\alpha+d} - \epsilon}}\, , \, \frac{k^{\frac{2d'}{d}}}{n^{\frac{2d'}{d}}}\, , \, \frac{k^\frac{\beta}{d}}{n^\frac{\beta}{d}} \biggr\} \biggr) = o(n^{-1/2}),
\]
uniformly for $f \in \mathcal{F}_{d,\theta}$, under our conditions on $k_1^*, \alpha$ and $\beta$. By Lemma~\ref{varthm} we have $\Var \hat{H}_n^{w}=n^{-1} V(f) +o(n^{-1})$ uniformly for $f \in \mathcal{F}_{d,\theta}$. Note that by Cauchy--Schwarz, very similar arguments to those used at~\eqref{Eq:LongDisplay} and Lemma~\ref{Lemma:VfBounds} in the supplement we have that, for $j \in \text{supp}(w)$,
\[
\Bigl| \text{Cov}_f \Bigl( \log \bigl( \xi_{(j),1} f(X_1) \bigr) , \log f(X_1) \Bigr) \Bigr| \leq \bigl\{ V(f) \mathbb{E}_f \bigl[ \log^2 \bigl( \xi_{(j),1} f(X_1) \bigr)\bigr] \bigr\}^{1/2} \rightarrow 0
\]
uniformly for $f \in \mathcal{F}_{d,\theta}$. Therefore, also using~\eqref{Eq:Cov1}, we have that
\begin{align*}
	&\text{Var}_f( \hat{H}_n^w - H_n^*) = \text{Var}_f \hat{H}_n^w + 2 \text{Cov}_f ( \hat{H}_n^w , \log f(X_1)) + n^{-1} V(f) \\
	& = \text{Var}_f \hat{H}_n^w - n^{-1} V(f) + \frac{2}{n} \sum_{j=1}^k w_j \text{Cov}_f \Bigl( \log \bigl( \xi_{(j),1} f(X_1) \bigr), \log f(X_1) \Bigr) \\
	& \hspace{2cm} + 2 (1-n^{-1}) \sum_{j=1}^k w_j \text{Cov} \Bigl( \log \bigl( \xi_{(j),2} f(X_2) \bigr), \log f(X_1) \Bigr) = o(n^{-1})
\end{align*}
as $n \rightarrow \infty$, uniformly for $f \in \mathcal{F}_{d,\theta}$. The conclusion~\eqref{Eq:Eff} follows on writing
\[
	\mathbb{E}_f\bigl\{(\hat{H}_n^w - H_n^*)^2\bigr\} = \text{Var}_f( \hat{H}_n^w - H_n^*) + ( \mathbb{E}_f \hat{H}_n^w - H(f))^2,
\]
and the final conclusion is then immediate.
%\textbf{Do we want a result on the case $\beta \in [2, d/2)$?}
%(ii) By Proposition~\ref{weightedbiasprop} again, we have
%\[
%	\mathbb{E}(\hat{H}_n^{w})-H= O \biggl( \frac{k^{\beta/d}}{n^{\beta/d}} \biggr).
%\]
%Similarly, by Proposition~\ref{weightedvarprop}, 
%\[
%\Var \hat{H}_n^{w}=n^{-1} \Var \log f(X_1) + o(n^{-1}) = O\biggl( \frac{k^{2\beta/d}}{n^{2\beta/d}} \biggr).
%\]
%The result follows by an application of Chebychev's inequality.
\end{proof}

\begin{proof}[Proof of Theorem~\ref{Thm:UnifWeightedCLT}]
We have
\begin{align}
\label{Eq:dBL1}
d_{\mathrm{BL}}\Bigl(\mathcal{L}\bigl(n^{1/2}&\{\hat{H}_n^w - H(f)\}\bigr),\mathcal{L}\bigl(n^{1/2}\{H_n^* - H(f)\}\bigr)\bigr)\Bigr) \nonumber \\
&\leq \sup_{h \in \mathcal{H}} \mathbb{E}_f\bigl|h\bigl(n^{1/2}\{\hat{H}_n^w - H(f)\}\bigr) - h\bigl(n^{1/2}\{H_n^* - H(f)\}\bigr)\bigr| \nonumber \\
&\leq n^{1/2}\mathbb{E}_f|\hat{H}_n^w - H_n^*| \leq n^{1/2}\bigl[\mathbb{E}_f\bigl\{(\hat{H}_n^w - H_n^*)^2\bigr\}\bigr]^{1/2}.
\end{align}
Now write $\mathcal{H}^*$ for the class of functions $h:\mathbb{R} \rightarrow \mathbb{R}$ having Lipschitz constant at most 1, and let $Z \sim N\bigl(0,V(f)\bigr)$.  Then by standard properties of the Wasserstein distance \citep[e.g.][p.~424]{GibbsSu2002} and the non-uniform version of the Berry--Esseen theorem \citep[e.g.][Theorem~1]{Paditz1989}, 
\begin{align}
\label{Eq:dBL2}
d_{\mathrm{BL}}\Bigl(\mathcal{L}&\bigl(n^{1/2}\{H_n^* - H(f)\}\bigr)\bigr),N\bigl(0,V(f)\bigr)\Bigr) \nonumber \\
&\leq \sup_{h \in \mathcal{H^*}} \bigl|\mathbb{E}_fh\bigl(n^{1/2}\{H_n^* - H(f)\}\bigr) - \mathbb{E}h(Z)\bigr| \nonumber \\
&= \int_{-\infty}^\infty \Bigl|\mathbb{P}_f\bigl(n^{1/2}\{H_n^* - H(f)\} \leq x\bigr) - \mathbb{P}(Z \leq x)\Bigr| \, dx \leq \frac{78\beta_3(f)}{n^{1/2}V(f)},
\end{align}
where
\[
\beta_3(f) := \mathbb{E}_f\bigl\{\bigl|\log f(X_1) + H(f)\bigr|^3\bigr\} = \int_{\mathcal{X}} f(x)|\log f(x) + H(f)|^3 \, dx.
\]
We conclude from~\eqref{Eq:dBL1} and~\eqref{Eq:dBL2}, together with Theorem~\ref{unifweightedclt} and Lemma~\ref{Lemma:VfBounds} in the online supplement, that
\[
\sup_{k \in \{k_0^*,\ldots,k_1^*\}} \sup_{f \in \mathcal{F}_{d,\theta}} d_{\mathrm{BL}}\Bigl(\mathcal{L}\bigl(n^{1/2}(\hat{H}_n^w - H(f))\bigr),N\bigl(0,V(f)\bigr)\Bigr) \rightarrow 0
%&\leq \sup_{k \in \{k_0^*,\ldots,k_1^*\}} \sup_{f \in \mathcal{F}_{d,\theta}} n^{1/2}\bigl[\mathbb{E}_f\bigl\{(\hat{H}_n^w - H_n^*)^2\bigr\}\bigr]^{1/2} + \frac{78\max(1,\log \gamma)}{n^{1/2}} \rightarrow 0
\]
as $n \rightarrow \infty$, as required.

For the second part of the theorem, set
\[
\epsilon_n = \epsilon_n^w(d,\theta) := \frac{\sup_{k \in \{1,\ldots,k^*\}} \sup_{f \in \mathcal{F}_{d,\theta}} \Bigl(2\mathbb{E}_f\bigl[\{\tilde{V}_n^w - V(f)\}^2\bigr]\Bigr)^{1/3}}{\inf_{f \in \mathcal{F}_{d,\theta}} V(f)^{2/3}},
\]
so that $\epsilon_n \rightarrow 0$, by Lemmas~\ref{Lemma:VfBounds}(ii) and~\ref{Lemma:Vw} in the online supplement.  Then, by two applications of Markov's inequality, for $n$ large enough that $\epsilon_n \leq 1$,
\begin{align*}
 \mathbb{P}_f\biggl(\biggl|\frac{(\hat{V}_n^w)^{1/2}}{V^{1/2}(f)} - 1\biggr| \geq \epsilon_n\biggr) &\leq \mathbb{P}_f\biggl(\biggl|\frac{\tilde{V}_n^w}{V(f)} - 1\biggr| \geq \epsilon_n\biggr) + \mathbb{P}_f(\tilde{V}_n^w \leq 0) \\
&\leq \frac{\mathbb{E}_f\bigl[\{\tilde{V}_n^w - V(f)\}^2\bigr]}{V(f)^2}\biggl(\frac{1}{\epsilon_n^2}+1\biggr) \leq \epsilon_n.
\end{align*}
%\[
%\sup_{k \in \{k_0^*,\ldots,k_1^*\}} \sup_{f \in \mathcal{F}_{d,\theta}} \mathbb{P}_f\biggl(\biggl|\frac{(\hat{V}_n^w)^{1/2}}{V^{1/2}(f)} - 1\biggr| > \epsilon_n\biggr) < \epsilon_n.
%\]
For $n \in \mathbb{N}$ and $L \geq 1$, define $h_{n,L}:\mathbb{R} \rightarrow [0,1]$ by
\[
h_{n,L}(x) := \left\{ \begin{array}{ll} 0 & \mbox{if $|x| > z_{q/2}(1 + \epsilon_n) + 1/L$} \\
L\{z_{q/2}(1 + \epsilon_n) + 1/L - |x|\} & \mbox{if $0 < |x| - z_{q/2}(1 + \epsilon_n) \leq 1/L$} \\
1 & \mbox{if $|x| \leq z_{q/2}(1 + \epsilon_n)$.} \end{array} \right. 
\]
Thus $h_{n,L}$ has Lipschitz constant $L$ and $h_{n,L}(x) \geq \mathbbm{1}_{\{|x| \leq z_{q/2}(1 + \epsilon_n)\}}$.  Then, with $Z \sim N(0,1)$,
\begin{align*}
\mathbb{P}_f\bigl(I_{n,q} &\ni H(f)\bigr) \\
&\leq \mathbb{P}_f\biggl(\frac{n^{1/2}|\hat{H}_n^w - H(f)|}{V^{1/2}(f)} \leq z_{q/2}(1+\epsilon_n)\biggr) + \mathbb{P}_f\biggl(\frac{V^{1/2}(f)}{(\hat{V}_n^w)^{1/2}} \leq \frac{1}{1+\epsilon_n}\biggr) \\
&\leq \mathbb{E}_f h_{n,L}\biggl(\frac{n^{1/2}\{\hat{H}_n^w - H(f)\}}{V^{1/2}(f)}\biggr) + \epsilon_n \\
&\leq \mathbb{E}_f h_{n,L}(Z) + \epsilon_n + L d_{\mathrm{BL}}\biggl(\mathcal{L}\biggl(\frac{n^{1/2}\{\hat{H}_n^w - H(f)\}}{V^{1/2}(f)}\biggr),\mathcal{L}(Z)\biggr) \\
&\leq \mathbb{P}\bigl(|Z| \leq z_{q/2}(1+\epsilon_n) + L^{-1}\bigr)
%1-q + \frac{z_{q/2}\epsilon_n + 1/L}{(2\pi)^{1/2}} 
+ \epsilon_n \\
&\hspace{0.5cm}+ L\max\bigl((1,V^{-1/2}(f)\bigr)d_{\mathrm{BL}}\Bigl(\mathcal{L}\bigl(n^{1/2}(\hat{H}_n^w - H(f))\bigr),N\bigl(0,V(f)\bigr)\Bigr).
\end{align*}
Since $L \geq 1$ was arbitrary, we deduce from the first part of the theorem and Lemma~\ref{Lemma:VfBounds} in the online supplement that
\[
\limsup_{n \rightarrow \infty} \sup_{q \in (0,1)} \sup_{k \in \{k_0^*,\ldots,k_1^*\}} \sup_{f \in \mathcal{F}_{d,\theta}}\mathbb{P}_f\bigl(I_{n,q} \ni H(f)\bigr) - (1-q) \leq \inf_{L \geq 1} \frac{2}{L(2\pi)^{1/2}} = 0.
\]
The lower bound is obtained by a similar argument, omitted for brevity.
\end{proof}

\textbf{Acknowledgements:} We thank the reviewers for constructive feedback on an earlier draft.  The second author is grateful to Sebastian Nowozin for introducing him to this problem, and to G\'erard Biau for helpful discussions.

\clearpage

\begin{frontmatter}

\title{Supplementary material to `Efficient multivariate entropy estimation via $k$-nearest neighbour distances'}
\runtitle{Efficient entropy estimation}
\begin{aug}
\author{\fnms{Thomas} B. \snm{Berrett}},
\author{\fnms{Richard} J. \snm{Samworth}}
\and
\author{\fnms{Ming} \snm{Yuan}}
%\ead[label=u1,url]{http://www.statslab.cam.ac.uk/\~{}rjs57}
%\ead[label=u2,url]{http://www.statslab.cam.ac.uk/\~{}tbb26}
%\ead[label=u3,url]{http://pages.stat.wisc.edu/\~{}myuan/}
%\thankstext{t1}{Research supported by a Ph.D. scholarship from the SIMS fund.}
%\thankstext{t2}{Research supported by an EPSRC Early Career Fellowship and a grant from the Leverhulme Trust.}
%\thankstext{t3}{Research supported by NSF FRG Grant DMS-1265202 and NIH Grant 1-U54AI117924-01.}
\runauthor{T. B. Berrett, R. J. Samworth and M. Yuan}
%\affiliation{University of Cambridge\thanksmark{m1}}
%\affiliation{University of Wisconsin--Madison\thanksmark{m2}}

%\address{Statistical Laboratory \\ Wilberforce Road \\ Cambridge \\ CB3 0WB \\ United Kingdom\\ 
%          \printead{e1}\\\printead{e2}\\ \printead{u1}\\\printead{u2}}
%\address{Department of Statistics \\University of Wisconsin--Madison \\ Medical Sciences Center \\ 1300 University Avenue \\ Madison, WI 53706 \\ United States of America \\ \printead{e3} \\ \printead{u3}}

\end{aug}

\end{frontmatter}

%{\Large \textbf{Supplementary material to `Efficient multivariate entropy estimation via $k$-nearest neighbour distances'}}

\section*{Appendix}
\renewcommand{\thesection}{A}

This is the supplementary material to \citet{BSY2016Main}, hereafter referred to as the main text.

\subsection{Proofs of auxiliary results}
\label{Appendix:Auxiliary}

\begin{proof}[Proof of Proposition~\ref{Prop:Moment}]
Fix $\tau \in \bigl(\frac{d}{\alpha+d},1\bigr]$.  We first claim that given any $\epsilon > 0$, there exists $A_\epsilon > 0$ such that $a(\delta) \leq A_\epsilon\delta^{-\epsilon}$ for all $\delta \in (0,\gamma]$.  To see this, observe that there exists $\delta_0 \in (0,\gamma]$ such that $a(\delta) \leq \delta^{-\epsilon}$ for $\delta \leq \delta_0$.  But then
\[
\sup_{\delta \in (0,\gamma]} \delta^\epsilon a(\delta) \leq \max\bigl\{1,\gamma^\epsilon a(\delta_0) \bigr\} \leq \gamma^\epsilon \delta_0^{-\epsilon},
\]
which establishes the claim, with $A_\epsilon := \gamma^\epsilon \delta_0^{-\epsilon}$.  Now choose $\epsilon = \frac{1}{3}\bigl(\tau - \frac{d}{\alpha+d}\bigr)$ and let $\tau' := \frac{\tau}{3}+\frac{2d}{3(\alpha+d)} \in \bigl(\frac{d}{\alpha+d},1\bigr)$.  Then, by H\"older's inequality, and since $\alpha\tau'/(1-\tau') > d$,
\begin{align*}
\sup_{f \in \mathcal{F}_{d,\theta}} \int_{\{x:f(x) < \delta\}} &a\bigl(f(x)\bigr)f(x)^\tau \, dx %\leq \frac{\gamma^\epsilon}{\delta_0^\epsilon} \sup_{f \in \mathcal{F}_{d,\theta}} \int_{\{x:f(x) < \delta\}} f(x)^{\tau'} \, dx \\
\leq A_\epsilon\delta^\epsilon \sup_{f \in \mathcal{F}_{d,\theta}} \int_{\{x:f(x) < \delta\}} f(x)^{\tau'} \, dx \\
&\leq A_\epsilon\delta^\epsilon(1+\nu)^{\tau'}\biggl\{\int_{\mathbb{R}^d} (1+\|x\|^\alpha)^{-\frac{\tau'}{1-\tau'}} \, dx \biggr\}^{1-\tau'} \rightarrow 0
\end{align*}
as $\delta \searrow 0$, as required.  

For the second part, fix $\rho > 0$, set $\epsilon := \frac{1}{2}\bigl(\tau - \frac{d}{\alpha+d}\bigr)$ and $\tau' := \frac{\tau}{2}+\frac{d}{2(\alpha+d)} \in \bigl(\frac{d}{\alpha+d},1\bigr)$.  Then, by H\"older's inequality again,
\begin{align*}
\sup_{f \in \mathcal{F}_{d,\theta}} \int_{\mathcal{X}} a\bigl(f(x)\bigr)^\rho &f(x)^\tau \, dx \leq A_{\epsilon/\rho} \sup_{f \in \mathcal{F}_{d,\theta}} \int_{\mathcal{X}} f(x)^{\tau'} \, dx \\
&\leq A_{\epsilon/\rho}(1+\nu)^{\tau'}\biggl\{\int_{\mathbb{R}^d} (1+\|x\|^\alpha)^{-\frac{\tau'}{1-\tau'}} \, dx \biggr\}^{1-\tau'} < \infty,
\end{align*}
as required.
\end{proof}
\begin{proof}[Proof of Lemma~\ref{Lemma:hxinvbounds}]
(i) The lower bound is immediate from the fact that $h_x(r) \leq V_d\|f\|_\infty r^d$ for any $r > 0$.  For the upper bound, observe that by Markov's inequality, for any $r > 0$,
\[
h_x(\|x\|+r) = \int_{B_x(\|x\|+r)} f(y) \, dy \geq \int_{B_0(r)} f(y) \, dy \geq 1 - \frac{\mu_\alpha(f)}{r^\alpha}.
\]
The result follows on substituting $r = \bigl(\frac{\mu_\alpha(f)}{1-s}\bigr)^{1/\alpha}$ for $s \in (0,1)$.

(ii) We first prove this result in the case $\beta \in (2,4]$, giving the stated form of $b_1(\cdot)$.  Let $C := 4dV_d^{-\beta/d}/(d+\beta)$, and let $y:= C a(f(x))^{\beta/2} s \{s/f(x)\}^{\beta/d}$. Now, by the mean value theorem, we have for $r \leq r_a(x)$ that
\begin{align*}
	\biggl| h_x(r) - V_d r^d f(x) - \frac{V_d}{2(d+2)} r^{d+2} \Delta f(x) \biggr| \leq a(f(x)) f(x) \frac{dV_d}{2(d+\beta)} r^{d+\beta}.
\end{align*}
It is convenient to write
\[
s_{x,y} := s- \frac{s^{1+2/d} \Delta f(x)}{2(d+2)V_d^{2/d} f(x)^{1+2/d}} + y.
\]
Then, provided $s_{x,y} \in (0,V_dr_a^d(x)f(x)]$, we have
\begin{align*}
	h_x \biggl( &\frac{s_{x,y}^{1/d}}{\{V_d f(x) \}^{1/d}}\biggr)  \\
	& \hspace{20pt} \geq s_{x,y} + \frac{ V_d^{-2/d} \Delta f(x)}{2(d+2) f(x)^{1+2/d} } s_{x,y}^{1+2/d} - \frac{a(f(x))dV_d^{-\beta/d}}{2(d+\beta)f(x)^{\beta/d}} s_{x,y}^{1+\beta/d}.
\end{align*}
Now, by our hypothesis, we know that 
\begin{align*}
\sup_{f \in \mathcal{F}_{d,\theta}} \sup_{s \in \mathcal{S}_n} \sup_{x \in \mathcal{X}_n} \max \biggl\{ &\frac{V_d^{-2/d}s^{2/d} |\Delta f(x)|}{2(d+2) f(x)^{1+2/d}}, \frac{y}{s} \biggr\} \\
&\leq \max \biggl\{ \frac{d^{1/2}V_d^{-2/d} C_n^{2/d}}{2(d+2)}, C C_n^{\beta/d} \biggr\} \rightarrow 0
\end{align*}
as $n \rightarrow \infty$.  Hence there exists $n_1 = n_1(d,\theta) \in \mathbb{N}$ such that for all $n \geq n_1$, all $f \in \mathcal{F}_{d,\theta}$, $s \in \mathcal{S}_n$ and $x \in \mathcal{X}_n$, we have
\begin{align*}
	\frac{1}{2(d+2)}(s_{x,y}^{1+2/d} - s^{1+2/d}) \geq - \frac{s^{1+2/d}}{2d} \biggl\{ \frac{d^{1/2}V_d^{-2/d}a(f(x))s^{2/d}}{2(d+2)f(x)^{2/d}} + \frac{y}{s} \biggr\}.
\end{align*}
Moreover, there exists $n_2 = n_2(d,\theta) \in \mathbb{N}$ such that for all $n \geq n_2$, all $s \in \mathcal{S}_n$, $x \in \mathcal{X}_n$ and $f \in \mathcal{F}_{d,\theta}$ we have
\begin{align*}
	|s_{x,y}|^{1+\beta/d} \leq 2s^{1+\beta/d}.
\end{align*}
Finally, we can choose $n_3 = n_3(d,\theta) \in \mathbb{N}$ such that
\[
\max\biggl\{\frac{C_n^{(4-\beta)/d}}{4(d+2)V_d^{(4-\beta)/d}}\, , \, \frac{2d^{1/2}C_n^{2/d}}{(d+\beta)V_d^{2/d}}\, , \, \frac{d^{3/2}C_n^{2/d}}{2(d+2)(d+\beta)V_d^{2/d}}\biggr\} \leq \frac{d}{d+\beta}
\]
and such that $C_n \leq (8d^{1/2})^{-d}V_d/2$ for $n \geq n_3$.  It follows that for $n \geq \max(n_1,n_2,n_3) =: n_*$, for $f \in \mathcal{F}_{d,\theta}$, $s \in \mathcal{S}_n$ and for $x \in \mathcal{X}_n$, we have that $s_{x,y} \in (0,V_dr_a^d(x)f(x)]$ and
\begin{align*}
&h_x \biggl( \frac{s_{x,y}^{1/d}}{\{V_d f(x) \}^{1/d}}\biggr) -s\\
&\geq y - \frac{a(f(x))s^{1+2/d}}{2d^{1/2}V_d^{2/d}f(x)^{2/d}} \biggl\{ \frac{d^{1/2}V_d^{-2/d} a(f(x))s^{2/d}}{2(d+2)f(x)^{2/d}}+\frac{y}{s} \biggr\} - \frac{d a(f(x))s^{1+\beta/d}}{(d+\beta)V_d^{\frac{\beta}{d}}f(x)^{\frac{\beta}{d}}} \\
	&\geq \frac{a(f(x))^{\beta/2}s^{1+\beta/d}}{f(x)^{\beta/d}}\biggl[ C - \frac{a(f(x))^{2-\beta/2}}{4(d+2)V_d^{4/d}} \biggl\{ \frac{s}{f(x)}\biggr\}^{(4-\beta)/d} \\
&\hspace{4cm}- \frac{Ca(f(x))}{2d^{1/2}V_d^{2/d}}\biggl\{ \frac{s}{f(x)}\biggr\}^{2/d} - \frac{dV_d^{-\beta/d}}{d+\beta} \biggr]  \geq 0.
\end{align*}
The lower bound is proved by very similar calculations, and the result for the case $\beta \in (2,4]$ follows.  The general case can be proved using very similar arguments, and is omitted for brevity.
\end{proof}

\subsection{Auxiliary results for the proof of Theorem~\ref{Thm:UnifWeightedCLT}}

Recall the definition of $V(f)$ given in the statement of Theorem~\ref{unifweightedclt}.
\begin{lemma}
\label{Lemma:VfBounds}
For each $d \in \mathbb{N}$ and $\theta \in \Theta$ and $m \in \mathbb{N}$, we have
\begin{enumerate}[(i)]
\item $\sup_{f \in \mathcal{F}_{d,\theta}} \int_{\mathcal{X}} f(x)|\log^m f(x)| \, dx < \infty$;
\item $\inf_{f \in \mathcal{F}_{d,\theta}} V(f) > 0$;
%\item $\sup_{f \in \mathcal{F}_{d,\theta}} V(f) < \infty$;
%\item $\sup_{f \in \mathcal{F}_{d,\theta}} \beta_3(f) < \infty$.
\end{enumerate}
\end{lemma}

\begin{proof}[Proof of Lemma~\ref{Lemma:VfBounds}]
Fix $d \in \mathbb{N}$ and $\theta = (\alpha,\beta,\gamma,\nu,a) \in \Theta$.

(i) For $\epsilon \in (0,1)$ and $t \in (0,1]$, we have
\[
	\log \frac{1}{t} \leq \frac{1}{\epsilon} t^{-\epsilon}.
\]
Let $\epsilon = \frac{\alpha}{m(\alpha+2d)}$, so that $\frac{\alpha(1-m\epsilon)}{m\epsilon} = 2d$.  Then, by H\"{o}lder's inequality, for any $f \in \mathcal{F}_{d,\theta}$,
\begin{align*}
\int_{\mathcal{X}} f(x)|\log^m f(x)| \, dx &\leq 2^{m-1}\int_{\mathcal{X}} f(x) \log^m\Bigl(\frac{\|f\|_\infty}{f(x)}\Bigr) \, dx + 2^{m-1}|\log^m \|f\|_\infty| \\
&\leq \frac{2^{m-1}\|f\|_\infty^{m\epsilon}}{\epsilon^m} \int_{\mathcal{X}} f(x)^{1-m\epsilon} \, dx + 2^{m-1}|\log^m \|f\|_\infty| \\
	& \leq \frac{2^{m-1}\gamma^{m\epsilon}}{\epsilon^m} (1+\nu)^{1-m\epsilon} \Bigl\{ \int_{\mathcal{X}} (1+\|x\|^\alpha)^{-\frac{1-m\epsilon}{m\epsilon}} \,dx \Bigr\}^{m\epsilon} \\
&+ 2^{m-1}\max\biggl\{\log^m \gamma \, , \, \frac{1}{\alpha^m}\log^m\biggl(\frac{V_d^\alpha\nu^d(\alpha+d)^{\alpha+d}}{\alpha^\alpha d^d}\biggr)\biggr\},
\end{align*}
where the bound on $\bigl|\log^m \|f\|_\infty\bigr|$ comes from~\eqref{Eq:LotsofTerms} in the main text.

(ii) Now define
\[
	A_{d,\theta} := \max \biggl\{ \sup_{f \in \mathcal{F}_{d,\theta}} |H(f)| \, , \, -\frac{1}{2} \log \inf_{f \in \mathcal{F}_{d,\theta}} \|f\|_\infty \, , \, 1\biggr\}
\]
and the set $S_{d,\theta}:= \{ x \in \mathcal{X} : e^{-4A_{d,\theta}} \leq f(x) \leq e^{-2A_{d,\theta}} \}$. For $f \in \mathcal{F}_{d,\theta}, x \in S_{d,\theta}$ and $y \in B_x( \{8d^{1/2}a(e^{-4A_{d,\theta}}) \}^{-1/(\beta \wedge 1)})$ we have by Lemma~\ref{Lemma:15over7} below that
\begin{equation}
\label{Eq:Lipschitz}
	|f(y)-f(x)| \leq \frac{15d^{1/2}}{7} a(e^{-4A_{d,\theta}}) e^{-2 A_{d,\theta}} \|y-x\|^{\beta \wedge 1}.
\end{equation}
By the continuity of $f$, there exists $x_0 \in S_{d,\theta}$ such that $f(x_0) = \frac{1}{2} e^{-2A_{d,\theta}}(1+e^{-2A_{d,\theta}})$. Thus, by~\eqref{Eq:Lipschitz}, we have that $B_{x_0}(r_{d,\theta}) \subseteq S_{d,\theta}$, where
\[
	r_{d,\theta} := \Bigl\{ \frac{7(1-e^{-2A_{d,\theta}})}{30 d^{1/2}a(e^{-4A_{d,\theta}})} \Bigr\}^{1/(\beta \wedge 1)} \wedge \frac{1}{8d^{1/2}a(e^{-4A_{d,\theta}})\}^{1/(\beta \wedge 1)}}.
\]
Hence
\begin{align*}
	V(f) = \mathbb{E}_f [ \{ \log f(X_1) +H(f) \}^2] \geq A_{d,\theta}^2 \mathbb{P}_f ( X_1 \in S_{d,\theta}) \geq A_{d,\theta}^2 e^{-4A_{d,\theta}} V_d r_{d,\theta}^d,
\end{align*}
as required.
\end{proof}
The following auxiliary result  provides control on deviations of the density arising from the smoothness condition of our $\mathcal{F}_{d,\theta}$ classes.
\begin{lemma}
\label{Lemma:15over7}
For $\theta = (\alpha,\beta,\gamma,\nu,a) \in \Theta$, $m:= \lceil \beta \rceil - 1$, $f \in \mathcal{F}_{d,\theta}$ and $y \in B_x\bigl( r_a(x)\bigr)$, we have, for multi-indices $t$ with $|t| \leq m$, that 
\[
\Bigl| \frac{\partial f^t(y)}{\partial x^t} - \frac{\partial f^t(x)}{\partial x^t} \Bigr| \leq \frac{15d^{1/2}}{7} a\bigl(f(x)\bigr)f(x)\|y-x\|^{\min(\beta-|t|,1)}.
\]
\end{lemma}

\begin{proof}[Proof of Lemma~\ref{Lemma:15over7}]
If $|t|=m$ then the result follows immediately from the definition of $\mathcal{F}_{d,\theta}$. Henceforth, therefore, assume that $m \geq 1$ and $|t| \leq m-1$.  Writing $\vertiii{\cdot}$ here for the largest absolute entry of an array, we have for $y \in B_x\bigl( r_a(x)\bigr)$ that
\begin{align*}
	&\Bigl| \frac{\partial f^t(y)}{\partial x^t} - \frac{\partial f^t(x)}{\partial x^t} \Bigr| \leq \|y-x\| \sup_{z \in B_x(\|y-x\|)} \Bigl\|\nabla \frac{\partial f^t(z)}{\partial x^t} \Bigr\| \\
%	& \leq \|y-x\| \Bigl\|\nabla \frac{\partial f^t(x)}{\partial x^t} \Bigr\| + \|y-x\| \sup_{z \in B_x(\|y-x\|)} \Bigl\|\nabla \frac{\partial f^t(z)}{\partial x^t} -\nabla \frac{\partial f^t(x)}{\partial x^t} \Bigr\| \\
	& \leq \|y-x\| \|f^{(|t|+1)}(x) \| + d^{1/2} \|y-x\| \sup_{z \in B_x(\|y-x\|)} \vertiii{f^{(|t|+1)}(z) - f^{(|t|+1)}(x)} \\
	& \leq  \sum_{\ell=1}^{m-|t|} d^{(\ell-1)/2} \|y-x\|^\ell \|f^{(|t|+\ell)}(x) \| \\
&\hspace{4cm}+ d^{m/2} \|y-x\|^m \sup_{z \in B_x(\|y-x\|)} \vertiii{f^{(m)}(z)-f^{(m)}(x)} \\
	& \leq a(f(x)) f(x) \|y-x\| \biggl\{ \frac{1}{1-d^{1/2}\|y-x\|} + d^{m/2} \|y-x\|^{\beta-1} \biggr\} \\
	& \leq \frac{15 d^{1/2}}{7} a(f(x)) f(x) \|y-x\|,
\end{align*}
as required.
\end{proof}

\begin{lemma}
\label{Lemma:Vw}
Under the conditions of Theorem~\ref{unifweightedclt} in the main text, we have that
\[
	\sup_{k \in \{k_0^*, \ldots, k_1^*\}} \sup_{f \in \mathcal{F}_{d,\theta}} \mathbb{E}_f [ \{\tilde{V}_n^w - V(f)\}^2] \rightarrow 0.
\]
\end{lemma}

\begin{proof}[Proof of Lemma~\ref{Lemma:Vw}]
For $w = (w_1,\ldots,w_k)^T \in \mathcal{W}^{(k)}$, write $\mathrm{supp}(w) := \{j : w_j \neq 0\}$.  Then
\begin{align*}
	&|\mathbb{E}_f \tilde{V}_n^w - V(f) | \\
	&\leq \biggl|\sum_{j=1}^k w_j \mathbb{E}_f \log^2 \xi_{(j),1}- \int_{\mathcal{X}} f \log^2f \biggr| + \bigl|\mathbb{E}_f\{(\hat{H}_n^w)^2\} - H(f)^2\bigr| \\
	& \leq \|w\|_1 \! \! \max_{j \in \text{supp}(w)} \biggl| \mathbb{E}_f \log^2 \xi_{(j),1} \! - \! \int_{\mathcal{X}} \! f \log^2f \biggr| + \mathrm{Var}_f \hat{H}_n^w + | ( \mathbb{E}_f \hat{H}_n^w )^2 \! - \! H(f)^2 |. 
\end{align*}
Thus, by Theorem~\ref{unifweightedclt} in the main text,~\eqref{Eq:LongDisplay} in the proof of that result and Lemma~\ref{Lemma:VfBounds}(i), we have that $\sup_{k \in \{k_0^*, \ldots, k_1^*\}} \sup_{f \in \mathcal{F}_{d,\theta}} |\mathbb{E}_f \tilde{V}_n^w - V(f) | \rightarrow 0$. Now, 
\begin{align}
\label{Eq:VarVw}
	\mathrm{Var}_f \tilde{V}_n^w \leq \frac{\|w\|_1^2}{n} \max_{j \in \text{supp}(w)} &\mathrm{Var}_f \log^2 \xi_{(j),1} \nonumber \\
	&+ \|w\|_1^2 \max_{j,\ell \in \text{supp}(w)} \bigl| \mathrm{Cov}_f(\log^2 \xi_{(j),1} , \log^2 \xi_{(\ell),2}) \bigr|.
\end{align}
Let $a_{n,j}^- := (j - 3j^{1/2} \log^{1/2}n) \vee 0$ and $a_{n,j}^+ := (j + 3j^{1/2} \log^{1/2}n) \wedge (n-1)$.  Mimicking arguments in the proof of Theorem~\ref{unifweightedclt}, for any $m \in \mathbb{N}$, $j \in \text{supp}(w)$ and $\epsilon > 0$,
\begin{align*}
%\label{moments}
	&\mathbb{E}_f\bigl\{ \log^m (\xi_{(j),1} f(X_1))\bigr\} \\
&\hspace{2cm}= \int_\mathcal{X} f(x) \int_0^\infty \log^m \biggl( \frac{V_d(n-1)f(x)h_x^{-1}(s)^d}{e^{\Psi(j)}} \biggr) \mathrm{B}_{j,n-j}(s) \,ds \,dx \nonumber \\
	&\hspace{2cm}= \int_\frac{a_{n,j}^-}{n-1}^\frac{a_{n,j}^+}{n-1} \log^m \biggl( \frac{(n-1)s}{e^{\Psi(j)}} \biggr) \mathrm{B}_{j,n-j}(s) \,ds \\
&\hspace{4cm} + O \biggl( \max \biggl\{ \frac{k^{\beta/d}}{n^{\beta/d}}\log^{m-1}n \, , \, \frac{k^{\frac{\alpha}{\alpha+d}-\epsilon}}{n^{\frac{\alpha}{\alpha+d}-\epsilon}} \biggr\} \biggr) \rightarrow 0,
\end{align*}
uniformly for $j \in \text{supp}(w)$, $k \in \{k_0^*, \ldots, k_1^*\}$ and $f \in \mathcal{F}_{d,\theta}$.  Moreover, by Cauchy--Schwarz, we can now show, for example, that
\begin{align*}
	\mathbb{E}_f \log^4 \xi_{(j),1} = \mathbb{E}_f[ \{ \log (\xi_{(j),1} f(X_1)) - \log f(X_1) \}^4] \rightarrow \mathbb{E}_f \log^4 f(X_1)
\end{align*}
uniformly for $j \in \text{supp}(w)$, $k \in \{k_0^*, \ldots, k_1^*\}$ and $f \in \mathcal{F}_{d,\theta}$. The result follows upon noting that we may use a similar approach for the covariance term in~\eqref{Eq:VarVw} to see that $\sup_{k \in \{k_0^*, \ldots, k_1^*\}} \sup_{f \in \mathcal{F}_{d,\theta}} \Var \tilde{V}_n^w \rightarrow 0$.
\end{proof}

\subsection{Proof of Proposition~\ref{Prop:Examples}}

\begin{proof}[Proof of Proposition~\ref{Prop:Examples}]
In each of the three examples, we provide $\theta = (\alpha,\beta,\gamma,\nu,a) \in \Theta$ such that $f \in \mathcal{F}_{d,\theta}$.  In fact, $\beta > 0$ may be chosen arbitrarily in each case.

(i) We may choose any $\alpha > 0$, and then set $\nu = d2^{\alpha/2-1}\Gamma\bigl(\frac{\alpha}{2} + \frac{d}{2}\bigr)/\Gamma\bigl(1+\frac{d}{2}\bigr)$.  We may also set $\gamma = (2\pi)^{-d/2}$.  It remains to find $a \in \mathcal{A}$ such that~\eqref{Eq:GrowthCond} holds.  Write $H_r(y) := (-1)^r e^{y^2/2}\frac{d^r}{dy^r}e^{-y^2/2}$ for the $r$th Hermite polynomial, and note that $|H_r(y)| \leq p_r(|y|)$, where $p_r$ is a polynomial of degree $r$ with non-negative coefficients.  Using multi-index notation for partial derivatives, if $t = (t_1,\ldots,t_d) \in \{0,1,\ldots,\}^d$ with $|t| := t_1 + \ldots + t_d$, we have
\[
\biggl|\frac{\partial f^t(x)}{\partial x^t}\biggr| = f(x)\prod_{j=1}^d |H_{t_j}(x_j)| \leq f(x)\prod_{j=1}^d p_{t_j}(\|x\|) \leq f(x)q_{|t|}(\|x\|),
\]
for some polynomial $q_r$ of degree $r$, with non-negative coefficients.  It follows that if $y \in B_x^\circ(1)$, then for any $\beta > 0$ with $m = \lceil \beta \rceil -1$, 
\begin{align*}
%	\| f_d^{(|t|)}(x)\| \leq q_t(\|x\|)f_d(x),
\frac{\|f^{(m)}(x) - f^{(m)}(y)\|}{f(x)\|y-x\|^{\beta - m}} &\leq \frac{d^{m/2}}{f(x)\|y-x\|^{\beta - m}}\max_{t:|t|=m}\biggl|\frac{\partial f^t(x)}{\partial x^t} - \frac{\partial f^t(y)}{\partial x^t}\biggr| \\
&\leq \frac{d^{(m+1)/2}}{f(x)}\max_{t:|t|=m+1} \sup_{w \in B_0(1)} \biggl|\frac{\partial f^t(x+w)}{\partial x^t}\biggr| \\
&\leq d^{(m+1)/2}\sup_{w \in B_0(1)} \frac{f(x+w)q_{m+1}(\|x+w\|)}{f(x)} \\
&\leq d^{(m+1)/2}e^{\|x\|} q_{m+1}(\|x\|+1).
\end{align*}
Similarly, 
\[
\max_{r=1,\ldots,m} \frac{\|f^{(r)}(x)\|}{f(x)} \leq d^{m/2}\max_{r=1,\ldots,m} q_r(\|x\|).
\]
Write $g(\delta) := \bigl\{-2 \log \bigl(\delta (2 \pi)^{d/2}\bigr)\bigr\}^{1/2}$ and define $a \in \mathcal{A}$ by setting $a(\delta) := \max\{1,\tilde{a}(\delta)\}$, where
\begin{align*}
\tilde{a}(\delta) &:= d^{m/2}\sup_{x:\|x\| \leq g(\delta)} \max\biggl\{\max_{r=1,\ldots,m} q_r(\|x\|) \, , \, d^{1/2}e^{\|x\|} q_{m+1}(\|x\|+1)\biggr\} \\
&= d^{m/2}\max\biggl\{\max_{r=1,\ldots,m} q_r\bigl(g(\delta)\bigr) \, , \, d^{1/2}e^{g(\delta)} q_{m+1}\bigl(g(\delta)+1\bigr)\biggr\}.
\end{align*}
Then $\sup_{x:f(x) \geq \delta} M_{f,a,\beta}(x) \leq a(\delta)$ and $a(\delta) = o(\delta^{-\epsilon})$ for every $\epsilon > 0$, so~\eqref{Eq:GrowthCond} holds.

(ii) We may choose any $\alpha < \rho$, and set 
\[
\nu = d2^{\alpha/2-1}\frac{\Gamma\bigl(\frac{\alpha}{2} + \frac{d}{2}\bigr)}{\Gamma\bigl(1+\frac{d}{2}\bigr)}\frac{(\rho/2)^{\alpha/2}\Gamma\bigl(\frac{\rho-\alpha}{2}\bigr)}{\Gamma\bigl(\frac{\rho}{2}\bigr)}.
\]
We may also set $\gamma = \frac{\Gamma\bigl(\frac{\rho}{2} + \frac{d}{2}\bigr)}{\Gamma(\rho/2)\rho^{\alpha/2}\pi^{d/2}}$.  To verify~\eqref{Eq:GrowthCond} for suitable $a \in \mathcal{A}$, we note by induction, that if $t = (t_1,\ldots,t_d) \in \{0,1,\ldots,\}^d$ with $|t| := t_1 + \ldots + t_d$, then
\[
\biggl|\frac{\partial f^t(x)}{\partial x^t}\biggr| \leq \frac{f(x)q_{|t|}(\|x\|)}{(1+\|x\|^2/\rho)^{|t|}},
\]
where $q_r$ is a polynomial of degree $r$ with non-negative coefficients.  Thus, similarly to the Gaussian example, for any $\beta > 0$ with $m = \lceil \beta \rceil -1$, 
\begin{align*}
\sup_{x \in \mathbb{R}^d} &\sup_{y \in B_x^\circ(1)} \frac{\|f^{(m)}(x) - f^{(m)}(y)\|}{f(x)\|y-x\|^{\beta - m}} \\
&\hspace{2cm}\leq d^{(m+1)/2}\sup_{x \in \mathbb{R}^d} \sup_{w \in B_0(1)} \frac{f(x+w)q_{m+1}(\|x+w\|)}{f(x)(1+\|x\|^2/\rho)^{m+1}} =: A_{d,m,\rho}^{(1)},
\end{align*}
say, where $A_{d,m,\rho}^{(1)} \in [0,\infty)$.  Similarly,
\[
\sup_{x \in \mathbb{R}^d} \max_{r=1,\ldots,m} \frac{\|f^{(r)}(x)\|}{f(x)} \leq d^{m/2}\sup_{x \in \mathbb{R}^d} \max_{r=1,\ldots,m} \frac{q_r(\|x\|)}{(1+\|x\|^2/\rho)^r} =: A_{d,m,\rho}^{(2)},
\]
say, where $A_{d,m,\rho}^{(2)} \in [0,\infty)$.  Now defining $a \in \mathcal{A}$ to be the constant function
\[
a(\delta) := \max\{1,A_{d,m,\rho}^{(1)},A_{d,m,\rho}^{(2)}\},
\]
we again have that $\sup_{x:f(x) \geq \delta} M_{f,a,\beta}(x) \leq a(\delta)$, so~\eqref{Eq:GrowthCond} holds.

(iii) We may take any $\alpha > 0$ and $\nu = 1$, $\gamma = 3$.  To verify~\eqref{Eq:GrowthCond}, fix $\beta > 0$, set $m := \lceil \beta \rceil - 1$, and define $a \in \mathcal{A}$ by
\[
a(\delta) := A_m\max\biggl\{1 \, , \, \log^{2(m+1)}\Bigl(\frac{1}{\delta}\Bigr)\biggr\},
\]
for some $A_m \geq 1$ depending only on $m$.  Then, by induction, we find that for some constants $A_m',B_m' > 0$ depending only on~$m$, and $x \in (-1,1)$
\begin{align*}
M_{f,a,\beta}(x) &\leq \max\biggl\{\max_{r=1,\ldots,m}\frac{A_r'}{(1-x^2)^{2r}} \, , \, \sup_{y: 0 < |y-x| \leq r_a(x)} \frac{A_{m+1}'f(y)}{(1-y^2)^{2(m+1)}f(x)}\biggr\} \\
&\leq \frac{B_{m+1}'}{(1-x^2)^{2(m+1)}} \leq a\bigl(f(x)\bigr),
\end{align*}
provided $A_m$ in the definition of $a$ is chosen sufficiently large.  Hence~\eqref{Eq:GrowthCond} again holds.
\end{proof}

\subsection{Proof of Proposition~\ref{Prop:WeakCond}}

\begin{proof}[Proof of Proposition~\ref{Prop:WeakCond}]
To deal with the integrals over $\mathcal{X}_n^c$, we first observe that by~\eqref{Eq:LotsofTerms} in the main text there exists a constant $C_{d,f} > 0$, depending only on $d$ and $f$, such that
\begin{align}
\label{Eq:Xnc1}
&\int_{\mathcal{X}_n^c} f(x) \int_0^1 \mathrm{B}_{k,n-k}(s) \log u_{x,s} \, ds \, dx \nonumber \\
&\leq C_{d,f} \int_{\mathcal{X}_n^c} f(x)\biggl\{\log n + \log\biggl(1+ \frac{\|x\|}{\mu_\alpha^{1/\alpha}(f)}\biggr)\biggr\} \, dx = O\bigl(\max\{q_n \log n,q_n^{1-\epsilon}\}\bigr), 
\end{align}
for every $\epsilon > 0$.  Moreover, 
\begin{equation}
\label{Eq:Xnc2}
\biggl|\int_{\mathcal{X}_n^c} f(x) \log f(x) \, dx\biggr| = O(q_n^{1-\epsilon}),
\end{equation}
for every $\epsilon>0$.  Now, a slightly simpler argument than that used in the proof of Lemma~\ref{Lemma:hxinvbounds}(ii) in the main text gives that for $r \in (0,r_x]$, we have
\[
|h_x(r) - V_df(x)r^d| \leq \frac{dV_d}{d+\tilde{\beta}} C_{n,\tilde{\beta}}(x)r^{d+\tilde{\beta}}.
\]
We deduce, again using a slightly simplified version of the argument in Lemma~\ref{Lemma:hxinvbounds}(ii) in the main text, that there exists $n_0 \in \mathbb{N}$ such that for $n \geq n_0$, $s \in [0,\frac{a_n}{n-1}]$ and $x \in \mathcal{X}_n$, we have
\begin{equation}
\label{Eq:Simplerhxinv}
\bigl|V_df(x)h_x^{-1}(s)^d - s\bigr| \leq \frac{2dV_d^{-\tilde{\beta}/d}}{d+\tilde{\beta}}s^{1+\tilde{\beta}/d}\frac{C_{n,\tilde{\beta}}(x)}{f(x)^{1+\tilde{\beta}/d}} \leq \frac{s}{2}.
\end{equation}
It follows from~\eqref{Eq:Xnc1},~\eqref{Eq:Xnc2},~\eqref{Eq:Simplerhxinv} and an almost identical argument to that leading to~\eqref{Eq:R2bound} in the main text that for every $n \geq n_0$ and $\epsilon > 0$,
\begin{align*}
&|\mathbb{E}_f(\hat{H}_n) - H| \leq \biggl|\int_{\mathcal{X}_n} f(x) \int_0^\frac{a_n}{n-1} \mathrm{B}_{k,n-k}(s)\log\biggl(\frac{V_df(x)h_x^{-1}(s)^d}{s}\biggr) \,ds\,dx\biggr| \\
&\hspace{6cm}+ O\bigl(\max\{q_n^{1-\epsilon},q_n \log n,n^{-1}\}\bigr) \\
&\leq 2 \int_{\mathcal{X}_n} f(x) \int_0^\frac{a_n}{n-1} \mathrm{B}_{k,n-k}(s)\biggl|\frac{V_df(x)h_x^{-1}(s)^d - s}{s}\biggr| \, ds \, dx \\
&\hspace{6cm}+ O\bigl(\max\{q_n^{1-\epsilon},q_n \log n,n^{-1}\}\bigr) \\
&\leq \frac{4dV_d^{-\tilde{\beta}/d}}{d+\tilde{\beta}}\frac{\mathrm{B}_{k+\tilde{\beta}/d,n-k}}{\mathrm{B}_{k,n-k}}\int_{\mathcal{X}_n} \frac{C_{n,\tilde{\beta}}(x)}{f(x)^{\tilde{\beta}/d}} \, dx + O\bigl(\max\{q_n^{1-\epsilon},q_n \log n,n^{-1}\}\bigr),
\end{align*}
as required.
\end{proof}

\subsection{Completion of the proof of Lemma~\ref{varthm}}
\label{Appendix:varproof}

To prove Lemma~\ref{varthm}, it remains to bound several error terms arising from arguments that approximate the variance of the unweighted Kozachenko--Leonenko estimator $\hat{H}_n$, and then to show how these arguments may be adapted to yield the desired asyptotic expansion for $\mathrm{Var}(\hat{H}_n^w)$.  

\subsubsection{Bounds on $S_1,\ldots,S_5$} 
\label{Appendix:S} 

\emph{To bound $S_1$:} By similar methods to those used to bound $R_1$ in the proof of Lemma~\ref{Lemma:Bias} in the main text, it is straightforward to show that for every $\epsilon > 0$, we have 
\[
S_1 = \int_{\mathcal{X}_n^c} f(x) \int_0^1 \mathrm{B}_{k,n-k}(s) \log^2 u_{x,s} \, ds \, dx = O\biggl(\frac{k^{\frac{\alpha}{\alpha + d} - \epsilon}}{n^{\frac{\alpha}{\alpha + d} - \epsilon}}\biggr).
\]

\bigskip

\noindent \emph{To bound $S_2$:} For every $\epsilon > 0$, we have that
\begin{align*}
	S_2 &= \int_{\mathcal{X}_n} f(x) \int_\frac{a_n}{n-1}^1 \mathrm{B}_{k,n-k}(s) \log^2 u_{x,s} \, ds \, dx =o(n^{-(3-\epsilon)}),
\end{align*}
by very similar arguments to those used to bound $R_2$ in the proof of Lemma~\ref{Lemma:Bias} in the main text.

\bigskip

\noindent \emph{To bound $S_3$:} We have
\begin{align*}
	&\log^2 u_{x,s} - \log^2 \biggl( \frac{(n-1)s}{e^{\Psi(k)}f(x)} \biggr) \\
&= \biggl\{2 \log \biggl( \frac{(n-1)s}{e^{\Psi(k)}f(x)} \biggr) + \log \biggl(\frac{V_df(x)h_x^{-1}(s)^d}{s} \biggr)\biggr\}\log \biggl(\frac{V_df(x)h_x^{-1}(s)^d}{s} \biggr).
\end{align*}
It therefore follows from Lemma~\ref{Lemma:hxinvbounds}(ii) in the main text that for every $\epsilon > 0$,
\begin{align*}
S_3 &= \int_{\mathcal{X}_n} f(x) \int_0^{\frac{a_n}{n-1}} \mathrm{B}_{k,n-k}(s) \biggl\{\log^2 u_{x,s} - \log^2\biggl( \frac{(n-1)s}{e^{\Psi(k)}f(x)} \biggr)\biggr\} \, ds \, dx \\
&= O\biggl\{\max\biggl( \frac{k^{\beta/d}}{n^{\beta/d}} \log n \, , \, \frac{k^{\frac{\alpha}{\alpha+d}-\epsilon}}{n^{\frac{\alpha}{\alpha+d}-\epsilon}} \biggr)\biggr\}.
\end{align*}

\bigskip

\noindent \emph{To bound $S_4$:} A simplified version of the argument used to bound $R_4$ in Lemma~\ref{Lemma:Bias} of the main text shows that for every $\epsilon > 0$, 
\[
S_4 = \int_{\mathcal{X}_n} f(x) \int_{\frac{a_n}{n-1}}^1 \mathrm{B}_{k,n-k}(s) \log^2\biggl( \frac{(n-1)s}{e^{\Psi(k)}f(x)} \biggr) \, ds \, dx = o(n^{-(3-\epsilon)}).
\]

\noindent \emph{To bound $S_5$:} Very similar arguments to those used to bound $R_1$ in Lemma~\ref{Lemma:Bias} in the main text show that for every $\epsilon > 0$, 
\[
S_5 = \int_{\mathcal{X}_n^c} f(x) \log^2 f(x) \, dx = O\biggl( \frac{k^{\frac{\alpha}{\alpha+d}-\epsilon}}{n^{\frac{\alpha}{\alpha+d}-\epsilon}} \biggr).
\]

\subsubsection{Bounds on $T_1$, $T_2$ and $T_3$}
\label{Appendix:T}

\noindent \emph{To bound $T_1$}: Let $\mathrm{B} \sim \mathrm{Beta}(k-1,n-k-1)$.  By~\eqref{Eq:LotsofTerms} in the main text, for every $\epsilon > 0$,
\begin{align*}
T_{11} &:= \biggl|\int_{\mathcal{X}_n^c \times \mathcal{X}_n^c} f(x)f(y) \log f(y) \int_{\tilde{u}_{n,x,y}}^\infty \log (uf(x)) \,d(\tilde{F}_{n,x}-F^-_{n,x})(u) \,dy \,dx\biggr| \\
&\leq \frac{n-2}{n-k-1}\int_{\mathcal{X}_n^c \times \mathcal{X}_n^c} f(x)f(y) |\log f(y)| \\
&\hspace{2cm}\int_0^1 \bigl|\log (u_{x,s}f(x))\bigr|\mathrm{B}_{k-1,n-k-1}(s)\biggl|1 - \frac{(n-2)s}{k-1}\biggr| \, ds \,dy \,dx \\
&\lesssim \int_{\mathcal{X}_n^c \times \mathcal{X}_n^c} f(x)f(y) |\log f(y)|\biggl[\mathbb{E}\biggl\{\biggl(\log \frac{1}{\mathrm{B}} + \log \frac{1}{1-B}\biggr)\biggl|1 - \frac{(n-2)\mathrm{B}}{k-1}\biggr|\biggr\} \\
&\hspace{0.5cm}+ \biggl\{\log n + |\log f(x)| + \log\biggl(1 +\frac{\|x\|}{\mu_\alpha^{1/\alpha}(f)}\biggr)\biggr\}\mathbb{E}\biggl|1- \frac{(n-2)\mathrm{B}}{k-1}\biggr|\biggr] \,dy \,dx \\
&= o\biggl(\frac{k^{-\frac{1}{2} + \frac{2\alpha}{\alpha+d} - \epsilon}}{n^{\frac{2\alpha}{\alpha+d} - \epsilon}}\biggr),
\end{align*}
where we used the Cauchy--Schwarz inequality and elementary properties of beta random variables to obtain the final bound.

Now let 
\[
u_n^*(x) := u_{x,a_n/(n-1)} = \frac{V_d(n-1)h_x^{-1}(\frac{a_n}{n-1})^d}{e^{\Psi(k)}},
\]
and consider
\[
T_{12} := \biggl|\int_{\mathcal{X}_n^c} \int_{\mathcal{X}_n} f(x)f(y) \log f(y) \int_{\tilde{u}_{n,x,y}}^\infty \log (uf(x)) \,d(\tilde{F}_{n,x}-F^-_{n,x})(u) \,dy \,dx\biggr|.
\]
If $\tilde{u}_{n,x,y} \geq u_n^*(x)$, then by very similar arguments to those used to bound $R_1$ and $R_2$ (cf.~\eqref{Eq:LotsofTerms} and~\eqref{Eq:betatail} in the main text), together with Cauchy--Schwarz, 
\begin{align}
\label{Eq:tilde>*}
\int_{\tilde{u}_{n,x,y}}^\infty \bigl|\log (uf(x))\bigr| \,&d(\tilde{F}_{n,x}-F^-_{n,x})(u) \nonumber \\
&\leq \int_\frac{a_n}{n-1}^1 |\log(u_{x,s}f(x))|\{ \mathrm{B}_{k-1,n-k}(s) + \mathrm{B}_{k,n-k-1}(s) \} \,ds  \nonumber\\
&\lesssim \frac{\log n + |\log f(x)| + \log\Bigl(1 + \frac{\|x\|}{\mu_\alpha^{1/\alpha}(f)}\Bigr)}{n^{3-\epsilon}},
\end{align}
for every $\epsilon > 0$.  On the other hand, if $\tilde{u}_{n,x,y} < u_n^*(x)$, then $\|x-y\| < r_{n,u_n^*(x)} + r_{n,u_n^*(y)}$, 
%\begin{align}
%\label{Eq:Tail}
%&\biggl|\int_{\mathcal{X}_n^c} \int_{\mathcal{X}_n} f(x)f(y) \log f(y) \mathbbm{1}_{\{\tilde{u}_{n,x,y} < u_n^*(x)\}} \int_{\tilde{u}_{n,x,y}}^\infty \log (uf(x)) \,d(\tilde{F}_{n,x}-F^-_{n,x})(u) \,dy \,dx\biggr| \nonumber \\
%&\leq \int_{\mathcal{X}_n^c} \int_{\mathcal{X}_n} f(x)f(y) |\log f(y)| \mathbbm{1}_{\{\|x-y\| < r_{n,u_n^*(x)} + r_{n,u_n^*(y)}\}} \int_0^\infty |\log (uf(x))| \,d(\tilde{F}_{n,x}-F^-_{n,x})(u) \,dy \,dx,
%\end{align}
where we have added the $r_{n,u_n^*(y)}$ term to aid a calculation later in the proof.  Define the sequence 
\[
\rho_n := \bigl[c_n \log^{1/d} (n-1)\bigr]^{-1}.
\]
From Lemma~\ref{Lemma:hxinvbounds}(ii) in the main text, 
\[
	\sup_{y \in \mathcal{X}_n} r_{n,u_n^*(y)} = \sup_{y \in \mathcal{X}_n} h_y^{-1}\Bigl(\frac{a_n}{n-1}\Bigr) \lesssim \sup_{y \in \mathcal{X}_n} \biggl\{ \frac{k \log n}{n f(y)} \biggr\}^{1/d} \! \! \! \! \leq \biggl( \frac{k \log n}{n \delta_n} \biggr)^{1/d} \! \! \! = o(\rho_n).
\]
Now suppose that $x \in \mathcal{X}_n^c$ and $y \in \mathcal{X}_n$ satisfy $\|y-x\| \leq \rho_n$.  %Choose $n_0 \in \mathbb{N}$ large enough that $r_{n,u_n^*(y)} \leq \rho_n/2$ for all $y \in \mathcal{X}_n$, that $\log(n-1) \geq \max\{12^dd^{d/2}, 12 V_d^{-1} 2^d\}$, and that
%\[
%\frac{15d^{1/2}}{7} \rho_n c_n + \frac{3dV_d}{d+1}c_n\Bigl(\frac{\rho_n}{2}\Bigr)^2 \leq \frac{1}{2},
%\]
%for all $n \geq n_0$ and $k \in \{k_0^*,\ldots,k_1^*\}$.  Then, for $n \geq n_0$, using Lemma~\ref{Lemma:15over7} and the fact that $B_x(\rho_n/2) \subseteq B_y(3\rho_n/2)$,
%\begin{align}
%\label{Eq:Bxrhon}
%	&\int_{B_x(\rho_n/2)} f(w) \,dw  \geq V_d f(x) (\rho_n/2)^d - \frac{3dV_d}{d+1} a(f(y))f(y) (\rho_n/2)^{d+2} \nonumber \\
%	& \geq V_d (\rho_n/2)^d f(y) \Bigl(1 -\frac{15d^{1/2}}{7} \rho_n c_n \Bigr)  - \frac{3dV_d}{d+1} f(y)c_n (\rho_n/2)^{d+2} \nonumber \\
%	& \geq \frac{1}{2} V_d (\rho_n/2)^d \delta_n \geq \frac{a_n}{n-1}. 
%\end{align}
Choose $n_0 \in \mathbb{N}$ large enough that $r_{n,u_n^*(y)} \leq \rho_n/2$ for all $y \in \mathcal{X}_n$, and that $\log(n-1) \geq \max\{(3/2)^d(8d^{1/2})^{d/\beta}, 12 V_d^{-1} 2^d\}$ %, and that $\log(n-1) \geq (3/2)^d (8d^{1/2})^{d/\beta}$
%\[
%\Bigl(\frac{3c_n \rho_n}{2}\Bigr)^\beta \leq \frac{1}{2},
%\]
for all $n \geq n_0$ and $k \in \{k_0^*,\ldots,k_1^*\}$. Then when $\beta \in (0,1]$ and $n \geq n_0$, using the fact that $B_x(\rho_n/2) \subseteq B_y(3\rho_n/2)$, we have
\begin{align}
\label{Eq:Bxrhon}
	\int_{B_x(\rho_n/2)}& f(w) \,dw  \geq V_d f(y) (\rho_n/2)^d - V_d a(f(y))f(y) (\rho_n/2)^d (3 \rho_n/2)^\beta \nonumber \\
	& \geq V_d f(y) (\rho_n/2)^d \{ 1 - (3c_n \rho_n/2)^\beta \} \geq \frac{1}{2} V_d (\rho_n/2)^d \delta_n \geq \frac{a_n}{n-1}.
\end{align}
%When $\beta \in (1,2]$, using Lemma~\ref{Lemma:15over7} and the fact that $B_x(\rho_n/2) \subseteq B_y(3\rho_n/2)$ we have that
%\begin{align*}
%	\int_{B_x(\rho_n/2)} f(w) \,dw &\geq V_d f(x) (\rho_n/2)^d - \frac{dV_d}{d+1} a(f(y))f(y) (\rho_n/2)^{d+1} (3 \rho_n/2)^{\beta-1} \\
%	& \geq V_d f(y) (\rho_n/2)^d \Bigl\{1 - \frac{15d^{1/2}}{7} \rho_n c_n - \frac{d}{3(d+1)} (3 c_n \rho_n/2)^\beta \Bigr\} \\
%	& \geq \frac{1}{2} V_d (\rho_n/2)^d \delta_n \geq \frac{a_n}{n-1}.
%\end{align*}
%}
Hence, for all $n \geq n_0$, $x \in \mathcal{X}_n^c$, $y \in \mathcal{X}_n$ with $\|y-x\| \leq \rho_n$ and $k \in \{k_0^*,\ldots,k_1^*\}$,
\begin{equation}
\label{Eq:rhon}
r_{n,u_n^*(x)} + r_{n,u_n^*(y)} \leq \rho_n.	
\end{equation}
On other hand, suppose instead that $x \in \mathcal{X}_n^c$ and $\rho_x^* := \inf_{y \in \mathcal{X}_n} \|y - x\| \geq \rho_n$.  Since $\mathcal{X}_n$ is a closed subset of $\mathbb{R}^d$, we can find $y^* \in \mathcal{X}_n$ such that $\|y^*-x\| = \rho_x^*$, and set $\tilde{x} := \frac{\rho_n}{\rho_x^*} x + \bigl(1 -\frac{\rho_n}{\rho_x^*}\bigr)y^*$.  Then $\|\tilde{x} - y^*\| = \rho_n$, so from~\eqref{Eq:Bxrhon}, we have $r_{n,u_n^*(\tilde{x})} \leq \rho_n/2$ for $n \geq n_0$ and $k \in \{k_0^*,\ldots,k_1^*\}$.  Since $B_{\tilde{x}}(\rho_n/2) \subseteq B_x(\rho_x^*-\rho_n/2)$, we deduce that $r_{n,u_n^*(x)} \leq \rho_x^*-\rho_n/2$ and 
\begin{equation}
\label{Eq:empty}
	\{y \in \mathcal{X}_n: \|x-y\| < r_{n,u_n^*(x)}+r_{n,u_n^*(y)} \} = \emptyset
\end{equation}
for $n \geq n_0$ and $k \in \{k_0^*,\ldots,k_1^*\}$.  But for $n \geq n_0$,
\begin{equation}
\label{Eq:tendstozero}
	\sup_{x \in \mathcal{X}_n^c} \sup_{y \in \mathcal{X}_n:\|y-x\| \leq \rho_n} \frac{1}{f(y)} |f(x)-f(y)| \leq \frac{15d^{1/2}}{7} (c_n \rho_n)^\beta < \frac{1}{2},
\end{equation}
so that if $x \in \mathcal{X}_n^c$, $y \in \mathcal{X}_n$ and $\|x-y\| \leq \rho_n$, then $f(y) < 2 \delta_n$ for $n \geq n_0$ and $k \in \{k_0^*,\ldots,k_1^*\}$.  

It therefore follows from~\eqref{Eq:tilde>*},~\eqref{Eq:rhon},~\eqref{Eq:empty},~\eqref{Eq:tendstozero} and the argument used to bound $T_{11}$ that for each $\epsilon > 0$ and $n \geq n_0$,
\begin{align*}
T_{12} &\leq \int_{\mathcal{X}_n^c} \int_{\mathcal{X}_n} f(x)f(y) |\log f(y)| \mathbbm{1}_{\{\|x-y\| < r_{n,u_n^*(x)} + r_{n,u_n^*(y)}\}} \\
&\hspace{2cm}\int_0^\infty |\log (uf(x))| \,d(\tilde{F}_{n,x}-F^-_{n,x})(u) \,dy \,dx + o(n^{-2}) \\
&\leq \int_{\mathcal{X}_n^c} \int_{y:f(y)<2\delta_n} f(x)f(y) |\log f(y)| \\
&\hspace{2cm} \int_0^\infty |\log (uf(x))| \,d(\tilde{F}_{n,x}-F^-_{n,x})(u) \,dy \,dx + o(n^{-2}) \\
&= o\biggl(\frac{k^{-\frac{1}{2} + \frac{2\alpha}{\alpha+d} - \epsilon}}{n^{\frac{2\alpha}{\alpha+d} - \epsilon}}\biggr).
\end{align*}
Finally for $T_1$, we define
\begin{align*}
T_{13} := \biggl|\int_{\mathcal{X}_n} \int_{B_x^c\bigl(\frac{r_{n,1}d_n}{f(x)^{1/d}}\bigr)} f(x) &f(y) \log f(y) \\
	& \int_{\tilde{u}_{n,x,y}}^\infty \log \bigl(uf(x)\bigr) \, d(\tilde{F}_{n,x} - F_{n,x}^-)(u) \, dy \, dx\biggr|.
\end{align*}
By~Lemma~\ref{Lemma:hxinvbounds}(ii) in the main text, we can find $n_1 \in \mathbb{N}$ such that for $n \geq n_1$, $k \in \{k_0^*,\ldots,k_1^*\}$, $x \in \mathcal{X}_n$ and $s \leq a_n/(n-1)$, we have $V_df(x)h_x^{-1}(s)^d \leq 2s$.  Thus, for $n \geq n_1$, $k \in \{k_0^*,\ldots,k_1^*\}$, $x \in \mathcal{X}_n$ and $y \in B_x^c(\frac{r_{n,1}d_n}{f(x)^{1/d}})$,
\[
\tilde{u}_{n,x,y} \geq \frac{24\log n}{f(x)} \geq \frac{2a_n}{f(x)e^{\Psi(k)}} \geq u_n^*(x).
\]
Thus, from~\eqref{Eq:tilde>*}, $T_{13} = O(n^{-2}\log n)$.  We conclude that for every $\epsilon > 0$, 
\[
|T_1| \leq T_{11} + T_{12} + T_{13} = o\biggl(\frac{k^{-\frac{1}{2} + \frac{2\alpha}{\alpha+d} - \epsilon}}{n^{\frac{2\alpha}{\alpha+d} - \epsilon}}\biggr).
\]

\bigskip

\noindent \emph{To bound $T_2$}: Fix $x \in \mathcal{X}_n$ and $z \in B_0(d_n)$.  Choosing $n_2 \in \mathbb{N}$ large enough that $\frac{r_{n,1}d_n}{\delta_n^{1/d}} \leq (8d^{1/2})^{-1/\beta} c_n^{-1}$ for $n \geq n_2$, we have by Lemma~\ref{Lemma:15over7} that
\[
\sup_{y \in B_x\bigl(\frac{r_{n,1}d_n}{\delta_n^{1/d}}\bigr)}\biggl|\frac{f(y)}{f(x)} - 1\biggr| \leq \frac{1}{2}
\]
for $n \geq n_2$, $k \in \{k_0^*,\ldots,k_1^*\}$.  Also, for all $n \geq n_2, k \in \{k_0^*,\ldots,k_1^*\}$, 
%{\bf and $\beta \in (1,2]$}, and for some $\xi$ on the line segment joining $x$ and $y_{x,z}$, by Lemma~\ref{Lemma:15over7}
%\begin{align*}
%&\bigl|f(y_{x,z})\log f(y_{x,z}) - f(x) \log f(x) - \{\log f(x) + 1\}\dot{f}(x)^T(y_{x,z}-x)\bigr| \\
%&= \bigl| \{\log f(\xi) + 1\}\dot{f}(\xi)^T(y_{x,z}-x)- \{\log f(x) + 1\}\dot{f}(x)^T(y_{x,z}-x)\bigr| \\
%& \leq \biggl|(y_{x,z}-x)^T \dot{f}(\xi) \log\Bigl(\frac{f(\xi)}{f(x)}\Bigr) \biggr| + \bigl|\{1+\log f(x)\} (y_{x,z}-x)^T \{\dot{f}(\xi) - \dot{f}(x)\}\bigr| \\
%& \leq 2 \|y_{x,z}-x\| \|\dot{f}(\xi)\| \Bigl| \frac{f(\xi)}{f(x)} -1 \Bigr| + \{1+|\log f(x)|\} \|y_{x,z}-x\| \|\dot{f}(\xi)-\dot{f}(x)\| \\
%& \leq 2 \|y_{x,z}-x\|^2 \frac{15d^{1/2}}{7} \Bigl( 1+ \frac{15d^{1/2}}{7} \Bigr) a(f(x))^2 f(x) + \{1+|\log f(x)|\} a(f(x)) f(x) \|y_{x,z}-x\|^\beta
%&\leq (r_{n,1}d_n)^2f(x)^{1-2/d}\biggl[ \frac{135d^{1/2}}{28}  a(f(x))^2 + a(f(x))\{|\log f(x)| + 1 \}\biggr].
%\end{align*}   
we have
\begin{align*}
	\bigl|f(y_{x,z})\log f(y_{x,z}) &- f(x) \log f(x) \bigr| \\
&\leq f(y_{x,z}) |\log (f(y_{x,z})/f(x))| + |\log f(x) | |f(y_{x,z})-f(x)| \\
	& \leq a(f(x))f(x) \|y_{x,z}-x\|^\beta \{ |\log f(x)| + 4\}.
\end{align*}
Moreover, by arguments used to bound $T_{11}$,
\begin{align*}
&\biggl|\int_{ \|z\|^d /f(x)}^\infty \log (uf(x)) \,d(\tilde{F}_{n,x}-F^-_{n,x})(u)\biggr| \lesssim \mathbb{E}\biggl|\log(\mathrm{B})\biggl(1 - \frac{(n-2)\mathrm{B}}{k-1}\biggr)\biggr| \\
&\hspace{10pt}+ \biggl\{\log n + |\log f(x)| + \log\biggl(1 + \frac{\|x\|}{\mu_\alpha^{1/\alpha}(f)}\biggr)\biggr\}\mathbb{E}\biggl|1- \frac{(n-2)\mathrm{B}}{k-1}\biggr|,
\end{align*}
where $\mathrm{B} \sim \mathrm{Beta}(k-1,n-k-1)$.  It follows that for every $\epsilon > 0$,
\begin{align*}
T_2 &= \frac{e^{\Psi(k)}}{V_d(n-1)}\int_{\mathcal{X}_n} \int_{B_0(d_n)} \{f(y_{x,z}) \log f(y_{x,z}) - f(x) \log f(x)\} \\
&\hspace{4cm}\int_{ \|z\|^d /f(x)}^\infty \log (uf(x)) \,d(\tilde{F}_{n,x}-F^-_{n,x})(u) \,dz \, dx \\
&= O\biggl(\frac{k^{1/2}}{n}\max\biggl\{\frac{k^{\frac{\alpha}{\alpha+d}-\epsilon}}{n^{\frac{\alpha}{\alpha+d}-\epsilon}} \, , \, \frac{k^{\beta/d}}{n^{\beta/d}}\log^{2+\beta/d} n\biggr\}\biggr).
\end{align*}

\bigskip

\noindent \emph{To bound $T_3$}: Note that by Fubini's theorem,
\begin{align*}
\int_{\mathcal{X}_n} &f(x) \log f(x) \int_{B_0(d_n)} \int_{\frac{\|z\|^d}{f(x)}}^\infty \log(uf(x)) \, d(\tilde{F}_{n,x}-F^-_{n,x})(u) \, dz \, dx \\
&= V_d \int_{\mathcal{X}_n} f(x) \log f(x) \int_0^\infty uf(x) \log(uf(x)) \, d(\tilde{F}_{n,x}-F^-_{n,x})(u) \, dx \\
&= V_d\int_{\mathcal{X}_n} f(x) \log f(x) \int_0^{u_n^*(x)} uf(x) \log(uf(x)) \, d(\tilde{F}_{n,x}-F^-_{n,x})(u) \, dx \\
& \hspace{8cm}+ O(n^{-(3-\epsilon)}),
\end{align*}
for every $\epsilon > 0$, where the order of the error term follows from the same argument used to obtain~\eqref{Eq:tilde>*} and Lemma~\ref{Lemma:hxinvbounds}(i).  Thus, for every $\epsilon > 0$,
\begin{align*}
T_3 &= \frac{k-1}{n-k-1}\int_{\mathcal{X}_n} f(x) \log f(x) \int_0^\frac{a_n}{n-1} \biggl\{\frac{V_df(x)h_x^{-1}(s)^d}{s}\log (u_{x,s}f(x)) \\
&\hspace{1cm}- \log \biggl( \frac{(n-1)s}{e^{\Psi(k)}} \biggr)\biggr\} \mathrm{B}_{k,n-k-1}(s)\biggl\{1 - \frac{(n-2)s}{k-1}\biggr\} \, ds \, dx + O(n^{-(3-\epsilon)}) \\
&= O\biggl(\frac{k^{1/2}}{n}\max\biggl\{\frac{k^{\frac{\alpha}{\alpha+d}-\epsilon}}{n^{\frac{\alpha}{\alpha+d}-\epsilon}} \, , \, \frac{k^{\beta/d}}{n^{\beta/d}}\log n\biggr\}\biggr).
\end{align*}

\subsubsection{Bounds on $U_1$ and $U_2$}
\label{Appendix:U}

\noindent \emph{To bound $U_1$}: Using Lemma~\ref{Lemma:hxinvbounds}(i) and~\eqref{Eq:LotsofTerms} in the main text as in our bounds on $T_{11}$ we have that for every $\epsilon > 0$,
\begin{align}
\label{Eq:U11}
	U_{11} &:= \biggl|\int_{\mathcal{X}_n^c} f(x) \int_0^{u_n^*(x)} \log\bigl(uf(x)\bigr) \, d(F_{n,x}^- - F_{n,x})(u) \, dx\biggr| \nonumber \\
	& \leq \int_{\mathcal{X}_n^c} f(x) \int_0^\frac{a_n}{n-1} |\log(u_{x,s}f(x)) | \mathrm{B}_{k,n-k-1}(s) \biggl| \frac{(n-1)s-k}{n-k-1} \biggr| \,ds \,dx \nonumber \\
	& \hspace{6cm}= o\biggl(\frac{k^{\frac{1}{2} + \frac{\alpha}{\alpha+d}-\epsilon}}{n^{1+\frac{\alpha}{\alpha+d}-\epsilon}}\biggr).
\end{align}
Moreover, using arguments similar to those used to bound $R_2$ in the proof of Lemma~\ref{Lemma:Bias} in the main text, for every $\epsilon > 0$,
\begin{equation}
\label{Eq:U12}
U_{12} := \biggl|\int_{\mathcal{X}} f(x) \int_{u_n^*(x)}^\infty \log\bigl(uf(x)\bigr) \, d(F_{n,x}^- - F_{n,x})(u) \, dx\biggr| = o(n^{-(3-\epsilon)}).
\end{equation}
From~\eqref{Eq:U11}, and~\eqref{Eq:U12}, we have for every $\epsilon > 0$ that
\[
|U_1| \leq U_{11} + U_{12} = o\biggl(\frac{k^{\frac{1}{2} + \frac{\alpha}{\alpha+d}-\epsilon}}{n^{1+\frac{\alpha}{\alpha+d}-\epsilon}}\biggr).
\]

\bigskip

\noindent \emph{To bound $U_2$}: By Lemma~\ref{Lemma:hxinvbounds}(ii) and letting $\mathrm{B} \sim \mathrm{Beta}(k+\beta/d,n-k-1)$, we have that for every $\epsilon > 0$,
\begin{align*}
U_{21} &:= \biggl| \int_{\mathcal{X}_n} \! \! f(x) \! \int_0^\frac{a_n}{n-1} \! \! \log\Bigl(\frac{V_d f(x)h_x^{-1}(s)^d}{s}\Bigr) \mathrm{B}_{k,n-k-1}(s) \Bigl\{\frac{(n-1)s - k}{n-k-1}\Bigr\} ds \, dx\biggr| \\
&\lesssim \frac{k^{\beta/d}}{n^{\beta/d}}\mathbb{E}\biggl(\biggl|\frac{(n-1)\mathrm{B} - k}{n-k-1}\biggr|\biggr) \int_{\mathcal{X}_n} a(f(x)) f(x)^{1-\beta/d} \, dx \\
&= O\biggl(\frac{k^{1/2}}{n}\max\biggl\{\frac{k^{\beta/d}}{n^{\beta/d}} \, , \, \frac{k^{\frac{\alpha}{\alpha+d} - \epsilon}}{n^{\frac{\alpha}{\alpha+d} - \epsilon}}\biggr\}\biggr).
\end{align*}
Moreover, we can use similar arguments to those used to bound $R_4$ in the proof of Lemma~\ref{Lemma:Bias} in the main text to show that for every $\epsilon > 0$,
\begin{align*}
U_{22} &:= \biggl|\int_{\mathcal{X}_n} f(x) \int_\frac{a_n}{n-1}^1 \log\biggl(\frac{(n-1)s}{e^{\Psi(k)}}\biggr) \mathrm{B}_{k,n-k-1}(s)\biggl\{\frac{(n-1)s - k}{n-k-1}\biggr\} \, ds \, dx\biggr| \\
	&= o(n^{-(3-\epsilon)}).
\end{align*}
We deduce that for every $\epsilon > 0$,
\[
|U_2| \leq U_{21} + U_{22} = O\biggl(\frac{k^{1/2}}{n}\max\biggl\{\frac{k^{\beta/d}}{n^{\beta/d}} \, , \, \frac{k^{\frac{\alpha}{\alpha+d} - \epsilon}}{n^{\frac{\alpha}{\alpha+d} - \epsilon}}\biggr\}\biggr).
\]

\subsubsection{Bounds on $W_1,\ldots,W_4$}
\label{Appendix:W}

\noindent \emph{To bound $W_1$}: We partition the region $([l_x,v_x]\times [l_y,v_y])^c$ into eight rectangles as follows:
\begin{align*}
\bigl([l_x,v_x]&\times [l_y,v_y]\bigr)^c \! = \! \bigl([0,l_x) \times [0,l_y)\bigr) \! \cup \! \bigl([0,l_x) \times [l_y,v_y]\bigr) \! \cup \! \bigl([0,l_x) \times (v_y,\infty)\bigr) \\
&\cup \bigl([l_x,v_x] \times [0,l_y)\bigr) \cup \bigl([l_x,v_x] \times (v_y,\infty)\bigr) \cup \bigl((v_x,\infty) \times [0,l_y)\bigr) \\
&\cup \bigl((v_x,\infty) \times [l_y,v_y]\bigr) \cup \bigl((v_x,\infty) \times (v_y,\infty)\bigr).
\end{align*}
Recall our shorthand $h(u,v) = \log(uf(x))\log(vf(y))$.  By Lemma~\ref{Lemma:hxinvbounds}(i) and the Cauchy--Schwarz inequality, as well as very similar arguments to those used to bound $R_2$ in the proof of Lemma~\ref{Lemma:Bias} in the main text, we can bound the contributions from each rectangle individually, to obtain that for every $\epsilon > 0$,
\begin{align*}
W_1 &= \int_{\mathcal{X} \times \mathcal{X}} f(x)f(y) \int_{([l_x,v_x]\times [l_y,v_y])^c} \! \! h(u,v) \, d(F_{n,x,y}-F_{n,x}F_{n,y})(u,v) \,dx \, dy \\
%&\leq 2\int_{\mathcal{X} \times \mathcal{X}} f(x)f(y) \int_0^{l_x} \int_{v_y}^\infty |h(u,v)| \, dF_{n,x,y}(u,v) \, dx \, dy \\
%&+ 2\int_{\mathcal{X} \times \mathcal{X}}\! f(x)f(y) \int_{l_x}^{v_x} \int_{v_y}^\infty |h(u,v)| \, dF_{n,x,y}(u,v) \, dx \, dy\\
%&+ 2\int_{\mathcal{X} \times \mathcal{X}} \! f(x)f(y) \int_{l_x}^{v_x} \int_0^{l_y} |h(u,v)| \, dF_{n,x,y}(u,v) \, dx \, dy\\
%&+ \int_{\mathcal{X} \times \mathcal{X}} \! f(x)f(y)  \int_{v_x}^\infty \int_{v_y}^\infty |h(u,v)| \, dF_{n,x,y}(u,v) \, dx \, dy\\
%&+ \int_{\mathcal{X} \times \mathcal{X}}\! f(x)f(y)  \int_0^{l_x} \int_0^{l_y} |h(u,v)| \, dF_{n,x,y}(u,v) \, dx \, dy + o(n^{-2}) \\
&= o(n^{-(9/2-\epsilon)}).
\end{align*}

\bigskip

\noindent \emph{To bound $W_2$}: We have
\[
W_2 = \int_{\mathcal{X} \times \mathcal{X}} f(x)f(y)  \int_{l_x}^{v_x} \int_{l_y}^{v_y}  h(u,v) \,d(G_{n,x,y} - F_{n,x}F_{n,y})(u,v) \, dx \, dy + \frac{1}{n}.
\]
We write $\mathrm{B}_{a,b,c} := \frac{\Gamma(a)\Gamma(b)\Gamma(c)}{\Gamma(a+b+c)}$, and, for $s, t > 0$ with $s + t < 1$, let 
\begin{equation}
\label{Eq:dirichlet}
\mathrm{B}_{a,b,c}(s,t) := \frac{s^{a-1}t^{b-1}(1-s-t)^{c-1}}{\mathrm{B}_{a,b,c}}
\end{equation}
denote the density of a $\mathrm{Dirichlet}(a,b,c)$ random vector at $(s,t)$.  For $a,b > -1$, writing $I_n:=[a_n^-/(n-1),a_n^+/(n-1)]$, let  
%It will be convenient in several places to note that by Pinsker's inequality, for every $a, b > -k$ and $\epsilon > 0$,
%We first claim that for every $a, b > -1$ and $\epsilon > 0$,
%\begin{align}
%\label{Eq:Pinsker}
%L_n := \int_0^{\frac{a_n^+}{n-1}} \int_0^{\frac{a_n^+}{n-1}} \biggl|\frac{s^{k+a-1}t^{k+b-1}(1-s-t)^{n-2k-2}}{\mathrm{B}(k+a,k+b,n-2k-1)} &- \frac{s^{k+a-1}t^{k+b-1}(1-s)^{n-k-1}(1-t)^{n-k-1}}{\mathrm{B}(k+a,n-k)\mathrm{B}(k+b,n-k)}\biggr| \, ds \, dt \nonumber \\
%&\hspace{3cm}= \frac{k}{n}\{1+o(1)\}.
%\end{align}
%To see this, it is convenient to write
\begin{align*}
\mathrm{B}^{(n)}_{k+a,n-k} &:= \int_{I_n} s^{k+a-1}(1-s)^{n-k-1} \, ds, \\
\mathrm{B}^{(n)}_{k+a,n-k}(s) &:=s^{k+a-1}(1-s)^{n-k-1}/\mathrm{B}^{(n)}_{k+a,n-k}  \\
\mathrm{B}^{(n)}_{k+a,k+b,n-2k-1} &:= \int_{I_n \times I_n} s^{k+a-1}t^{k+b-1}(1-s-t)^{n-2k-2} \, ds \, dt \\
\mathrm{B}^{(n)}_{k+a,k+b,n-2k-1}(s,t) &:= s^{k+a-1}t^{k+b-1}(1-s-t)^{n-2k-2} /\mathrm{B}^{(n)}_{k+a,k+b,n-2k-1}.
\end{align*}
Then by the triangle and Pinsker's inequalities, and Beta tail bounds similar to those used previously, we have that
\begin{align}
\label{Eq:Pinsker}
&\int_{I_n \times I_n} \bigl|\mathrm{B}_{k+a,k+b,n-2k-1}(s,t) -\mathrm{B}_{k+a,n-k}(s)\mathrm{B}_{k+b,n-k}(t)\bigr| \, ds \, dt \nonumber \\
&\leq \biggl|\frac{\mathrm{B}^{(n)}_{k+a,k+b,n-2k-1}}{\mathrm{B}_{k+a,k+b,n-2k-1}}-1\biggr| + \biggl|\frac{\mathrm{B}^{(n)}_{k+a,n-k}\mathrm{B}^{(n)}_{k+b,n-k}}{\mathrm{B}_{k+a,n-k}\mathrm{B}_{k+b,n-k}}-1\biggr|\nonumber \\
&+ \biggl\{2 \! \int_{I_n \times I_n} \mathrm{B}^{(n)}_{k+a,k+b,n-2k-1}(s,t) \log\Bigl(\frac{\mathrm{B}^{(n)}_{k+a,k+b,n-2k-1}(s,t)}{\mathrm{B}^{(n)}_{k+a,n-k}(s)\mathrm{B}^{(n)}_{k+b,n-k}(t)}\Bigr) \, ds \, dt \biggr\}^{1/2} \nonumber \\
&= \biggl\{2 \! \int_0^1 \int_0^{1-t} \!\mathrm{B}_{k+a,k+b,n-2k-1}(s,t) \log\biggl(\frac{\mathrm{B}_{k+a,k+b,n-2k-1}(s,t)}{\mathrm{B}_{k+a,n-k}(s)\mathrm{B}_{k+b,n-k}(t)}\biggr) \, ds \, dt \biggr\}^{1/2} \nonumber \\	
& \hspace{10cm} + o(n^{-2}) \nonumber \\
&= 2^{1/2}\biggl[\log\Bigl(\frac{\Gamma(n+a+b-1)\Gamma(n-k)^2}{\Gamma(n-2k-1)\Gamma(n+a)\Gamma(n+b)}\Bigr) + (n-2k-2)\psi(n-2k-1) \nonumber \\
&\hspace{1.5cm}-(n-k-1)\{\psi(n+b-k-1) + \psi(n+a-k-1)\}  \nonumber \\ 
&\hspace{5cm}+n\psi(n+a+b-1)\biggr]^{1/2} + o(n^{-2}) \nonumber \\
&= \frac{k}{n}\{1+o(1)\}.
\end{align}
As a first step towards bounding $W_2$ note that
\begin{align}
\label{Eq:W21}
W_{21} &:= \int_{\mathcal{X}_n \times \mathcal{X}_n} f(x)f(y) \int_{l_x}^{v_x} \int_{l_y}^{v_y} h(u,v) \, d(G_{n,x,y} - F_{n,x}F_{n,y})(u,v) \,dx \, dy \nonumber \\
&= \int_{\mathcal{X}_n \times \mathcal{X}_n} f(x)f(y) \int_{I_n \times I_n} \log (u_{x,s}f(x))\log (u_{y,t}f(y)) \nonumber\\
&\hspace{2cm}\bigl\{\mathrm{B}_{k,k,n-2k-1}(s,t) -\mathrm{B}_{k,n-k}(s)\mathrm{B}_{k,n-k}(t)\bigr\} \, ds \, dt \,dx \, dy \nonumber \\
&= \int_{\mathcal{X}_n \times \mathcal{X}_n} f(x)f(y) \int_{I_n \times I_n} \log\biggl(\frac{(n-1)s}{e^{\Psi(k)}}\biggr)\log\biggl(\frac{(n-1)t}{e^{\Psi(k)}}\biggr) \nonumber \\
&\hspace{2cm}\bigl\{\mathrm{B}_{k,k,n-2k-1}(s,t) -\mathrm{B}_{k,n-k}(s)\mathrm{B}_{k,n-k}(t)\bigr\} \, ds \, dt \,dx \, dy + W_{211} \nonumber \\
&= -\frac{1}{n} + O\biggl(\frac{k^{\frac{\alpha}{\alpha+d}-\epsilon}}{n^{1+\frac{\alpha}{\alpha+d}-\epsilon}}\biggr) + O(n^{-2}) + W_{211},
\end{align}
for every $\epsilon > 0$.  But, by Lemma~\ref{Lemma:hxinvbounds}(ii) and~\eqref{Eq:Pinsker}, for every $\epsilon > 0$,
\begin{align}
\label{Eq:W211}
|W_{211}| &= \biggl|\int_{\mathcal{X}_n \times \mathcal{X}_n} f(x)f(y) \int_{I_n \times I_n} \biggl\{2\log\biggl(\frac{V_dh_x^{-1}(s)^df(x)}{s}\biggr)\log\biggl(\frac{(n-1)t}{e^{\Psi(k)}}\biggr) \nonumber \\
&\hspace{2.5cm}+\log\biggl(\frac{V_dh_x^{-1}(s)^df(x)}{s}\biggr)\log\biggl(\frac{V_dh_y^{-1}(t)^df(y)}{t}\biggr)\biggr\} \nonumber \\
&\hspace{2.5cm}\bigl\{\mathrm{B}_{k,k,n-2k-1}(s,t) -\mathrm{B}_{k,n-k}(s)\mathrm{B}_{k,n-k}(t)\bigr\} \, ds \, dt \,dx \, dy\biggr| \nonumber \\
&\leq 2\biggl|\int_{\mathcal{X}_n \times \mathcal{X}_n} f(x)f(y) \int_{I_n}\log\biggl(\frac{V_dh_x^{-1}(s)^df(x)}{s}\biggr) \nonumber \\
&\hspace{1cm}\Bigl[\bigl\{\log(n-1) - \Psi(n-k-1) + \log(1-s)\bigr\}\mathrm{B}_{k,n-k-1}(s) \nonumber \\
&\hspace{1cm}- \bigl\{\log(n-1) - \Psi(n)\bigr\}\mathrm{B}_{k,n-k}(s)\Bigr] \, ds \, dx \, dy\biggr| \nonumber \\
& \hspace{3cm} +O\biggl(\max\biggl\{\frac{k^{1+\frac{2\beta}{d}}}{n^{1+\frac{2\beta}{d}}} \, , \, \frac{k^{1+\frac{2\alpha}{\alpha+d}-\epsilon}}{n^{1+\frac{2\alpha}{\alpha+d}-\epsilon}}\biggr\}\biggr) \nonumber \\
&= O\biggl(\frac{k^{1/2}}{n}\max\biggl\{\frac{k^{\beta/d}}{n^{\beta/d}} \, , \, \frac{k^{\frac{\alpha}{\alpha+d} - \epsilon}}{n^{\frac{\alpha}{\alpha+d} - \epsilon}}\biggr\}\biggr).
\end{align}
Moreover, by  Lemma~\ref{Lemma:hxinvbounds}(i) and (ii) and very similar arguments, for every $\epsilon > 0$,
\begin{align}
\label{Eq:W22}
W_{22} &:= \int_{\mathcal{X}_n \times \mathcal{X}_n^c} f(x)f(y) \int_{l_x}^{v_x} \int_{l_y}^{v_y} h(u,v) \, d(G_{n,x,y} - F_{n,x}F_{n,y})(u,v) \,dx \, dy \nonumber \\
&= O\biggl(\frac{k^{1+\frac{\alpha}{\alpha+d} - \epsilon}}{n^{1+\frac{\alpha}{\alpha+d} - \epsilon}} \max\biggl\{\frac{k^{\frac{\alpha}{\alpha+d} - \epsilon}}{n^{\frac{\alpha}{\alpha+d} - \epsilon}} \, , \, \frac{k^{\beta/d}}{n^{\beta/d}} \, , \, \frac{1}{k^{1/2}}\biggr\}\biggr) \nonumber \\
W_{23} &:= \int_{\mathcal{X}_n^c \times \mathcal{X}_n^c} f(x)f(y)  \int_{l_x}^{v_x} \int_{l_y}^{v_y} h(u,v) \, d(G_{n,x,y} - F_{n,x}F_{n,y})(u,v) \,dx \, dy \nonumber \\
&= O\biggl(\frac{k^{1+\frac{2\alpha}{\alpha+d} - \epsilon}}{n^{1+\frac{2\alpha}{\alpha+d} - \epsilon}}\biggr).
\end{align}
Incorporating our restrictions on $k$, we conclude from~\eqref{Eq:W21},~\eqref{Eq:W211} and~\eqref{Eq:W22} that for every $\epsilon > 0$,
\[
|W_2| \leq \biggl|W_{21}+\frac{1}{n}\biggr| + 2|W_{22}| + |W_{23}| = O\biggl(\frac{k^{1/2}}{n}\max\biggl\{\frac{k^{\beta/d}}{n^{\beta/d}} \, , \, \frac{k^{\frac{\alpha}{\alpha+d} - \epsilon}}{n^{\frac{\alpha}{\alpha+d} - \epsilon}}\biggr\}\biggr).
\]

\bigskip

\noindent \emph{To bound $W_3$}: We write $h_u$, $h_v$ and $h_{uv}$ for the partial derivatives of $h(u,v)$ and write, for example, $(h_uF)(u,v) = h_u(u,v)F(u,v)$.  We find on integrating by parts that, writing $F=F_{n,x,y} - G_{n,x,y}$,
\begin{align}
\label{Eq:Nasty}
&\int_{[l_x,v_x] \times [l_y,v_y]} (h \,dF)(u,v) - \int_{l_x}^{v_x} \int_{l_y}^{v_y} (h_{uv}F(u,v)) \,du \,dv  \nonumber \\
&= \! \int_{l_x}^{v_x} \! \bigl[ (h_uF)(u,l_y)-(h_uF)(u,v_y) \bigr] du + \!\int_{l_y}^{v_y}\! \bigl[ (h_vF)(l_x,v)-(h_vF)(v_x,v) \bigr] dv \nonumber \\
& \hspace{2cm}+(hF)(v_x,v_y) + (hF)(l_x,l_y)-(hF)(v_x,l_y)-(hF)(l_x,v_y).
\end{align}
Using standard binomial tail bounds as used to bound $W_1$ together with~\eqref{Eq:LotsofTerms} in the main text we therefore see that for every $\epsilon > 0$,
\begin{align}
\label{Eq:byparts2}
W_{31}&:= \! \int_{\mathcal{X} \times \mathcal{X}} \! \! \! \! \! f(x)f(y) \Bigl\{\int_{l_x}^{v_x} \! \! \!\int_{l_y}^{v_y} \! (h \,dF)(u,v) \! -\! \int_{l_x}^{v_x}\! \! \! \int_{l_y}^{v_y} \!(h_{uv}F)(u,v) \, du \, dv\Bigr\} dx \, dy  \nonumber \\
	&= \!- \int_{\mathcal{X} \times \mathcal{X}}  \! \! \! \! \! f(x)f(y) \biggl\{ \int_{l_x}^{v_x}(h_uF)(u,v_y) \,du + \int_{l_y}^{v_y} (h_vF)(v_x,v) \,dv \biggr\} \,dx \,dy \nonumber \\
	& \hspace{3cm}+o(n^{-(9/2-\epsilon)}).
\end{align}
Now, uniformly for $u \in [l_x,v_x]$ and $(x,y) \in \mathcal{X} \times \mathcal{X}$ and for every $\epsilon > 0$,
\begin{align}
\label{Eq:Margin}
	F(u,v_y) &= \mathbbm{1}_{\{\|x-y\| \leq r_{n,u}\}} \binom{n-2}{k-1} p_{n,x,u}^{k-1}(1-p_{n,x,u})^{n-k-1} +o(n^{-(9/2-\epsilon)}) \nonumber \\
	& = \mathbbm{1}_{\{\|x-y\| \leq r_{n,u}\}} \frac{\mathrm{B}_{k,n-k}(p_{n,x,u})}{n-1} +o(n^{-(9/2-\epsilon)}) \nonumber \\
	& \leq \mathbbm{1}_{\{\|x-y\| \leq r_{n,v_x}\}} \frac{1}{(2 \pi k)^{1/2}} \{1+o(1)\} +o(n^{-(9/2-\epsilon)}).
\end{align}
By~\eqref{Eq:tendstozero} and the arguments leading up to it, we have
\begin{equation}
\label{Eq:DensRatio}
	\sup_{x \in \mathcal{X}_n^c} \sup_{y \in \mathcal{X}_n \cap B_x(r_{n,v_x}+r_{n,v_y})} \Bigl| \frac{f(x)}{f(y)} - 1 \Bigr| \rightarrow 0.
\end{equation}
We therefore have by~\eqref{Eq:LotsofTerms} in the main text that, for every $\epsilon>0$,
\begin{equation}
\label{Eq:W411}
	\int_{\mathcal{X}_n^c \times \mathcal{X}} f(x)f(y) \int_{l_x}^{v_x}(h_uF)(u,v_y) \,du \,dy \,dx = O \biggl( \frac{k^{-\frac{1}{2} + \frac{2 \alpha}{\alpha+d} - \epsilon}}{n^{ \frac{2 \alpha}{\alpha+d} - \epsilon}} \biggr).
\end{equation}
Now, using Lemma~\ref{Lemma:hxinvbounds}(ii), for $x \in \mathcal{X}_n$, 
\begin{equation}
\label{Eq:lxvxbounds}
	\max\{|l_xf(x)-1|,|v_xf(x)-1|\} \lesssim a(f(x)) \biggl( \frac{k}{nf(x)} \biggr)^{\beta/d} + \frac{\log^{1/2} n }{k^{1/2}}.
\end{equation}
We also need some control over $vf(y)$. By~\eqref{Eq:tendstozero} and the work leading up to it, for $n \geq \max(n_0,5), x \in \mathcal{X}_n$ and $\|y-x\|\leq r_{n,v_x}+r_{n,v_y}$, 
\[
	f(y) \geq \Bigl\{1-\frac{15d^{1/2}}{7} (c_n \rho_n)^{\beta} \Bigr\} \delta_n \geq \delta_n/2 \geq k/(n-1).
\]
Thus $a(f(y)) \leq c_n^\beta$ and using~\eqref{Eq:DensRatio} we may apply Lemma~\ref{Lemma:hxinvbounds}(ii) to the set
\[
	\mathcal{X}_n'= \mathcal{X}_n \cup \{y: \|y-x\| \leq r_{n,v_x}+r_{n,v_y} \, \, \text{for some} \, \, x \in \mathcal{X}_n \}.
\]
From this and~\eqref{Eq:DensRatio}, for any $x \in \mathcal{X}_n$ and $y \in B_x(r_{n,v_x}+r_{n,v_y})$, 
\begin{equation}
\label{Eq:lyvy}
	\max( |l_yf(y)-1|, |v_yf(y)-1|) \lesssim a(f(y)) \biggl( \frac{k}{nf(x)} \biggr)^{\beta/d} + \frac{ \log^{1/2} n }{k^{1/2}}.
\end{equation}
Using~\eqref{Eq:DensRatio} again, we have that $a(f(y_{x,z})) \lesssim f(x)^{-\epsilon}$ for each $\epsilon >0$, uniformly for $x \in \mathcal{X}_n$ and $\|z\| \leq \{v_xf(x)\}^{1/d} + \{v_yf(x)\}^{1/d}$.  From~\eqref{Eq:Margin}, \eqref{Eq:lxvxbounds} and~\eqref{Eq:lyvy} we therefore have that
\begin{align}
\label{Eq:W412}
	&\biggl| \int_{\mathcal{X}_n \times \mathcal{X}} f(x)f(y) \int_{l_x}^{v_x}(h_uF)(u,v_y) \,du \,dy \,dx \biggr| \nonumber \\
 	& \lesssim k^{-1/2} \int_{\mathcal{X}_n \times \mathcal{X}} f(x)f(y) \mathbbm{1}_{\{\|x-y\| < r_{n,v_x}\}} | \log(v_yf(y))| \log(v_x/l_x) \,dy \,dx \nonumber \\
	& = O \biggl( \max \biggl\{ \frac{k^{1/2+2\beta/d}}{n^{1+2\beta/d}} \, , \, \frac{\log n}{nk^{1/2}} \, , \, \frac{k^{\frac{1}{2}+\frac{\alpha}{\alpha+d}-\epsilon}}{n^{1+\frac{\alpha}{\alpha+d}-\epsilon}} \biggr\} \biggr)
\end{align}
for every $\epsilon>0$. By~\eqref{Eq:byparts2}, \eqref{Eq:W411} and~\eqref{Eq:W412} we therefore have that
\begin{equation}
\label{Eq:W41}
W_{31} = O \biggl( \max \biggl\{ \frac{k^{1/2+2\beta/d}}{n^{1+2\beta/d}} \, , \, \frac{\log n}{nk^{1/2}} \, , \, \frac{k^{-1/2 + \frac{2 \alpha}{\alpha+d} - \epsilon}}{n^{ \frac{2 \alpha}{\alpha+d} - \epsilon}} \biggr\} \biggr).
\end{equation}
Finally, by~\eqref{Eq:LotsofTerms} in the main text and~\eqref{Eq:DensRatio}, we have since $F=0$ when $\|x-y\| > r_{n,u}+r_{n,v}$ that
\begin{align}
\label{Eq:W42}
W_{32}:=\int_{\mathcal{X}_n^c \times \mathcal{X}} f(x)f(y) \int_{l_x}^{v_x} \int_{l_y}^{v_y} (h_{uv}F)(u,v) \,du \,dv \,dx \,dy =O\biggl( \frac{k^{\frac{2\alpha}{\alpha+d} - \epsilon}}{n^{\frac{2\alpha}{\alpha+d} - \epsilon}} \biggr).
\end{align}
Combining~\eqref{Eq:W41} and~\eqref{Eq:W42} we have that 
\[
	W_3=W_{31}+W_{32} = O \biggl( \max \biggl\{ \frac{k^{1/2+2\beta/d}}{n^{1+2\beta/d}} \, , \, \frac{\log n}{nk^{1/2}} \, , \, \frac{k^{\frac{2 \alpha}{\alpha+d} - \epsilon}}{n^{ \frac{2 \alpha}{\alpha+d} - \epsilon}} \biggr\} \biggr).
\]

\bigskip

\noindent \emph{To bound $W_4$}: Let $p_\cap := \int_{B_x(r_{n,u}) \cap B_y(r_{n,v})} f(y) \, dy$ and let $(N_1,N_2,N_3,N_4) \sim \mathrm{Multi}(n-2,p_{n,x,u} - p_\cap,p_{n,y,v} - p_{\cap},p_\cap, 1-p_{n,x,u} - p_{n,y,v} + p_\cap)$.  Further, let
\[
F_{n,x,y}^{(1)}(u,v) := \mathbb{P}(N_1+N_3 \geq k, N_2 + N_3 \geq k),
\]
so that 
\begin{align*}
	(F_{n,x,y}-&F_{n,x,y}^{(1)})(u,v) = \mathbb{P}(N_1+N_3 = k-1,N_2+N_3 \geq k)\mathbbm{1}_{\{\|x-y\| \leq r_{n,u}\}} \nonumber \\
&+ \mathbb{P}(N_2+N_3=k-1,N_1+N_3 \geq k)\mathbbm{1}_{\{\|x-y\| \leq r_{n,v}\}} \nonumber \\
&+ \mathbb{P}(N_1+N_3 = k-1,N_2+N_3=k-1)\mathbbm{1}_{\{\|x-y\| \leq r_{n,u} \wedge r_{n,v}\}}.
\end{align*}
Now $\mathbb{P}(N_1+N_3=k-1) = \binom{n-2}{k-1}p_{n,x,u}^{k-1}(1-p_{n,x,u})^{n-k-1} \leq (2 \pi k)^{-1/2}\{1+o(1)\}$ and $F_{n,x,y}(u,v)=G_{n,x,y}(u,v)$ if $\|x-y\| > r_{n,u}+r_{n,v}$, and so, by~\eqref{Eq:lxvxbounds} and~\eqref{Eq:lyvy}, we have that 
\begin{align}
\label{Eq:prenormal}
&\int_{\mathcal{X}_n \times \mathcal{X}} f(x)f(y) \int_{l_x}^{v_x} \int_{l_y}^{v_y} \frac{(F_{n,x,y}-G_{n,x,y})(u,v)}{uv} \,du \,dv \,dx \, dy \nonumber \\
&= \int_{\mathcal{X}_n \times \mathcal{X}} f(x)f(y) \int_{l_x}^{v_x} \int_{l_y}^{v_y} \frac{(F_{n,x,y}^{(1)} - G_{n,x,y})(u,v)}{uv} \,du \,dv \,dx \, dy \nonumber \\
& \hspace{50pt} + O\biggl( \max \biggl\{ \frac{\log n}{nk^{1/2}}\, , \, \frac{k^{\frac{1}{2} + \frac{2\beta}{d}}}{n^{1+\frac{2\beta}{d}}} \, , \, \frac{k^{\frac{1}{2}+ \frac{\alpha}{\alpha+d}-\epsilon}}{n^{1+\frac{\alpha}{\alpha+d}-\epsilon}} \biggr\} \biggr).
\end{align}
We can now approximate $F_{n,x,y}^{(1)}(u,v)$ by $\Phi_\Sigma(k^{1/2}\{uf(x)-1\},k^{1/2}\{vf(x)-1\})$ and $G_{n,x,y}(u,v)$ by $\Phi(k^{1/2}\{uf(x)-1\})\Phi(k^{1/2}\{vf(x)-1\})$. To avoid repetition, we focus on the former of these terms.  To this end, for $i=3,\ldots,n$, let
\[
Y_i := \begin{pmatrix} \mathbbm{1}_{\{X_i \in B_x(r_{n,u})\}} \\ \mathbbm{1}_{\{X_i \in B_y(r_{n,v})\}} \end{pmatrix},
\]
so that $\sum_{i=3}^n Y_i = \begin{pmatrix}N_1+N_3 \\ N_2 + N_3\end{pmatrix}$.  We also define
\begin{align*}
\mu &:= \mathbb{E}(Y_i) = \begin{pmatrix} p_{n,x,u} \\ p_{n,y,v} \end{pmatrix} \\
V &:= \mathrm{Cov}(Y_i) = \begin{pmatrix} p_{n,x,u}(1-p_{n,x,u}) & p_\cap-p_{n,x,u}p_{n,y,v} \\ p_\cap-p_{n,x,u}p_{n,y,v} & p_{n,y,v}(1-p_{n,y,v}) \end{pmatrix},
\end{align*}
When $x \in \mathcal{X}_n$ and $y \in B_x^\circ(r_{n,v_x}+r_{n,v_y})$ we have that, writing $\Delta$ for the symmetric difference and using~\eqref{Eq:DensRatio}, $\mathbb{P}(X_1 \in B_x(r_{n,u}) \Delta B_y(r_{n,v})) > 0$ and so $V$ is invertible. We may therefore set $Z_i := V^{-1/2} (Y_i-\mu)$.  Then by the Berry--Esseen bound of \citet{Gotze:91}, writing $\mathcal{C}$ for the set of closed, convex subsets of $\mathbb{R}^2$ and letting $Z \sim N_2(0,I)$, there exists a universal constant $C_2 > 0$ such that 
\begin{equation}
\label{Eq:berryesseen}
\sup_{C \in \mathcal{C}}\biggl|\mathbb{P}\biggl(\frac{1}{(n-2)^{1/2}}\sum_{i=3}^n Z_i \in C\biggr) - \mathbb{P}(Z \in C)\biggr| \leq \frac{C_2\mathbb{E}(\|Z_3\|^3)}{(n-2)^{1/2}}.
\end{equation}
The distribution of $Z_3$ depends on $x, y, u$ and $v$, but, recalling the substitution $y = y_{x,z}$ as defined in~\eqref{Eq:yxz} in the main text, we claim that for $x \in \mathcal{X}_n$, $y = y_{x,z} \in B_x(r_{n,u} + r_{n,v})$, $u \in [l_x,v_x]$ and $v \in [l_y,v_y]$,
\begin{equation}
\label{Eq:zbound}
\mathbb{E}(\|Z_3\|^3) \lesssim \Bigl(\frac{n}{k \|z\|}\Bigr)^{1/2}.
\end{equation}
To establish this, note that for $x \in \mathcal{X}_n$ and $\|y-x\|\leq r_{n,v_x}+r_{n,v_y}$, we have by~\eqref{Eq:DensRatio},~\eqref{Eq:lxvxbounds} and~\eqref{Eq:lyvy} that $\|y-x\| \lesssim (\frac{k}{nf(x)})^{1/d}$. Thus, for $v \in [l_y,v_y]$, and using Lemma~\ref{Lemma:15over7}, we also have that
\begin{align}
\label{Eq:lyvybounds}
	|vf(x)-1| & \leq \max( |v_yf(y)-1|, |l_yf(y)-1|) + v_y|f(y)-f(x)| \nonumber \\
%	& \lesssim a(f(y)) \biggl( \frac{k}{nf(x)} \biggr)^{\beta/d} + a(f(x)) \biggl( \frac{k}{nf(x)} \biggr)^{\beta/d} + \frac{\log^{1/2} n}{k^{1/2}} \nonumber \\
	& \lesssim a(f(x)\wedge f(y)) \biggl( \frac{k}{nf(x)} \biggr)^{\beta/d} + \frac{\log^{1/2} n}{k^{1/2}}.
\end{align} 
Now, by the definition of $l_x$ and $v_x$,
\begin{equation}
\label{Eq:pbound}
	\max\bigl\{|p_{n,x,u}- k/(n-1)| \, , \, |p_{n,y,v}-k/(n-1)|\bigr\}  \leq \frac{3k^{1/2}\log^{1/2} n}{n-1}
\end{equation}
for all $x,y \in \mathcal{X}$ and $u \in [l_x,v_x], v \in [l_y,v_y]$.  Next, we bound $|\frac{n-2}{k} p_\cap - \alpha_z|$ for $x \in \mathcal{X}_n$ and $y=y_{x,z}$ with $\|z\| \leq \{v_xf(x)\}^{1/d} + \{v_yf(x)\}^{1/d}$. First suppose that $u\geq v$. We may write
\[
	B_x(r_{n,u}) \cap B_y(r_{n,v}) = \{B_x(r_{n,v}) \cap B_y(r_{n,v}) \} \cup [ \{B_x(r_{n,u}) \setminus B_x(r_{n,v})\} \cap B_y(r_{n,v}) ],
\]
where this is a disjoint union.  Writing $I_{a,b}(x) := \int_0^x \mathrm{B}_{a,b}(s) \, ds$ for the regularised incomplete beta function and recalling that $\mu_d$ denotes Lebesgue measure on $\mathbb{R}^d$, we have
\begin{align*}
\mu_d\bigl(B_x(r_{n,v}) \cap B_y(r_{n,v})\bigr) &= V_d r_{n,v}^d I_{\frac{d+1}{2},\frac{1}{2}} \biggl( 1- \frac{\|x-y\|^2}{4r_{n,v}^2} \biggr) \\
&= \frac{ve^{\Psi(k)}}{n-1} I_{\frac{d+1}{2},\frac{1}{2}}  \biggl(1- \frac{\|z\|^2}{4\{vf(x)\}^{2/d} }\biggr)
\end{align*}
and
\[
	\alpha_z =  I_{\frac{d+1}{2},\frac{1}{2}}  \biggl(1- \frac{\|z\|^2}{4}\biggr).
\]
Now,
\[
	\biggl|\frac{d}{dr}  I_{\frac{d+1}{2},\frac{1}{2}}  \biggl(1- \frac{r^2}{4}\biggr) \biggr|= \frac{(1-r^2/4)^\frac{d-1}{2}}{\mathrm{B}_{(d+1)/2,1/2}} \leq \frac{1}{\mathrm{B}_{(d+1)/2,1/2}}.
\]
Hence by the mean value inequality,
\begin{align*}
	\biggl|\mu_d\bigl(B_x(r_{n,v}) &\cap B_y(r_{n,v})\bigr) - \frac{e^{\Psi(k)} \alpha_z }{(n-1)f(x)} \biggr| \\
	& \leq \frac{e^{\Psi(k)}}{n-1} \biggl[ \frac{v\|z\||1-\{vf(x)\}^{-1/d}|}{\mathrm{B}_{(d+1)/2,1/2}} + \frac{\alpha_z}{f(x)} |1-vf(x)| \biggr].
\end{align*}
It follows that for all $x \in \mathcal{X}_n$, $y \in B_x(r_{n,v_x} + r_{n,v_y})$ and $v \in [l_y,v_y]$,
\begin{align*}
	\biggl| \int_{B_x(r_{n,v}) \cap B_y(r_{n,v})} f(w) \,dw &- \frac{e^{\Psi(k)} \alpha_z}{n-1} \biggr| \\
&\lesssim \frac{k}{n} a(f(x) \wedge f(y)) \biggl( \frac{k}{nf(x)} \biggr)^{\beta/d} + \frac{k^{1/2}\log^{1/2} n}{n}
\end{align*}
using \eqref{Eq:lyvybounds} and Lemma~\ref{Lemma:15over7}. We also have by~\eqref{Eq:pbound} that
\begin{align*}
	\int_{\{B_x(r_{n,u}) \backslash B_x(r_{n,v})\} \cap B_y(r_{n,v})} &f(w) \,dw \leq p_{n,x,u}-p_{n,x,v} \\
	& \lesssim \frac{k}{n} a(f(x) \wedge f(y)) \biggl( \frac{k}{nf(x)} \biggr)^{\beta/d} + \frac{k^{1/2}\log^{1/2} n}{n}.
\end{align*}
Thus, when $x \in \mathcal{X}_n$, $y = y_{x,z} \in B_x(r_{n,v_x} + r_{n,v_y})$, $u \in [l_x,v_x]$, $v \in [l_y,v_y]$ and $u\geq v$,
\begin{equation}
\label{Eq:pcapbound}
	\biggl| \frac{n-2}{k} p_\cap - \alpha_z \biggr| \lesssim a(f(x) \wedge f(y)) \biggl( \frac{k}{nf(x)} \biggr)^{\beta/d} + \frac{\log^{1/2} n}{k^{1/2}}.
\end{equation}
We can prove the same bound when $v>u$ similarly, using~\eqref{Eq:lxvxbounds},~\eqref{Eq:lyvybounds} and Lemma~\ref{Lemma:15over7}. We will also require a lower bound on $p_{n,x,u}+p_{n,y,v}-2p_\cap$ in the region where $B_x(r_{n,u}) \cap B_y(r_{n,v}) \neq \emptyset$, i.e., $\|z\| \leq \{uf(x)\}^{1/d}+ \{vf(x)\}^{1/d}$.  By the mean value theorem, 
\[
	1- I_{\frac{d+1}{2},\frac{1}{2}}(1- \delta^2) \geq 2^{1/2}\delta\max\biggl\{\frac{2^{-d/2}}{\mathrm{B}_{(d+1)/2,1/2}} \, , \, 1 - I_{\frac{d+1}{2},\frac{1}{2}}(1/2)\biggr\}
\]
for all $\delta \in [0,1]$. Thus, for $u \geq v$, with $v \in [l_y,v_y]$, $x \in \mathcal{X}_n$, and $y = y_{x,z}$ with $\|z\| \leq 2 \{vf(x)\}^{1/d}$, by~\eqref{Eq:lyvybounds} we have,
\begin{align*}
	\mu_d\bigl(B_x(r_{n,u}) \cap B_y(r_{n,v})^c\bigr) &\geq \mu_d\bigl(B_x(r_{n,v}) \cap B_y(r_{n,v})^c\bigr) \\
&= V_d r_{n,v}^d \biggl\{1- I_{\frac{d+1}{2},\frac{1}{2}} \biggl(1-\frac{\|x-y\|^2}{4r_{n,v}^2} \biggr) \biggr\} \gtrsim \frac{k\|z\|}{nf(x)}.
\end{align*}
When $\|z\|>2 \{vf(x)\}^{1/d}$ we simply have $\mu_d\bigl(B_x(r_{n,v}) \cap B_y(r_{n,v})^c\bigr)=V_dr_{n,v}^d$ and the same overall bound applies.  Moreover, the same lower bound for $\mu_d\bigl(B_y(r_{n,v}) \cap B_x(r_{n,u})^c\bigr)$ holds when $u < v$, $u \in [l_x,v_x]$, $x \in \mathcal{X}_n$, and $y = y_{x,z} \in B_x(r_{n,v_x}+r_{n,v_y})$.  We deduce that for all $x \in \mathcal{X}_n$, $y = y_{x,z} \in B_x(r_{n,v_x}+r_{n,v_y})$, $u \in [l_x,v_x]$ and $v \in [l_y,v_y]$,
\begin{equation}
\label{Eq:p-pcapbound}
p_{n,x,u}+p_{n,y,v}-2p_\cap \geq \max\{p_{n,x,u}-p_\cap \, , \, p_{n,y,v}-p_\cap\} \gtrsim \frac{k}{n} \|z\|.
\end{equation}
We are now in a position to bound $\mathbb{E}(\|Z_3\|^3)$ above for $x \in \mathcal{X}_n$, $y=y_{x,z} \in B_x(r_{n,v_x}+r_{n,v_y})$, $u \in [l_x,v_x]$, $v \in [l_y,v_y]$. 
%First note that
%\begin{align*}
%	|V| &= p_{n,x,u}p_{n,y,v}(1-p_{n,x,u})(1-p_{n,y,v}) - (p_\cap - p_{n,x,u}p_{n,y,v})^2 \\
%	&= (p_{n,x,u}+p_{n,y,v}-2p_\cap) \Bigl\{p_\cap - p_{n,x,u}p_{n,y,v} + \frac{(p_{n,x,u}-p_\cap)(p_{n,y,v}-p_\cap)}{p_{n,x,u}+p_{n,y,v}-2p_\cap} \Bigr\} \\
%	& \geq (p_{n,x,u}+p_{n,y,v}-2p_\cap) (p_\cap - p_{n,x,u}p_{n,y,v}).
%\end{align*}
%Using the first equality for $|V|$ in the above if $\|z\| \geq 1$ and the final inequality in the above if $\|z\| < 1$, by~\eqref{Eq:pbound}, \eqref{Eq:pcapbound} and~\eqref{Eq:p-pcapbound} we have that
%\[
%	\liminf_n \frac{n^2|V|}{k^2 \|z\|} >0,
%\]
%and so $V$ is invertible for $n$ large enough. 
We write
\begin{align}
\label{Eq:Z3}
	\mathbb{E}(\|Z_3\|^3&) = p_\cap \biggl\|V^{-1/2} \begin{pmatrix} 1-p_{n,x,u} \\ 1-p_{n,y,v} \end{pmatrix} \biggr\|^{3} +(p_{n,x,u}-p_\cap) \biggl\|V^{-1/2} \begin{pmatrix} 1-p_{n,x,u} \\ -p_{n,y,v} \end{pmatrix} \biggr\|^{3}  \nonumber \\
	& \hspace{0.1cm}+(p_{n,y,v}-p_\cap) \biggl\|V^{-1/2} \begin{pmatrix} -p_{n,x,u} \\ 1-p_{n,y,v} \end{pmatrix} \biggr\|^{3} \nonumber \\
	& \hspace{0.1cm}+(1-p_{n,x,u}-p_{n,y,v}+p_\cap) \biggl\|V^{-1/2} \begin{pmatrix} p_{n,x,u} \\ p_{n,y,v} \end{pmatrix} \biggr\|^{3},
\end{align}
and bound each of these terms in turn.  First,
\begin{align}
\label{Eq:V11}
	&p_\cap \biggl\|V^{-1/2} \begin{pmatrix} 1-p_{n,x,u} \\ 1-p_{n,y,v} \end{pmatrix} \biggr\|^{3} \nonumber \\
	&= p_\cap|V|^{-3/2} \{(1-p_{n,x,u})(1-p_{n,y,v})(p_{n,x,u}+p_{n,y,v}-2p_\cap) \}^{3/2} \nonumber \\
	&= p_\cap \Biggl\{ \frac{(1-p_{n,x,u})(1-p_{n,y,v})}{p_\cap-p_{n,x,u}p_{n,y,v} + \frac{(p_{n,x,u}-p_\cap)(p_{n,y,v}-p_\cap)}{p_{n,x,u}+p_{n,y,v}-2p_\cap}} \Biggr\}^{3/2} \nonumber \\
	&\leq p_\cap \min \biggl\{ \frac{p_{n,x,u}+p_{n,y,v}}{|V|}, \frac{1}{p_\cap - p_{n,x,u}p_{n,y,v}} \biggr\}^{3/2} \lesssim n^{1/2}/k^{1/2},
\end{align}
using~\eqref{Eq:pbound} and \eqref{Eq:pcapbound}, and where we derive the final bound from the left hand side of the minimum if $\|z\| \geq 1$ and the right hand side if $\|z\| < 1$.  Similarly,
\begin{equation}
\label{Eq:V1p}
	(p_{n,x,u}-p_\cap) \biggl\|V^{-1/2} \begin{pmatrix} 1-p_{n,x,u} \\ -p_{n,y,v} \end{pmatrix} \biggr\|^{3} \leq (p_{n,x,u}-p_\cap)p_{n,y,v}^{3/2}|V|^{-3/2} \lesssim \Bigl(\frac{n}{k\|z\|}\Bigr)^{1/2},
\end{equation}
where we have used~\eqref{Eq:p-pcapbound} for the final bound.  By symmetry, the same bound holds for the third term on the right-hand side of~\eqref{Eq:Z3}.  Finally, very similar arguments yield
\begin{equation}
\label{Eq:Vpp}
	(1-p_{n,x,u}-p_{n,y,v}+p_\cap) \biggl\|V^{-1/2} \begin{pmatrix} p_{n,x,u} \\ p_{n,y,v} \end{pmatrix} \biggr\|^{3} \lesssim (k/n)^{3/2}.
\end{equation}
Combining~\eqref{Eq:V11}, \eqref{Eq:V1p} and \eqref{Eq:Vpp} gives~\eqref{Eq:zbound}.

Writing $\mathbf{\Phi}_A( \cdot)$ for the measure associated with the $N_2(0,A)$ distribution for invertible $A$, and $\phi_A$ for the corresponding density, we have by Pinsker's inequality and a Taylor expansion of the log-determinant function that 
\begin{align*}
	&2 \sup_{C \in \mathcal{C}} |\mathbf{\Phi}_A(C) - \mathbf{\Phi}_B(C)|^2 \leq \int_{\mathbb{R}^2} \phi_A \log \frac{\phi_A}{\phi_B}  \\
	&= \frac{1}{2} \{ \log |B| - \log |A| + \tr(B^{-1}(A-B)) \} \leq \|B^{-1/2}(A-B)B^{-1/2} \|^2,
\end{align*}
provided $\|B^{-1/2}(A-B)B^{-1/2} \| \leq 1/2$. Hence
\[
	\sup_{C \in \mathcal{C}} |\mathbf{\Phi}_A(C) - \mathbf{\Phi}_B(C)| \leq \min\{1, 2 \|B^{-1/2}(A-B)B^{-1/2} \|\}.
\]
We now take $A=(n-2)V/k$, $B= \Sigma$ and use the submultiplicativity of the Frobenius norm along with~\eqref{Eq:pbound} and \eqref{Eq:pcapbound} and the fact that $\|\Sigma^{-1/2}\| = \{(1+\alpha_z)^{-1}+(1-\alpha_z)^{-1}\}^{1/2}$ to deduce that 
\begin{equation}
\label{Eq:covariancematrices}
	\sup_{C \in \mathcal{C}} |\mathbf{\Phi}_A(C) - \mathbf{\Phi}_B(C)| \lesssim \frac{1}{\|z\|} \biggl\{ a(f(x) \wedge f(y)) \biggl( \frac{k}{nf(x)} \biggr)^{\beta/d} + \frac{\log^{1/2} n}{k^{1/2}} \biggr\}
\end{equation}
for $x \in \mathcal{X}_n$, $y \in B_x^\circ(r_{n,v_x}+r_{n,v_y})$, $u \in [l_x,v_x]$, $v \in [l_y,v_y]$. Now let $u=f(x)^{-1}(1+k^{-1/2}s)$ and $v=f(x)^{-1}(1+k^{-1/2}t)$. By the mean value theorem, \eqref{Eq:lxvxbounds} and~\eqref{Eq:lyvybounds},
\begin{align}
\label{Eq:arguments}
	\biggl| \Phi_\Sigma \biggl(k^{-1/2} &\biggl\{ (n-2)\mu - \begin{pmatrix} k \\ k \end{pmatrix} \biggr\} \biggr) - \Phi_\Sigma(s,t) \biggr| \nonumber \\
	& \leq \frac{1}{(2 \pi)^{1/2}} \biggl\{\biggl|\frac{(n-2)p_{n,x,u}-k}{k^{1/2}}-s\biggr| + \biggl|\frac{(n-2)p_{n,y,v}-k}{k^{1/2}}-t\biggr|\biggr\} \nonumber \\
	& \lesssim k^{1/2} a(f(x) \wedge f(y))\biggl( \frac{k}{nf(x)} \biggr)^{\beta/d} + k^{-1/2}.
\end{align}
It follows by~\eqref{Eq:berryesseen}, \eqref{Eq:zbound}, \eqref{Eq:covariancematrices} and \eqref{Eq:arguments} that for $x \in \mathcal{X}_n$ and $y \in B_x^\circ(r_{n,v_x}+r_{n,v_y})$,
\begin{align*}
	&\sup_{u \in [l_x,v_x], v \in [l_y,v_y]} |F_{n,x,y}^{(1)}(u,v) - \Phi_\Sigma(s,t)|\\
	& \lesssim \min \biggl\{1, \frac{\log^{1/2} n}{ k^{1/2}\|z\|}+ a(f(x) \wedge (f(y)) \biggl(\frac{k }{nf(x)} \biggr)^{\beta/d} \Bigl(k^{1/2}+\frac{1}{\|z\|} \Bigr) \biggr\}.
\end{align*}
Therefore, by \eqref{Eq:lxvxbounds} and~\eqref{Eq:lyvy}, and since $f(y) \geq f(x)/2$ for $x \in \mathcal{X}_n$, $y \in B_x(r_{n,v_x}+r_{n,v_y})$ and $n \geq n_0$, we conclude that for each $\epsilon>0$ and $n \geq n_0$
\begin{align}
\label{Eq:normal}
	\biggl| \int_{\mathcal{X}_n \times \mathcal{X}}\!\!\!& f(x)f(y) \int_{l_x}^{v_x}\!\!\! \int_{l_y}^{v_y} \frac{F_{n,x}^{(1)}(u,v) - \Phi_\Sigma(s,t)}{uv} \mathbbm{1}_{\{\|x-y\|\leq r_{n,u}+r_{n,v} \}} \,du \,dv \,dy \,dx \biggr| \nonumber \\
	& \lesssim \frac{k}{n} \int_{\mathcal{X}_n} \! \! f(x) \biggl\{ \frac{\log^{1/2} n}{k^{1/2}} + a(f(x)/2) \biggl(\frac{k}{nf(x)} \biggr)^{\beta/d} \biggr\}^2 \nonumber \\
	& \hspace{2cm} \int_{B_0(3)} \sup_{u \in [l_x,v_x], v \in [l_{y_{x,z}},v_{y_{x,z}}]} \! \! \! |F_{n,x,y_{x,z}}^{(1)}(u,v) - \Phi_\Sigma(s,t)| \,dz \,dx \nonumber \\
	& = O \biggl( \frac{k}{n} \max \biggl\{ \frac{\log^{5/2} n}{k^{3/2}}\, , \, \frac{k^{\frac{1}{2}+\frac{\alpha}{\alpha+d}- \epsilon}}{n^{\frac{\alpha}{\alpha+d}- \epsilon}}\, , \, \frac{k^{-1/2+\beta/d} \log n}{n^{\beta/d}} \, , \, \frac{k^{1/2+{2\beta/d}}}{n^{2\beta/d}} \biggr\} \biggr).
\end{align}
By similar (in fact, rather simpler) means we can establish the same bound for the approximation of $G_{n,x,y}$ by $\Phi(k^{1/2}\{uf(x)-1\})\Phi(k^{1/2}\{vf(x)-1\})$.

To conclude the proof for the unweighted case, we write $\mathcal{X}_n = \mathcal{X}_n^{(1)} \cup \mathcal{X}_n^{(2)}$, where 
\[
	\mathcal{X}_n^{(1)} := \{x: f(x) \geq k^\frac{d}{2\beta} \delta_n \} \, , \quad \mathcal{X}_n^{(2)} := \{x: \delta_n \leq f(x) <k^\frac{d}{2\beta} \delta_n \},
\]
and deal with these two regions separately.  We have by Slepian's inequality that $\Phi_\Sigma(s,t) \geq \Phi(s) \Phi(t)$ for all $s$ and $t$.  Hence, recalling that $s = s_{x,u} = k^{1/2}\{uf(x)-1\}$ and $t = t_{x,v} = k^{1/2}\{vf(x)-1\}$, by~\eqref{Eq:DensRatio}, \eqref{Eq:lxvxbounds} and~\eqref{Eq:lyvybounds}, for every $\epsilon>0$, 
\begin{align}
\label{Eq:mathcalx2}
	\int_{\mathcal{X}_n^{(2)} \times \mathcal{X}}\!\!\!& f(x)f(y) \int_{l_x}^{v_x} \! \! \! \int_{l_y}^{v_y} \frac{\Phi_\Sigma(s,t) - \Phi(s) \Phi(t)}{uv} \mathbbm{1}_{\{\|x-y\|\leq r_{n,u}+r_{n,v} \}} du \,dv \,dy \,dx \nonumber \\
	& \leq \frac{e^{\Psi(k)}}{V_d(n-1)k} \int_{\mathcal{X}_n^{(2)}} \int_{\mathbb{R}^d} f(y_{x,z}) \frac{\mathbbm{1}_{\{\|x-y_{x,z}\|\leq r_{n,v_x}+r_{n,v_{y_{x,z}}} \}}}{f(x)^2 l_xl_{y_{x,z}}} \nonumber \\ 
	& \hspace{3cm} \int_{-\infty}^\infty \int_{-\infty}^\infty \{\Phi_\Sigma(s,t) - \Phi(s) \Phi(t)\} \,ds \,dt \,dz \,dx \nonumber \\
	& \lesssim \frac{1}{n} \int_{\mathcal{X}_n^{(2)}} f(x) \int_{B_0(2)} \alpha_z \, dz \,dx = o \Bigl( \frac{k^{(1+\frac{d}{2\beta})\frac{\alpha}{\alpha+d} - \epsilon}}{n^{1+\frac{\alpha}{\alpha+d} - \epsilon}} \Bigr),
\end{align}
where to obtain the final error term, we have used the fact that $\int_{B_0(2)} \alpha_z \, dz = V_d$.  By~\eqref{Eq:lxvxbounds} and~\eqref{Eq:lyvy} we have, for each $\epsilon > 0$,
\begin{align}
\label{Eq:mathcalx1}
	&\int_{\mathcal{X}_n^{(1)} \times \mathcal{X}} \!\!\! f(x)f(y) \int_{l_x}^{v_x} \!\!\! \int_{l_y}^{v_y} \frac{\Phi_\Sigma(s,t) - \Phi(s) \Phi(t)}{uv} \mathbbm{1}_{\{\|x-y\|\leq r_{n,u}+r_{n,v} \}} \,du \,dv \,dy \,dx \nonumber \\
	& \leq \frac{e^{\Psi(k)}}{V_d(n-1)k} \int_{\mathcal{X}_n^{(1)}} \int_{\mathbb{R}^d} f(y_{x,z}) \frac{\mathbbm{1}_{\{\|x-y_{x,z}\|\leq r_{n,v_x}+r_{n,v_{y_{x,z}}} \}}}{f(x)^2 l_xl_{y_{x,z}}} \alpha_z \,dz \,dx \nonumber \\
	& = \frac{e^{\Psi(k)}}{(n-1)k} \int_{\mathcal{X}_n^{(1)}} f(x) \,dx + O \biggl(\max \biggl\{\frac{\log^{1/2} n}{nk^{1/2}}\, , \, \frac{k^{\beta/d}}{n^{1+\beta/d}}\, , \, \frac{k^{\frac{\alpha}{\alpha+d} - \epsilon}}{n^{1+\frac{\alpha}{\alpha+d} - \epsilon}} \biggr\} \biggr) \nonumber \\
	& = \frac{e^{\Psi(k)}}{(n-1)k} + O \biggl( \max \biggl\{ \frac{\log^{1/2} n}{nk^{1/2}} \, , \, \frac{k^{\beta/d}}{n^{1+\beta/d}} \, , \, \frac{k^{(1+\frac{d}{2\beta})\frac{\alpha}{\alpha+d} - \epsilon}}{n^{1+\frac{\alpha}{\alpha+d} - \epsilon}}  \biggr\} \biggr).
\end{align}
By Lemma~\ref{Lemma:hxinvbounds}(ii) as for~\eqref{Eq:lyvybounds} we have, for $x \in \mathcal{X}_n^{(1)}, y \in B_x(r_{n,v_x}+r_{n,v_y})$,
\begin{equation}
\label{Eq:lxvxbounds2}
\max_{v \in \{v_x,v_y\}} |vf(x)-1- 3 k^{-1/2} \log^{1/2} n| \lesssim a(f(x) \wedge f(y)) \Bigl( \frac{k}{nf(x)} \Bigr)^{\beta/d} = o(k^{-1/2}),
\end{equation}
with similar bounds holding for $l_x$ and $l_y$. A corresponding lower bound of the same order for the left-hand side of~\eqref{Eq:mathcalx1} follows from~\eqref{Eq:lxvxbounds2} and the fact that
\[
	\int_{-2 \sqrt{\log n}}^{2 \sqrt{\log n}} \int_{-2 \sqrt{\log n}}^{2 \sqrt{\log n}} \{\Phi_\Sigma(s,t) - \Phi(s) \Phi(t)\} \,ds \,dt = \alpha_z + O(n^{-2})
\]
uniformly for $z \in \mathbb{R}^d$. It now follows from~\eqref{Eq:prenormal}, \eqref{Eq:normal}, \eqref{Eq:mathcalx2} and \eqref{Eq:mathcalx1} that for each $\epsilon>0$,
\[
	W_4= O \biggl( \max \biggl\{ \frac{\log^{5/2} n}{n k^{1/2}}\, , \, \frac{k^{\frac{3}{2}+\frac{\alpha-\epsilon}{\alpha+d}}}{n^{1+\frac{\alpha-\epsilon}{\alpha+d}}}\, , \, \frac{k^{3/2+2\beta/d}}{n^{1+2\beta/d}} \, , \, \frac{k^{(1+\frac{d}{2\beta})\frac{\alpha-\epsilon}{\alpha+d}}}{n^{1+\frac{\alpha-\epsilon}{\alpha+d}}}\, , \, \frac{k^{\frac{1}{2}+\frac{\beta}{d}} \log n}{n^{1+\frac{\beta}{d}}} \biggr\} \biggr),
\]
as required. 

\bigskip

We now turn our attention to the variance of the weighted Kozachenko--Leonenko estimator $\hat{H}_n^w$.  We first claim that 
\begin{align}
\label{eq:weighteddiagonal}
	\Var \biggl(\sum_{j=1}^k w_j \log \xi_{(j),1} \biggr) &= \sum_{j,l=1}^k w_j w_l \Cov( \log \xi_{(j),1} , \log \xi_{(l),1} ) = V(f) + o(1).
\end{align}
By~\eqref{Eq:LongDisplay},~\eqref{Eq:Sbounds} and Lemma~\ref{Lemma:Bias} in the main text, for $j$ such that $w_j \neq 0$,
\[
	\Var \log \xi_{(j),1} = V(f) + o(1)
\]
as $n \rightarrow \infty$. %Now, if $U_1, \ldots, U_{n-1} \stackrel{iid}{\sim} U[0,1]$, then for $l>j$,
%\[
%	(U_{(j)}, U_{(l)}- U_{(j)}, 1- U_{(l)} ) \sim \text{Dirichlet}(j, l-j, n-l),
%\]
%a Dirichlet distribution with the first two components having density $\mathrm{B}_{j,l-j,n-l}$, defined by~\eqref{Eq:dirichlet}. 
For $l >j$, using similar arguments to those used in the proof of Lemma~\ref{Lemma:Bias} in the main text, and writing $u_{x,s}^{(k)} := u_{x,s} =V_d(n-1) h_x^{-1}(s)^de^{-\Psi(k)}$ for clarity, we have
\begin{align*}
	&\mathbb{E}( \log \xi_{(j),1} \log \xi_{(l),1} )  \\
	&= \int_\mathcal{X} f(x) \int_0^1 \int_0^{1-s} \log (u_{x,s}^{(j)}) \!\log (u_{x,s+t}^{(l)}) \mathrm{B}_{j,l-j,n-l}(s,t) \, dt \,ds \,dx \\
	&\! \! = \!\!\int_\mathcal{X} \!\!\!f(x) \!\int_0^1 \!\!\int_0^{1-s}\! \! \!\! \! \! \log \Bigl(\! \frac{(n-1)s}{f(x)e^{\Psi(j)}}\! \Bigr) \! \log \Bigl(\! \frac{(n-1)(s+t)}{f(x)e^{\Psi(l)}}\! \Bigr)\! \mathrm{B}_{j,l-j,n-l}(s,t) dt\, ds\, dx \!+\!o(1) \\
	&= \!\!\int_\mathcal{X}\!\!\! f(x) \log^2 f(x) \,dx +o(1)
\end{align*}
as $n \rightarrow \infty$, uniformly for $1 \leq j < l \leq k_1^*$. Now~\eqref{eq:weighteddiagonal} follows on noting that $\sup_{k \geq k_d} \|w\| < \infty$.

Next we claim that
\begin{equation}
\label{Eq:WeightedCov}
	\Cov \biggl( \sum_{j=1}^k w_j \log \xi_{(j),1}, \sum_{l=1}^k w_l \log \xi_{(l),2} \biggr) = o(n^{-1})
\end{equation}
as $n \rightarrow \infty$. In view of~\eqref{Eq:Cov1} in the main text and the fact that $\sup_{k \geq k_d} \|w\| < \infty$, it is sufficient to show that
\[
	\Cov\bigl( \log(f(X_1) \xi_{(j),1}) , \log(f(X_2) \xi_{(l),2}) \bigr) = o(n^{-1})
\]
as $n \rightarrow \infty$, whenever $w_j,w_l \neq 0$. We suppose without loss of generality here that $j < l$, since the $j=l$ case is dealt with in~\eqref{Eq:LongerDisplay}. We broadly follow the same approach used to bound $W_1,\ldots,W_4$, though we require some new (similar) notation. Let $F_{n,x,y}'$ denote the conditional distribution function of $(\xi_{(j),1}, \xi_{(l),2})$ given $X_1=x, X_2=y$ and let $F_{n,x}^{(j)}$ denote the conditional distribution function of $\xi_{(j),1}$ given $X_1=x$. Let
\[
r_{n,u}^{(j)} := \biggl\{  \frac{ue^{\Psi(j)}}{V_d(n-1)} \biggr\}^{1/d}, \quad p_{n,x,u}^{(j)} := h_x (r_{n,u}^{(j)}).
\]
Recall the definitions of $a_{n,j}^{\pm}$ given in the proof of Lemma~\ref{Lemma:Vw}, and let $v_{x,j}:=\inf\{u \geq 0 : (n-1)p_{n,x,u}^{(j)}= a_{n,j}^+ \}$ and $l_{x,j}:=\inf\{u \geq 0 : (n-1)p_{n,x,u}^{(j)}= a_{n,j}^- \}$. For pairs $(u,v)$ with $u \leq v_{x,j}$ and $v \leq v_{y,l}$, let $(M_1,M_2,M_3) \sim \text{Multi}(n-2; p_{n,x,u}^{(j)}, p_{n,y,v}^{(l)}, 1- p_{n,x,u}^{(j)} -p_{n,y,v}^{(l)})$ and write
\[
	G_{n,x,y}'(u,v) := \mathbb{P}(M_1 \geq j, M_2 \geq l).
\]
Also write
\[
	\Sigma' := \begin{pmatrix} 1& (j/l)^{1/2}\alpha_z' \\ (j/l)^{1/2}\alpha_z'&1 \end{pmatrix},
\]
where $\alpha_z':= V_d^{-1} \mu_d\bigl(B_0(1) \cap B_z(\exp(\Psi(l)-\Psi(j))^{1/d})\bigr)$. Writing $W_i'$ for remainder terms to be bounded later, we have
\begin{align}
\label{Eq:LongerDisplay2}
	&\Cov\bigl( \log(f(X_1) \xi_{(j),1}) , \log(f(X_2) \xi_{(l),2}) \bigr) \nonumber \\
	&= \int_{\mathcal{X} \times \mathcal{X}} f(x)f(y) \int_{[l_{y,l},v_{y,l}] \times [l_{x,j},v_{x,j}]} \! \! \! \! \! \! \! \! \! \! \! \! \! \!h(u,v) \,d(F_{n,x,y}' \!-\!F_{n,x}^{(j)}F_{n,y}^{(l)})(u,v) \, dx \, dy +W_1' \nonumber \\
	&\!\!= \! \! \int_{\mathcal{X} \times \mathcal{X}} \! \! \! \! \! f(x)f(y) \! \! \int_{[l_{y,l},v_{y,l}] \times [l_{x,j},v_{x,j}]} \! \! \! \! \! \! \! \! \! \! \! \! \! h(u,v) \,d(F_{n,x,y}'\!-\!G_{n,x,y}')(u,v) \, dx \, dy - \frac{1}{n} +\sum_{i=1}^2 W_i' \nonumber \\ 
	&\!\!= \!\!\int_{\mathcal{X}_n \times \mathcal{X}} \!\!\!\!\! f(x)f(y)\!\! \int_{l_{y,l}}^{v_{y,l}}\!\! \int_{l_{x,j}}^{v_{x,j}} \frac{(F_{n,x,y}'-G_{n,x,y}')(u,v)}{uv} \,du \,dv \,dx \,dy - \frac{1}{n} +\sum_{i=1}^3 W_i'\nonumber  \\
	&\!\!= \!\!\frac{V_d^{-1}e^{\Psi(j)}}{(n-1)(jl)^{1/2}} \int_{\mathbb{R}^d} \int_{-\infty}^\infty \int_{-\infty}^\infty \{\Phi_{\Sigma'}(s,t)-\Phi(s) \Phi(t) \} \,ds \,dt \,dz - \frac{1}{n} +\sum_{i=1}^4 W_i' \nonumber \\
	&\!\!=\!\!\frac{V_d^{-1}e^{\Psi(j)}}{(n-1)l} \int_{\mathbb{R}^d} \alpha_z' \,dz - \frac{1}{n} +\sum_{i=1}^4 W_i = O\biggl( \frac{1}{nk} \biggr) +\sum_{i=1}^4 W_i'
\end{align}
as $n \rightarrow \infty$. The final equality here follows from the fact that, for Borel measurable sets $K,L \subseteq \mathbb{R}^d$,
\begin{equation}
\label{Eq:intvol}
	\int_{\mathbb{R}^d} \mu_d\bigl((K+z) \cap L\bigr) \,dz = \mu_d(K) \mu_d(L),
\end{equation}
so that $\int_{\mathbb{R}^d} \alpha_z' \,dz = V_d e^{\Psi(l)-\Psi(j)}$.

\bigskip

\emph{To bound $W_1'$}: Very similar arguments to those used to bound $W_1$ show that $W_1'=o(n^{-(9/2-\epsilon)})$ as $n \rightarrow \infty$, for every $\epsilon > 0$.

\bigskip

\emph{To bound $W_2'$}: Similar to our work used to bound $W_2$, we may show that
\begin{align*}
	\int_{\frac{a_{n,j}^-}{n-1}}^{\frac{a_{n,j}^+}{n-1}} \int_{\frac{a_{n,l}^-}{n-1}}^{\frac{a_{n,l}^+}{n-1}} |\mathrm{B}_{j+a,l+b,n-j-l-1}(s,t) -& \mathrm{B}_{j+a,n-j}(s)\mathrm{B}_{l+b,n-l}(t) | \,dt \,ds \\
	&\leq \frac{(jl)^{1/2}}{n} \{1+o(1)\}
\end{align*}
as $n \rightarrow \infty$, for fixed $a,b>-1$. Also,
\begin{align*}
	\int_0^1\!\!\! \int_0^{1-s} \!\!\!\!\!\! \log \Bigl(\! \frac{(n-1)s}{e^{\Psi(j)}} \Bigr)\! \log \Bigl(\! \frac{(n-1)t}{e^{\Psi(l)}} \Bigr)\! \{ \mathrm{B}_{j,l,n-j-l-1}\!(s,t) -& \mathrm{B}_{j,n-j}(s)\mathrm{B}_{l,n-l}(t) \} dtds\\
	 =&-\frac{1}{n} +O(n^{-2})
\end{align*}
as $n \rightarrow \infty$. Using these facts and very similar arguments to those used to bound $W_2$ we have for every $\epsilon > 0$ that
\[
	W_2' = O\biggl(\frac{k^{1/2}}{n}\max\biggl\{\frac{k^{\beta/d}}{n^{\beta/d}} \, , \, \frac{k^{\frac{\alpha}{\alpha+d} - \epsilon}}{n^{\frac{\alpha}{\alpha+d} - \epsilon}}\biggr\}\biggr).
\]

\bigskip

\emph{To bound $W_3'$}: Similarly to~\eqref{Eq:Nasty} and the surrounding work, we can show that for every $\epsilon > 0$,
\[
	W_3' = O\biggl( \max \biggl\{ \frac{\log n}{nk^{1/2}}\, , \, \frac{k^{\frac{1}{2} +\frac{2\beta}{d}}}{n^{1+ \frac{2\beta}{d}}}\, , \,\frac{k^{\frac{2\alpha}{\alpha+d} - \epsilon}}{n^{\frac{2\alpha}{\alpha+d} - \epsilon}} \biggr\} \biggr).
\]

\bigskip

\emph{To bound $W_4'$}: Let $(N_1,N_2,N_3,N_4) \sim \text{Multi}(n-2; p_{n,x,u}^{(j)} - p_\cap, p_{n,y,v}^{(l)} - p_\cap, p_\cap, 1- p_{n,x,u}^{(j)} -p_{n,y,v}^{(l)}+ p_\cap)$, where $p_\cap := \int_{B_x(r_{n,u}^{(j)}) \cap B_y(r_{n,v}^{(l)})} f(w) \,dw$. Further, let 
\[
	F_{n,x,y}^{',(1)} := \mathbb{P}(N_1+N_3 \geq j, N_2+N_3 \geq l).
\]
Then, as in~\eqref{Eq:prenormal}, we have
\begin{align*}
\int_{\mathcal{X}_n \times \mathcal{X}} &f(x)f(y) \int_{l_{x,j}}^{v_{x,j}} \int_{l_{y,l}}^{v_{y,l}} \frac{(F_{n,x,y}'-G_{n,x,y}')(u,v)}{uv} \,du \,dv \,dx \, dy  \\
&= \int_{\mathcal{X}_n \times \mathcal{X}} f(x)f(y) \int_{l_{x,j}}^{v_{x,j}} \int_{l_{y,l}}^{v_{y,l}} \frac{(F_{n,x,y}^{',(1)}-G_{n,x,y}')(u,v)}{uv} \,du \,dv \,dx \, dy  \\
& \hspace{50pt} + O\biggl( \max \biggl\{ \frac{\log n}{nk^{1/2}}\, , \, \frac{k^{\frac{1}{2} + \frac{2\beta}{d}}}{n^{1+\frac{2\beta}{d}}} \, , \, \frac{k^{\frac{1}{2}+ \frac{\alpha}{\alpha+d}-\epsilon}}{n^{1+\frac{\alpha}{\alpha+d}-\epsilon}} \biggr\} \biggr).
\end{align*}
We can now approximate $F_{n,x,y}^{',(1)}(u,v)$ by $\Phi_{\Sigma'}(j^{1/2}\{uf(x)-1\},l^{1/2}\{vf(x)-1\})$ and $G_{n,x,y}'(u,v)$ by $\Phi(j^{1/2}\{uf(x)-1\})\Phi(l^{1/2}\{vf(x)-1\})$.  This is rather similar to the corresponding approximation in the bounds on $W_4$, so we only present the main differences. First, let
\[
Y_i' := \begin{pmatrix} \mathbbm{1}_{\{X_i \in B_x(r_{n,u}^{(j)})\}} \\ \mathbbm{1}_{\{X_i \in B_y(r_{n,v}^{(l)})\}} \end{pmatrix}.
\]
We also define
\begin{align*}
\mu' := \mathbb{E}(Y_i') = \begin{pmatrix} p_{n,x,u}^{(j)} \\ p_{n,y,v}^{(l)} \end{pmatrix}
\end{align*}
and
\begin{align*}
\quad V' := \mathrm{Cov}(Y_i') = \begin{pmatrix} p_{n,x,u}^{(j)}(1-p_{n,x,u}^{(j)}) & p_\cap-p_{n,x,u}^{(j)}p_{n,y,v}^{(l)} \\ p_\cap-p_{n,x,u}^{(j)}p_{n,y,v}^{(l)} & p_{n,y,v}^{(l)}(1-p_{n,y,v}^{(l)}) \end{pmatrix},
\end{align*}
and set $Z_i' := V'^{-1/2} (Y_i'-\mu)$.  Our aim is to provide a bound on $p_\cap$. Since the function
\[
	(r,s) \mapsto \mu_d\bigl( B_0(r^{1/d}) \cap B_z(s^{1/d}) \bigr),
\]
is Lipschitz we have for $x \in \mathcal{X}_n, y=x+f(x)^{-1/d}r_{n,1}^{(j)}z \in B_x(r_{n,v_{x,j}}^{(j)}+r_{n,v_{y,l}}^{(l)}), u \in [l_{x,j}, v_{x,j}]$ and $v \in [l_{y,l}, v_{y,l}]$ that
\begin{equation}
\label{Eq:pcap2}
	\biggl| \frac{n-2}{e^{\Psi(j)}} p_\cap - \alpha_z' \biggr| \lesssim a(f(x) \wedge (f(y)) \biggl( \frac{k}{nf(x)} \biggr)^{\beta/d} + \frac{\log^{1/2} n}{k^{1/2}},
\end{equation}
using similar equations to~\eqref{Eq:lxvxbounds}, \eqref{Eq:lyvy} and~\eqref{Eq:lyvybounds}. From this and similar bounds to~\eqref{Eq:pbound}, we find that $|V'| \gtrsim k^2/n^2$ and $\|(V')^{-1/2}\| \lesssim (n/k)^{1/2}$. We therefore have
\[
	\mathbb{E} \|Z_3'\|^3 \leq \|(V')^{-1/2} \|^3 \mathbb{E}\|Y_3'-\mu'\|^3 \lesssim n^{1/2}/k^{1/2},
\]
which is as in the $l=j$ case except with the factor of $\|z\|^{-1/2}$ missing.  Note now that
\[
\limsup_{n \rightarrow \infty} \sup_{\overset{(j,l):j < l}{w_j,w_l \neq 0}} \sup_{z \in B_0(1+e^{(\Psi(l)-\Psi(j))/d})} \|(\Sigma')^{-1/2}\| < \infty.
\]
Hence, using~\eqref{Eq:pcap2}, similar bounds to~\eqref{Eq:pbound} and the same arguments as leading up to \eqref{Eq:covariancematrices}, 
\begin{equation}
\label{Eq:covariancematrices2}
	\sup_{C \in \mathcal{C}} |\mathbf{\Phi}_A(C) - \mathbf{\Phi}_B(C)| \lesssim a(f(x) \wedge f(y)) \biggl( \frac{k}{nf(x)} \biggr)^{\beta/d} + \frac{\log^{1/2} n}{k^{1/2}},
\end{equation}
where $B:=\Sigma'$ and
\[
	A:= (n-2)\begin{pmatrix} j^{-1}p_{n,x,u}^{(j)}(1-p_{n,x,u}^{(j)}) & j^{-1/2}l^{-1/2}(p_\cap-p_{n,x,u}^{(j)}p_{n,y,v}^{(l)}) \\ j^{-1/2}l^{-1/2}(p_\cap-p_{n,x,u}^{(j)}p_{n,y,v}^{(l)}) & l^{-1}p_{n,y,v}^{(l)}(1-p_{n,y,v}^{(l)}) \end{pmatrix}.
\]
Now let $u:=f(x)^{-1}(1+j^{-1/2}s)$ and $v:=f(x)^{-1}(1+l^{-1/2}t)$. Similarly to~\eqref{Eq:arguments}, we have
\begin{align*}
	\biggl| \Phi_{\Sigma'} \Bigl(\frac{(n-2)p_{n,x,u}^{(j)}-j}{j^{1/2}},&\frac{(n-2)p_{n,y,v}^{(l)}-l}{l^{1/2}} \Bigr) - \Phi_{\Sigma'}(s,t) \biggr|\! \\
&\lesssim k^{1/2}a(f(x) \wedge f(y)) \biggl( \frac{k}{nf(x)} \biggr)^{\beta/d} \!\! + k^{-1/2}.
\end{align*}
Similarly to the arguments leading up to~\eqref{Eq:normal}, it follows that
\begin{align*}
	&\biggl| \int_{\mathcal{X}_n\times \mathcal{X}}\!\!\!\!\!\! f(x)f(y)\!\! \int_{l_{x,j}}^{v_{x,j}}\!\!\! \int_{l_{y,l}}^{v_{y,l}} \frac{F_{n,x}^{',(1)}(u,v) \!-\! \Phi_{\Sigma'}(s,t)}{uv} \mathbbm{1}_{\{\|x-y\|\leq r_{n,u}^{(j)}+r_{n,v}^{(l)} \}} du \,dv \,dy \,dx \biggr| \\
	& = O \biggl( \frac{k}{n} \max \biggl\{ \frac{\log^{3/2} n}{k^{3/2}}\, , \, \frac{k^{\frac{1}{2}+\frac{\alpha}{\alpha+d}- \epsilon}}{n^{\frac{\alpha}{\alpha+d}- \epsilon}}\, , \, \frac{k^{-1/2+\beta/d} \log n}{n^{\beta/d}}\, , \, \frac{k^{1/2+{2\beta/d}}}{n^{2\beta/d}} \biggr\} \biggr),
\end{align*}
where the power on the first logarithmic factor is smaller because of the absence of the factor of the $\|z\|^{-1}$ term in~\eqref{Eq:covariancematrices2}. The remainder of the work required to bound $W_4'$ is very similar to the work done from~\eqref{Eq:mathcalx2} to~\eqref{Eq:mathcalx1}, using also~\eqref{Eq:intvol}, so is omitted. We conclude that
\[
	W_4'= O \biggl( \max \biggl\{ \frac{\log^\frac{3}{2} n}{n k^\frac{1}{2}}\, , \, \frac{k^{\frac{3}{2}+\frac{\alpha-\epsilon}{\alpha+d}}}{n^{1+\frac{\alpha-\epsilon}{\alpha+d}}}\, , \, \frac{k^{\frac{3}{2}+\frac{2\beta}{d}}}{n^{1+\frac{2\beta}{d}}} \, , \, \frac{k^{(1+\frac{d}{2\beta})\frac{\alpha-\epsilon}{\alpha+d}}}{n^{1+\frac{\alpha-\epsilon}{\alpha+d}}}\, , \, \frac{k^{\frac{1}{2}+\frac{\beta}{d}} \log n}{n^{1+\frac{\beta}{d}}} \biggr\} \biggr).
\]
The equation~\eqref{Eq:LongerDisplay2}, together the bounds on $W_1',\ldots,W_4'$ just proved, establish the claim~\eqref{Eq:WeightedCov}.  We finally conclude from~\eqref{eq:weighteddiagonal} and \eqref{Eq:WeightedCov} that
\begin{align*}
\mathrm{Var}(\hat{H}_n^w) &= \frac{1}{n}\mathrm{Var}\biggl(\sum_{j=1}^k w_j \log \xi_{(j),1}\biggr) \\
&\hspace{2.5cm}+ \biggl(1-\frac{1}{n}\biggr)\mathrm{Cov}\biggl(\sum_{j=1}^k w_j \log \xi_{(j),1} \, , \, \sum_{l=1}^k w_l \log \xi_{(l),2}\biggr) \\
&= V(f) + o(n^{-1}),
\end{align*}
as required.

\section{Proof of Theorem~\ref{Thm:LowerBound}}

\begin{proof}[Proof of Theorem~\ref{Thm:LowerBound}]
For the first part of the theorem we aim to apply Theorem 25.21 of~\citet{vanderVaart1998}, and follow the notation used there. With $\dot{\mathcal{P}} := \{ \lambda( \log f + H(f)) : \lambda \in \mathbb{R}\}$ we will first show that the entropy functional $H$ is differentiable at $f$ relative to the tangent set $\dot{\mathcal{P}}$, with efficient influence function $\tilde{\psi}_f = -\log f - H(f)$. Following Example 25.16 in~\citet{vanderVaart1998}, for $g \in \dot{\mathcal{P}}$, the paths $f_{t,g}$ defined in~\eqref{Eq:Paths} of the main text are differentiable in quadratic mean at $t=0$ with score function $g$.  Note that $\int_{\mathcal{X}} gf = 0$ and $\int_{\mathcal{X}} g^2f < \infty$ for all $g \in \dot{\mathcal{P}}$.  %To show the differentiability of $H$ with the claimed efficient influence function, we must therefore bound the quantity
%\begin{align*}
%	&t^{-1}\{H(f_{t,g})-H(f)\}+ \int_{\mathcal{X}} \{\log f +H(f)\}fg \\
%	&= \frac{1}{t} \int_{\mathcal{X}} \! \Bigl\{\! \Bigl(1 - \frac{2c(t)}{1\!+\!e^{-2tg}} \Bigr) \! \log f - \frac{2c(t)}{1+e^{-2tg}} \log \Bigl( \frac{2c(t)}{1+e^{-2tg}} \Bigr)\! + tg (1+\log f )\Bigr\} f.
%\end{align*}
It is convenient to define, for $t \geq 0$, the set $A_t :=\{x \in \mathcal{X}: 8t|g(x)| \leq 1\}$, on which we may expand $e^{-2tg}$ easily as a Taylor series.  By H\"{o}lder's inequality, for $\epsilon \in (0,1/2)$,
\begin{align*}
	\int_{A_t^c} f |\log f|  &\leq (8t)^{2(1-\epsilon)} \int_{\mathcal{X}} f |g|^{2(1-\epsilon)} |\log f| \\
	& \leq (8t)^{2(1-\epsilon)} \Bigl\{ \int_{\mathcal{X}} g^2f \Bigr\}^{1-\epsilon} \Bigl\{ \int_{\mathcal{X}} f |\log f|^{1/\epsilon} \Bigr\}^\epsilon = o(t)
\end{align*}
as $t \searrow 0$.  Moreover,
\begin{align*}
	\int_{A_t^c} f \log (1+e^{-2tg}) \leq \int_{A_t^c} ( \log 2 + 2t|g|) f \leq 16t^2 (4 \log 2 + 1) \int_{\mathcal{X}} g^2f.
\end{align*}
We also have that
\begin{align}
\label{Eq:c(t)}
	|c(t)^{-1}-1| &= \biggl|\int_{\mathcal{X}} \biggl(\frac{2}{1+e^{-2tg}} - 1 - tg\biggr)f\biggr| \nonumber \\
&\leq \int_{A_t} \biggl|\frac{e^{-2tg} - 1 + 2tg + tg(e^{-2tg}-1)}{1+e^{-2tg}}\biggr|f + \int_{A_t^c} (1 + t|g|)f \nonumber \\
&\leq \frac{16}{3}t^2 \int_{A_t} g^2f + 72t^2 \int_{A_t^c} g^2f \leq 72 t^2 \int_{\mathcal{X}} g^2 f.
\end{align}
%Thus
%\[
%	\frac{1}{t} \int_{A_t^c} \! \Bigl\{ \! \Bigl(1 - \frac{2c(t)}{1\!+\!e^{-2tg}} \Bigr) \! \log f - \frac{2c(t)}{1+e^{-2tg}} \! \log \Bigl( \frac{2c(t)}{1+e^{-2tg}} \Bigr) + tg \log f + tg  \Bigr\} f  \rightarrow 0
%\]
%as $t \searrow 0$. 
It follows that
\begin{align*}
	& \biggl|t^{-1}\{H(f_{t,g})-H(f)\}+ \int_{\mathcal{X}} \{\log f +H(f)\}fg\biggr| \\
	&= \biggl|\frac{1}{t} \int_{\mathcal{X}} \! \Bigl\{\! \Bigl(1 - \frac{2c(t)}{1\!+\!e^{-2tg}} \Bigr) \! \log f - \frac{2c(t)}{1\!+\!e^{-2tg}} \log \Bigl( \frac{2c(t)}{1\!+\!e^{-2tg}} \Bigr)\! + tg (1+\log f )\Bigr\} f\biggr| \\
&\leq \frac{1}{t} \int_{A_t} f \Bigl| \{e^{-2tg}-1+2tg + tg(e^{-2tg}-1)\} \log f  \\
	& \hspace{4cm}- 2 \log \Bigl( \frac{2}{1+e^{-2tg}} \Bigr) + tg(1+e^{-2tg}) \Bigr| + o(1) \\
	& \leq \frac{16}{3}t \int_{\mathcal{X}} g^2 f |\log f| + 22t\int_\mathcal{X} g^2 f+o(1) \rightarrow 0.
\end{align*}
The conclusion~\eqref{Eq:LowerBound} in the main text therefore follows from \citet[][Theorem~25.21]{vanderVaart1998}.

We now establish the second part of the theorem.  First, by our previous bound on $c(t)$ in~\eqref{Eq:c(t)}, for $12t < \{\int_{\mathcal{X}} g^2 f\}^{-1/2}$ we have that
\[
	\|f_{t,g}\|_\infty \leq 2c(t) \|f\|_\infty \leq \frac{2 \|f\|_\infty}{1-72t^2 \int_{\mathcal{X}} g^2f} \leq 4 \|f\|_\infty,
\]
and $\mu_\alpha( f_{t,g}) \leq 4 \mu_\alpha(f)$. 

We now study the smoothness properties of $f_{t,g}$.  This requires some involved calculations, because we first need to understand corresponding properties of $g$.  To this end, for an $m$ times differentiable function $g:\mathbb{R}^d \rightarrow \mathbb{R}$, define
\[
	M_{g}^*(x):=\max \biggl\{ \max_{t=1, \ldots, m} \|g^{(t)}(x) \| \, , \, \sup_{y \in B_x^\circ(r_a(x))} \frac{\|g^{(m)}(y)-g^{(m)}(x) \|}{\|y-x\|^{\beta-m}} \biggr\}
\]
and
\[
D_g:=\max \biggl\{1, \sup_{\delta \in (0,\|f\|_\infty)} \frac{ \sup_{x:f(x) \geq \delta} M_{g}^*(x)}{a(\delta)^{m+1}} \biggr\}.
\]
%By Lemma~\ref{Lemma:15over7}, for $\|y-x\| \leq r_a(x)$,
%\[
%	|f(y)/f(x)-1| \leq \frac{15d^{1/2}}{7}a(f(x)) \|y-x\|^{1 \wedge \beta} < 1/2.
%\]
%Moreover, by Lemma~\ref{Lemma:15over7} again, and using multi-index notation for partial derivatives, for $|\omega| \leq m$ and $\|y-x\| \leq r_a(x)$,
%\[
%	\Bigl| \frac{\partial f^\omega(y)}{\partial x^\omega} - \frac{\partial f^\omega(x)}{\partial x^\omega} \Bigr| \leq \frac{15d^{1/2}}{7} a(f(x))f(x) \|y-x\|^{\min(\beta-|\omega|,1)}.
%\]
Let $\mathcal{J}_m$ denote the set of multisets of elements $\{1,\ldots,d\}$ of cardinality at most $m$, and for $J = \{j_1,\ldots,j_s\} \in \mathcal{J}_m$, define $g_J(x) := \frac{\partial^sg}{\prod_{\ell=1}^s \partial x_\ell}(x)$.  Moreover, for $i \in \{1,\ldots,s\}$, let $\mathcal{P}_i(J)$ denote the set of partitions of $J$ into $i$ non-empty multisets.  As an illustration, if $d=2$, then 
\begin{align*}
\mathcal{J}_3 \! = \! \bigl\{&\emptyset, \{1\}, \{2\}, \{1,1\}, \{1,2\}, \{2,1\}, \{2,2\}, \\
&\{1,1,1\},\! \{1,1,2\},\!\{1,2,1\},\!\{1,2,2\},\!\{2,1,1\},\!\{2,1,2\},\!\{2,2,1\},\!\{2,2,2\}\bigr\}.
\end{align*}
Moreover, if $J = \{1,1,2\} \in \mathcal{J}_3$, then
\[
\mathcal{P}_2(J) = \Bigl\{\bigl\{\{1,1\},\{2\}\bigr\},\bigl\{\{1,2\},\{1\}\bigr\},\bigl\{\{1,2\},\{1\}\bigr\}\Bigr\}.
\]
Then, by induction, and writing $g^* := g_1 = \log f + H(f)$, it may be shown that
\[
g_J^*(x) = \sum_{i=1}^{\mathrm{card}(J)} \frac{(-1)^{i-1}(i-1)!}{f^i} \sum_{\{P_1,\ldots,P_i\} \in \mathcal{P}_i(J)} f_{P_1}\ldots f_{P_i}.
\] 
Now, the cardinality of $\mathcal{P}_i(J)$ is given by a Stirling's number of the second kind:
\[
\mathrm{card}\bigl(\mathcal{P}_i(J)\bigr) = \frac{1}{i!}\sum_{\ell=0}^i (-1)^{i-\ell}\binom{i}{\ell}\ell^{\mathrm{card}(J)} =: S\bigl(\mathrm{card}(J),i\bigr),
\] 
say.  Thus, if $\mathrm{card}(J) \leq m$, then
\begin{equation}
\label{Eq:gJ*1}
|g_J^*(x)| \leq \sum_{i=1}^{\mathrm{card}(J)} (i-1)!S\bigl(\mathrm{card}(J),i\bigr)a(f(x))^i \leq \frac{1}{2}m^{m+1}m!a(f(x))^m.
\end{equation}
Moreover, if $\|y-x\| \leq r_a(x)$ and $m \geq 1$, then
\begin{align*}
|g_J^*(y) - g_J^*(x)| &\leq \sum_{i=1}^{\mathrm{card}(J)} (i-1)! \! \! \! \sum_{\{P_1,\ldots,P_i\} \in \mathcal{P}_i(J)} \! \! \biggl\{\frac{|f_{P_1}\ldots f_{P_i}(y) - f_{P_1}\ldots f_{P_i}(x)|}{f^i(y)} \\
&\hspace{4.5cm}+ \frac{|f_{P_1}\ldots f_{P_i}(x)|}{f^i(y)}\biggl|\frac{f^i(y)}{f^i(x)} - 1\biggr|\biggr\}.
\end{align*}
Now, by Lemma~\ref{Lemma:15over7},
\[
\biggl|\frac{f^i(y)}{f^i(x)} - 1\biggr| \leq i\biggl|\frac{f(y)}{f(x)} - 1\biggr|\biggl(1 + \biggl|\frac{f(y)}{f(x)} - 1\biggr|\biggr)^{i-1} \leq \Bigl(\frac{71}{56}\Bigr)^{i-1}i\biggl|\frac{f(y)}{f(x)} - 1\biggr|.
\]
Moreover, by induction and Lemma~\ref{Lemma:15over7} again,
\[
%\label{Eq:fs}
|f_{P_1}\ldots f_{P_i}(y) - f_{P_1}\ldots f_{P_i}(x)| \leq 8d^{1/2}\Bigl\{\Bigl(\frac{71}{56}\Bigr)^i-1\Bigr\}a(f(x))^if^i(x)\|y-x\|^{\beta - m}.
\]
We deduce that (even when $m=0$),
\begin{equation}
\label{Eq:gJ*2}
|g_J^*(y) - g_J^*(x)| \leq 8d^{1/2}\Bigl(\frac{71}{41}\Bigr)^{m}m!(m+1)^{m+2}a(f(x))^{m+1}\|y-x\|^{\beta-m}.
\end{equation}
Comparing~\eqref{Eq:gJ*1} and~\eqref{Eq:gJ*2}, we see that 
\begin{equation}
\label{Eq:Dlambda}
D_{g^*} \leq 8d^{1/2}\Bigl(\frac{71}{41}\Bigr)^{m}m!(m+1)^{m+2} =:D.
\end{equation}
%Using inductive arguments, it follows that
%\begin{equation}
%\label{Eq:Dlambda}
%	D_{g_1} \leq \max\{ 1 , 5d^{1/2}(7 d^{1/2} m^2)^m/2 \} =:D.
%\end{equation}
Now let $q(y) := (1+e^{-2ty})^{-1}$, so that $f_{t,g}(x) = 2c(t)q\bigl(g(x)\bigr)f(x)$.  Similar inductive arguments to those used above yield that when $J \in \mathcal{J}_m$ with $m \geq 1$ and $g$ is $m$ times differentiable,
\[
(q \circ g)_J(x) = \sum_{i=1}^{\mathrm{card}(J)} q^{(i)}\bigl(g(x)\bigr)\sum_{\{P_1,\ldots,P_i\} \in \mathcal{P}_i(J)} g_{P_1}\ldots g_{P_i}(x),
\]
and we now bound the derivatives of $q$.  By induction,
\[
q^{(i)}(y) = (2t)^i \sum_{\ell=1}^i (-1)^{i-\ell} \frac{a_\ell^{(i)} e^{-2t \ell y}}{(1+e^{-2ty})^{\ell+1}},
\]
where for each $i \in \mathbb{N}$, we have $a_1^{(i)} = 1$, $a_i^{(i)} = i!$ and $a_\ell^{(i)} = \ell(a_\ell^{(i-1)} + a_{\ell-1}^{(i-1)})$ for $\ell \in \{2,\ldots,i-1\}$.  Since $\max_{1 \leq \ell \leq i} a_{\ell}^{(i)} \leq (2i)^{i-1}$ (again by induction), we deduce that
\begin{equation}
\label{Eq:qi}
(1+e^{-2ty})|q^{(i)}(y)| \leq 2^{2i-1} i^i t^i.
\end{equation}
Writing $s := \mathrm{card}(J)$, it follows that
\begin{align}
\label{Eq:qgbound}
|(q \circ g)_J(x)| &\leq q\bigl(g(x)\bigr)\sum_{i=1}^{s} 2^{2i-1} i^i t^i S(s,i)a(f(x))^{i(m+1)}D_g^i \nonumber \\
&\leq q\bigl(g(x)\bigr)s^{s+1} 2^{2s-1}\max(1,t)^s B_s a(f(x))^{s(m+1)}D_g^s,
\end{align}
where $B_s := \sum_{i=1}^s S(s,i)$ denotes the $s$th Bell number.  We can now apply the multivariate Leibniz rule, so that for a multi-index $\omega = (\omega_1,\ldots,\omega_d)$ with $|\omega| \leq m$, and for $t \leq 1$ and $m \geq 1$,
\begin{align}
\label{Eq:ftgderivativesbounds}
	\biggl| \frac{\partial^{\omega}f_{t,g^*}(x)}{\partial x^{\omega}}\biggr| &= \biggl| 2c(t) \sum_{\nu : \nu \leq \omega} \binom{\omega}{\nu} \frac{\partial^\nu q\bigl(g^*(x)\bigr)}{\partial x^\nu} \frac{\partial^{\omega-\nu} f(x)}{\partial x^{\omega-\nu}} \biggr| \nonumber \\
	&\leq 2^{3m-1}m^{m+1} B_m D_{g^*}^m a(f(x))^{m^2+m} f_{t,g^*}(x).
\end{align}
Now, in order to control $\bigl| \frac{\partial^{\omega}f_{t,g^*}(y)}{\partial x^{\omega}} - \frac{\partial^{\omega}f_{t,g^*}(x)}{\partial x^{\omega}}\bigr|$, we first note that by~\eqref{Eq:gJ*2} and~\eqref{Eq:Dlambda}, we have for $\|y-x\| \leq r_a(x)$, $i \in \mathbb{N}$, $J \in \mathcal{J}_m$ with $\mathrm{card}(J) = s$ and $\{P_1,\ldots,P_i\} \in \mathcal{P}_i(J)$,
\begin{equation}
\label{Eq:gdiffs}
|g_{P_1}^*\ldots g_{P_i}^*(y) - g_{P_1}^*\ldots g_{P_i}^*(x)| \leq (2D)^ia(f(x))^{i(m+1)}\|y-x\|^{\beta-m}.
\end{equation}
Thus, by~\eqref{Eq:qi},~\eqref{Eq:gdiffs}, the mean value theorem and Lemma~\ref{Lemma:15over7}, for $t \leq 1$, $\|y-x\| \leq r_a(x)$ and $m \geq 1$,
\begin{align}
\label{Eq:qcircg}
&|(q \circ g^*)_J(y) - (q \circ g^*)_J(x)| \nonumber \\
&\leq \biggl|\sum_{i=1}^s q^{(i)}(g^*(x))\sum_{\{P_1,\ldots,P_i\} \in \mathcal{P}_i(J)} \{g_{P_1}^* \ldots g_{P_i}^*(y) - g_{P_1}^* \ldots g_{P_i}^*(x)\}\biggr| \nonumber \\
&\hspace{0.5cm}+ \biggl|\sum_{i=1}^s \{q^{(i)}(g^*(y)) - q^{(i)}(g^*(x))\} \sum_{\{P_1,\ldots,P_i\} \in \mathcal{P}_i(J)} g_{P_1}^* \ldots g_{P_i}^*(y)\biggr| \nonumber \\
&\leq  D^m q(g^*(x))a(f(x))^{m^2+m+1} \|y-x\|^{\beta-m} \nonumber \\
&\hspace{3cm}\times \frac{B_m2^{3m+5}d^{1/2}(m+1)^{m+1}(1+e^{2tg^*(x)})}{e^{2tg^*(x)} + e^{-2t|g^*(y)-g^*(x)|}} \nonumber \\
&\leq D^m q(g^*(x))a(f(x))^{m^2+m+1} \|y-x\|^{\beta-m} B_m2^{3m+5}d^{1/2}(m+1)^{m+1} \Bigl(\frac{56}{41}\Bigr)^{2t}.
\end{align}
Using the multivariate Leibnitz rule again, together with~\eqref{Eq:qgbound},~\eqref{Eq:qcircg} and Lemma~\ref{Lemma:15over7}, for $t \leq 1$, $\|y-x\| \leq r_a(x)$ and $|\omega| = m \geq 1$,
\begin{align}
\label{Eq:ftg2}
\biggl|&\frac{\partial^\omega f_{t,g^*}(y)}{\partial x^\omega} - \frac{\partial^\omega f_{t,g^*}(x)}{\partial x^\omega}\biggr| \nonumber \\
&\leq 2c(t)\sum_{\nu:\nu \leq \omega} \binom{\omega}{\nu} \biggl\{\biggl|\frac{\partial^{\omega - \nu}f(y)}{\partial y^{\omega - \nu}}\biggr|\biggl|\frac{\partial^\nu q(g^*(y))}{\partial x^\nu} - \frac{\partial^\nu q(g^*(x))}{\partial x^\nu}\biggr| \nonumber \\
&\hspace{2cm}+ \biggl|\frac{\partial^\nu q(g^*(x))}{\partial x^\nu}\biggr|\biggl|\frac{\partial^\nu f(y)}{\partial x^\nu} - \frac{\partial^\nu f(x)}{\partial x^\nu}\biggr|\biggr\} \nonumber \\
&\leq 2^{4m+9} d^{1/2}B_m(m+1)^{m+1}D^m a(f(x))^{m^2+m+1}f_{t,g^*}(x)\|y-x\|^{\beta-m} \nonumber \\
&=: C_m'D^m a(f(x))^{m^2+m+1}f_{t,g^*}(x)\|y-x\|^{\beta-m}.
\end{align}
%Inductively again we may show that
%\begin{equation}
%\label{Eq:ftgsmooth}
%	| \partial^{\omega} (f_{t,g})(y) - \partial^{\omega} (f_{t,g})(x)| \leq C_m' D_g^m a(f(x))^{m^2+m+1} \|y-x\|^{\beta-m} f_{t,g}(x),
%\end{equation}
%for $|\omega|=m,|t|< \{412 \int g^2 f\}^{-1/2}$ and $ \|y-x\| \leq r_a(x)$, where we increase $C_m'$ if necessary. We may take
%\[
%	C_m' = 12 B_m \{(m+1)!\}^3 (m+1)^{m+1} 2^{m+1}d^{m+1/2} (15/7)^{(m+1)^2},
%\]
%where $B_m$ is the $m^{th}$ Bell number. 
This also holds in the case $m=0$.  Now note that if $12t< \{\int_{\mathcal{X}} (g^*)^2 f\}^{-1/2}$ we have
\[
%\label{Eq:ftgupperbound}
	f(x) = \frac{1+e^{-2tg^*(x)}}{2c(t)} f_{t,g^*}(x) \geq \frac{f_{t,g^*}(x)}{4}.
\]
Finally, define the function 
\begin{equation}
\label{Eq:atilde}
	\tilde{a}(\delta):=d^{m/2} C_m' D^m a(\delta/4)^{m^2+m+1}.
\end{equation}
Then $\tilde{a} \in \mathcal{A}$ and from~\eqref{Eq:ftgderivativesbounds} and~\eqref{Eq:ftg2}, we have $M_{f_{t,g^*},\tilde{a},\beta}(x) \leq \tilde{a}(f_{t,g^*}(x))$.  We conclude that for $t < \min \bigl(1, \{144\int g^2 f\}^{-1/2} \bigr)$, we have that $f_{t,g^*} \in \mathcal{F}_{d,\theta'}$, where $\theta'=(\alpha,\beta, 4\gamma, 4 \nu, \tilde{a}) \in \Theta$. The result follows on noting that $f_{t,g_\lambda}=f_{t \lambda, g^*}$.
\end{proof}

\end{document}